\documentclass[11pt]{amsart}
\usepackage{amsmath}
\usepackage{amssymb}
\usepackage{amsfonts}
\usepackage{mathrsfs}
\usepackage{graphicx}
\usepackage{epsfig}

\usepackage{euscript,graphicx,color}

\DeclareMathAlphabet{\mathpzc}{OT1}{pzc}{m}{it}

\theoremstyle{plain}

\newtheorem*{theo }{Theorem}
\newtheorem{prop}{Proposition}[section]

\newtheorem{coro}[prop]{Corollary}
\newtheorem{lemm}[prop]{Lemma}
\newtheorem{theo}[prop]{Theorem}

\newtheorem{nota}[prop]{Notation} 
\theoremstyle{definition}
\newtheorem{defi}[prop]{Definition}

\theoremstyle{remark}
\newtheorem{rema}[prop]{Remark}
\newtheorem{exem}[prop]{Example}

\newtheoremstyle{citing}
  {3pt}
  {3pt}
  {\itshape}
  {}
  {\bfseries}
  {.}
  {.5em}
  {\thmnote{#3}}

\theoremstyle{citing}
\newtheorem*{generic}{}

\newcommand{\partn}[1]{{\smallskip \noindent \textbf{#1.}}}

\numberwithin{equation}{section}       

\newcommand{\C}{\mathbb{C}}

\newcommand{\R}{\mathbb{R}}

\newcommand{\cA}{\mathcal{A}}

\newcommand{\cC}{\mathcal{C}}

\newcommand{\cL}{\mathcal{L}}

\newcommand{\fD}{\mathfrak{D}}

\newcommand{\fL}{\mathfrak{L}}

\newcommand{\sA}{\mathscr{A}}
\newcommand{\sB}{\mathscr{B}}

\newcommand{\sE}{\mathscr{E}}

\newcommand{\sK}{\mathscr{K}}

\newcommand{\sM}{\mathscr{M}}

\newcommand{\sW}{\mathscr{W}}

\newcommand{\hD}{\widehat{D}}

\newcommand{\hF}{\widehat{F}}

\newcommand{\hI}{\widehat{I}}

\newcommand{\hK}{\widehat{K}}

\newcommand{\hV}{\widehat{V}}
\newcommand{\hW}{\widehat{W}}

\newcommand{\hf}{\widehat{f}}

\newcommand{\hmu}{\widehat{\mu}}

\newcommand{\tF}{\widetilde{F}}

\newcommand{\tK}{\widetilde{K}}

\newcommand{\tX}{\widetilde{X}}

\newcommand{\tf}{\widetilde{f}}

\newcommand{\teta}{\widetilde{\teta}}

\newcommand{\tmu}{\widetilde{\mu}}
\newcommand{\tnu}{\widetilde{\nu}}

\newcommand{\tsigma}{\widetilde{\tsigma}}
\newcommand{\tvarsigma}{\widetilde{\tvarsigma}}

\newcommand{\e}{\varepsilon}
\newcommand{\la}{\lambda}

\renewcommand{\=}{ : = }

\DeclareMathOperator{\Const}{Const}

\DeclareMathOperator{\diam}{diam}
\DeclareMathOperator{\Jac}{Jac}

\DeclareMathOperator{\Lip}{Lip} 
\DeclareMathOperator{\dist}{dist}

\DeclareMathOperator{\interior}{interior}
\DeclareMathOperator{\HD}{HD}
\DeclareMathOperator{\BD}{BD}

\DeclareMathOperator{\Crit}{Crit}

\DeclareMathOperator{\PC} 

\DeclareMathOperator{\CE2}{CE2}
\DeclareMathOperator{\Lyap}{Lyap}
\DeclareMathOperator{\varhyp}{varhyp}
\DeclareMathOperator{\var}{var}

\DeclareMathOperator{\Indiff}{Indiff}

\DeclareMathOperator{\reg}{r}
\DeclareMathOperator{\sing}{s}

\newcommand{\HDhyp}{\HD_{\operatorname{hyp}}}

\DeclareMathOperator{\ttop}{top}
\DeclareMathOperator{\conf}{conf}
\DeclareMathOperator{\Bconf}{Bconf}
\DeclareMathOperator{\Per}{Per}
\DeclareMathOperator{\Comp}{Comp}
\DeclareMathOperator{\hyp}{hyp}
\DeclareMathOperator{\Exp}{Exp}
\DeclareMathOperator{\con}{con}

\DeclareMathOperator{\NO}{NO}

\DeclareMathOperator{\tree}{tree}

\newcommand{\tpos}{t_+} 
\newcommand{\tneg}{t_-} 
\newcommand{\chiinf}{\chi_{\inf}}
\newcommand{\chisup}{\chi_{\sup}} 

\newcommand{\pressure}{\mathscr{P}}

\newcommand{\badp}{Y}
\newcommand{\uW}{\underline{W}}

\begin{document}


\

\

\title[Geometric pressure for multimodal maps of the interval]{Geometric pressure for multimodal maps of the interval}
\author[F. Przytycki]{Feliks Przytycki$^\dag$}
\author[J. Rivera-Letelier]{Juan Rivera-Letelier$^\ddag$}

\date{\today}

\thanks{$\dag$ Partially supported by Polish MNiSW Grant N N201 607640.}
\address{$\dag$ Feliks Przytycki, Institute of Mathematics, Polish Academy of Sciences, ul. \'Sniadeckich 8, 00956 Warszawa, Poland.}
\email{feliksp@impan.gov.pl}
\thanks{$\ddag$ Partially supported by Chilean FONDECYT grants 1100922 and 1141091.}
\address{$\ddag$ Juan Rivera-Letelier, University of Rochester, UR Mathematics, 811 Hylan Building, University of Rochester, RC Box 270138, Rochester, NY 14627, U.S.A.}
\email{riveraletelier@gmail.com}



\begin{abstract} This paper is an interval dynamics counterpart of three theories founded earlier by the authors, S. Smirnov and others in the setting of the iteration of rational maps on the Riemann sphere: the equivalence of several notions of non-uniform hyperbolicity, Geometric Pressure, and Nice Inducing Schemes methods leading to results in thermodynamical formalism. We work in a setting of generalized multimodal maps, that is smooth maps $f$ of a finite union of compact intervals $\hI$ in $\R$ into $\R$ with non-flat critical points,
such that on its maximal forward invariant set $K$ the map $f$ is topologically transitive and has positive topological entropy. We prove that several notions of non-uniform hyperbolicity of $f|_K$ are equivalent (including uniform hyperbolicity on periodic orbits, TCE \& all periodic orbits in $K$ hyperbolic repelling, Lyapunov hyperbolicity, and exponential shrinking of pull-backs). We prove that several definitions of geometric pressure $P(t)$, that is pressure
for the map $f|_K$ and the potential $-t\log |f'|$, give the same value (including pressure on periodic orbits, "tree" pressure, variational pressures and conformal pressure).
Finally we prove that, provided all periodic orbits in $K$ are hyperbolic repelling, the function $P(t)$ is real analytic for $t$ between the "condensation" and "freezing" parameters and that for each such $t$ there exists unique equilibrium (and conformal) measure satisfying strong statistical properties.
\end{abstract}

\maketitle
\tableofcontents

\

\section{Introduction. The main results}

This paper is devoted to extending to interval maps results for iteration of rational functions on the Riemann sphere concerning statistical and `thermodynamical' properties obtained mainly in  \cite{PR-LS2}, \cite{PR-L1} and \cite{PR-L2}. We work with a class of
generalized multimodal maps, that is smooth maps $f$ of  a finite union of compact intervals $\hI$ in $\R$ into $\R$ with non-flat critical points, and investigate statistical and `thermodynamical' properties of $f$ restricted to the compact set $K\subset \hI$, maximal forward invariant in $\hI$, such that $f|_K$ is topologically transitive and has positive topological entropy.

This includes all sufficiently regular multimodal maps, and other maps that are naturally defined on a union of intervals, like fixed points of generalized renormalization operators, see e.g. \cite{LS} and references therein, Examples \ref{Julia}, \ref{basic set}, \ref{notwi}, and the ones mentioned in Subsection \ref{complementary}.

Several strategies are similar to the complex case, but they often need to be adapted to the interval case in a non-trivial way.

\smallskip

The paper concerns three topics, closely related to each other.

\medskip

1. Extending the results for unimodal maps in \cite{NS}, and multimodal maps in \cite{NP} and \cite{R-L} to generalized multimodal maps considered here, we prove the equivalence of several notions of non-uniform hyperbolicity,  including uniform hyperbolicity on periodic orbits, Topological Collet-Eckmann condition (abbr. TCE) \& all periodic orbits hyperbolic repelling, Lyapunov hyperbolicity, and exponential shrinking of  pull-backs. For the complex setting see \cite{PR-LS1} and the references therein.

\medskip

2. We prove that several definitions of geometric pressure $P(t)$, that is pressure
for the map $f|_K$ and the potential $-t\log |f'|$, give the same value (including pressure on periodic orbits, "tree" pressure and variational pressures). For the complex setting, see \cite{PR-LS2} and the references therein.

\medskip

3. We prove that, provided all periodic orbits in $K$ are hyperbolic repelling, the geometric pressure function $t\mapsto P(t)$ is real analytic for $t$ (inverse of temperature) between the "condensation" and "freezing" parameters, $t_- $ and $t_+ $, and for each such $t$ there exists unique equilibrium (and conformal) measure satisfying strong statistical properties. All this is contained in Theorem A, the main result of the paper. For the complex setting, see \cite{PR-L1} and \cite{PR-L2}.

\medskip

 Our results extend \cite{BT1}, which proved existence and uniqueness of equilibria and the analyticity of pressure only on  a small interval of parameters $t$
 and
assumed additionally a  growth of absolute values of the derivatives of the iterates at critical values.

Our  results extend also \cite[Theorem 8.2]{PS}, where the real analyticity of the geometric
pressure function on a neighborhood of $[0, 1]$ was proved, also
assumed some growth of absolute values of the
derivatives of the iterates at critical values and additionally assumed a slow
recurrence condition.

Our main results are stronger
in that growth assumptions are not used and the domain of real analyticity of $t$ is the whole $(t_-,t_+)$, i.e. the maximal possible domain. Precisely in this domain it holds $P(t)>\sup \int\phi\, d\mu$, where supremum is taken over all $f$-invariant ergodic probability measures on $K$, and $\phi=-t\log|f'|$ (it is therefore clear that analyticity
cannot hold at neither  $t = t_-$ nor $t = t_+$, as $P(t)$ is affine to the left of $t_-$ and to the right of $t_+$).

 Let us mention also the paper \cite{IT} where, under the restriction that $f$ has no preperiodic critical points, the existence of equilibria was proved  for all $-\infty=t_-<t<t_+$.
The authors proved that $P(t)$ is of class $C^1$ and that their method does not allow them to obtain
statistical properties of the equilibria.

 Our results are related to papers on thermodynamical formalism for $\phi$ being H\"older continuous, satisfying the assumption
$P(f|_K,\phi)>\sup \phi$ or related ones, see \cite{BT2} and \cite{LR-L} and references therein.

 In this paper, topic 3 above, as well as in \cite{PR-L1}, \cite{PR-L2}, we use some `inducing schemes', that is dynamics of return maps to `nice' domains.

\

 \subsection{Generalized multimodal maps, maximal invariant sets and related notions}

\

We shall assume throughout the paper that intervals are non-degenerate (i.e. containing more than one point).
For an integer $r \ge 1$ and a finite union of intervals $U$, a function $f : U \to \R$ is of \emph{class}
$C^r$, if it is the restriction of a function of class $C^r$ defined on an open neighborhood of $U$.

\begin{defi}\label{def:critical}
Let $U$ be a finite union of intervals, and $f : U \to \R$ a map of
class $C^2$. A \emph{critical point} of $f$ is a point $c$ such that $f'(c) = 0$. An
isolated critical point $c$ of $f$ is \emph{inflection} (resp. \emph{turning})
if $f$ is locally (resp. is not locally) injective at $c$. Furthermore, $c$ is \emph{non-flat} if
$f(x)=\pm|\phi(x)|^\ell +f(c)$ for some real $\ell\ge 2$ and $\phi$ a $C^2$ diffeomorphism in a neighbourhood of $c$ with $\phi(c)=0$, see \cite[Chapter IV]{dMvS}.
The set of critical points will be denoted by $\Crit(f)$ turning critical points by $\Crit^T(f)$ and inflection critical points by $\Crit^I(f)$.
\end{defi}

\begin{defi}\label{multimodal}
Let $\hI$ be the union of a finite family of pairwise disjoint compact  intervals
$I^j, j=1,...,m$
in the real line $\R$.
Moreover, consider a map $f : \hI \to \R$ of class $C^2$ with only non-flat critical points.
Let $K=K(f)$ be the maximal forward invariant set for $f$, more precisely the set of all points in $\hI$ whose forward trajectories stay in $\hI$. We call $K$ the \emph{maximal repeller} of $f$.
The map $f$ is said to be a \emph{generalized multimodal map} if $K$ is infinite and if $f|_K$ is topologically transitive (that is
for all $U,V\subset K$ open in it, non-empty, there exists $n>0$ such that $f^n(U)\cap V\not=\emptyset$).

\smallskip

A generalized multimodal map $f : \hI \to \R$ is said to be \emph{reduced} if $\partial \hI \subset K$ and if \begin{equation}\label{no Crit}
(\hI\setminus K) \cap \Crit(f)=\emptyset .
\end{equation}

\end{defi}

\bigskip

We will show that
$f(K)=K$ and
$K$ is either a finite union of compact intervals or a Cantor set, see Lemma~\ref{infinite}.

Note that we do not assume neither $f(\hI)\subset \hI$ nor $f(\partial \hI)\subset \partial \hI$.

We shall usually assume that $K$ is dynamically sufficiently rich, namely that the topological entropy $h_{\rm{top}}(f|_K)$
is positive.

In examples of generalized multimodal maps,  
topological transitivity and positive entropy of $f|_K$ must be verified, which is sometimes not easy.

\smallskip

Every generalized multimodal map $f : \hI \to \R$ determines in a canonical way a subset $\hI_K$ of
$\hI$ such that the restriction $f : \hI_K \to \R$ is a reduced generalized multimodal map with
$K(f|_{\hI_K}) = K(f)$ as follows:

Remove from $\hI\setminus K$ a finite number of open intervals $V_s, s=1,...,k$, each containing a critical point of $f$. Next, using for $\hI\setminus \bigcup_s V_s$ the same notation $\hI=\bigcup_{j=1,...,m} I^j$,
 consider for each $j=1,...,m$ the set $\hI^j$ being the convex hull of $I^j\cap K$. Finally define   $\hI_K:=\bigcup_{j=1,...,m}\hI^j$.
Then $f:\hI_K\to \R$ is a reduced generalized multimodal map.
In fact, the maximality of $K$ in $\hI_K$ obviously follows, since $\hI_K\subset \hI$. Since $K$ has no
isolated points,
it also follows
that $\hI^j $ is
non-degenerate.

\smallskip

Thus we shall consider only reduced generalized multimodal maps.
We shall usually skip the word `reduced'.

 \begin{rema}\label{starting from K}
 All assertions of our theorems concern the action of $f$ on $K$. So it is natural to organize definitions in the opposite order as follows.
\smallskip

Let $K\subset\R$ be an infinite compact subset of the real line.
Let $f:K\to K$ be a continuous topologically transitive mapping.
Assume there exists a covering of $K$ by a finite family of
pairwise disjoint closed intervals $\hI^j$ with end points in $K$ such that $f$ extends to a generalized multimodal map on their union $\hI_K$ and $K$ is the maximal forward invariant set $K(f)$ in it.
Reducing $\hI_K$ if necessary, assume all the critical
points of $f$ in $\hI_K$ are in $K$. Then $f : \hI_K \to \R$ is a
reduced generalized multimodal map.
\end{rema}

\smallskip

\begin{defi}\label{extension-1}
Let $f : \hI \to \R$ be a generalized multimodal map.
Consider a $C^2$ extension of $f$ to an open neighbourhood ${\bf{U}}^j$ of each $\hI^j$.
We consider ${\bf{U}}^j$'s small enough so that they are pairwise disjoint. Moreover we assume that all critical points in ${\bf{U}}^j$ are in $\hI^j$. Thus together with (\ref{no Crit}) we assume that the union
$\bf{U}$ of ${\bf{U}}^1,\ldots, {\bf{U}}^m$ satisfies
\begin{equation}\label{no Crit2}
({\bf{U}}\setminus K) \cap \Crit(f)=\emptyset.
\end{equation}
We call the quadruple $(f,K, \hI_K, {\bf{U}})$
a \emph{(reduced) generalized multimodal quadruple}.
In fact it is always sufficient to start with a triple $(f,K,{\bf{U}})$, because this already uniquely defines
$\hI_K$.

As we do not assume $f({\bf{U}})\subset {\bf{U}}$, when we iterate $f$, i.e. consider $f^k$ for positive integers $k$,  we consider them on their domains, which can be smaller than ${\bf{U}}$ for $k\ge 2$,

Note that we do not assume $K$ to be maximal forward invariant in ${\bf{U}}$. We assume maximality only in $\hI_K$.

See Proposition~\ref{reduce} for an alternative approach, replacing maximality of $K$ in $\hI_K$ by so-called Darboux property.
\end{defi}

\begin{nota}\label{sets-A}
1. Let us stress that the properties of the extension of $f$ beyond $K$ are used only to specify assumptions imposed on the action of $f$ on $K$, in particular the way it is embedded in $\R$. So we will sometimes talk about
 generalized multimodal pairs $(f,K)$, understanding that this notion involves
 $\hI_K$ and ${\bf{U}}$ as above, sometimes about the triples $(f,K,\hI_K)$, sometimes about the triples $(f,K,{\bf{U}})$.


2. Denote the space of all (reduced) generalized multimodal quadruples (pairs, triples) discussed above by $\sA$. For $f$ and $\phi$ (change of coordinates as in the definition of non-flatness) of
class $C^r$ we write $\sA^r$, so $\sA=\sA^2$. If $h_{\rm{top}}(f|_K)>0$ is assumed we write $\sA_+$ or $\sA_+^r$.
\end{nota}

\

\subsection{Periodic orbits and basins of attraction. Bounded distortion property}

\begin{defi}\label{periodic}
As usual we call a point $p\in {\bf{U}}$ \emph{periodic} if there exists $m\ge 1$ such that $f^m(p)=p$.
We denote by $O(p)$ its periodic orbit.

Define the \emph{basin of attraction} of the periodic orbit $O(p)\subset {\bf{U}}$ by
$$
B(O(p)):=\interior \{x\in {\bf{U}}: f^n(x) \to O(p), \; \text{as}\; n\to\infty\}.
$$

The orbit $O(p)$ is called \emph{attracting} if $O(p)\subset B(O(p))$.
Notice that this happens if
$|(f^m)'(p)|<1$, when we call the orbit \emph{hyperbolic attracting}, and it can happen also if $|(f^m)'(p)|=1$.
The orbit is called \emph{repelling} if
for $g:=f|_W^{-m}$, where $W$ is a small neighbourhood of $O(p)$ in $\R$, we have
$g(W)\subset W$ and $g^n(W)\to O(p)$ as $n\to\infty$.
If $|(f^m)'(p)|> 1$ then $O(p)$ is repelling and we call the orbit \emph{hyperbolic repelling}.

When $O(p)$ is neither attracting nor repelling,
we call $O(p)$ \emph{indifferent}.
The union of the set of
indifferent periodic orbits will be denoted by $\Indiff(f)$.
If $O(p)$ is indifferent,
then it is said to be
\emph{one-sided attracting} if its basin of attraction contains a one-sided neighbourhood of each point of the orbit. Finally $O(p)$ is said to be \emph{one-sided repelling} if
 it is not repelling and if for a local inverse, the orbit of every point in a one-sided neighborhood of  $O(p)$  converges to  $O(p)$.

We say also that a periodic point $p$ is \emph{(hyperbolic) attracting, (hyperbolic) repelling, one-sided attracting, one-sided repelling} or \emph{indifferent} if $O(p)$ is (hyperbolic)
attracting, (hyperbolic) repelling, one-sided attracting, one-sided repelling, or indifferent, respectively.

When we discuss a specific (left or right) one-side neighbourhood of a point in $O(p)$  in the above definitions,  we sometimes  say  \emph{ attracting on one side} or \emph{repelling on one side}.

\smallskip

 If an indifferent periodic point $p$ of period $m$ is neither one-sided attracting nor one-sided repelling on the same side, then obviously it must be an accumulation point of periodic points of period $m$ from that side.
 Notice that indifferent one-sided attracting (repelling) implies  $f^m$ preserves the orientation at $p$, equivalently: $(f^m)'(p)=1$; otherwise, if $(f^m)'(p)=-1$, the point $p$ would be attracting or repelling.

\end{defi}




 \begin{rema}\label{finite indifferent}
For $(f,K,\hI_K,{\bf{U}})\in \sA$
 it follows by the maximality of $K$ that there are no periodic orbits in $\hI_K\setminus K$.
  By the topological transitivity of $f$ on $K$,  there are no attracting periodic orbits in $K$.
Moreover each point in an indifferent periodic orbit in $K$ is one-sided repelling on the side from which it is accumulated by $K$; then it is not accumulated by $K$ from the other side.
Also by the topological transitivity and by smoothness of $f$ the number of
periodic orbits in $K$ that are not hyperbolic repelling is finite.
The proof uses \cite[Ch. IV, Theorem B]{dMvS}.
Therefore, the number of periodic orbits in $\hI$ that are not hyperbolic repelling is finite.

For further details see
Corollary~\ref{no periodic} in Appendix A and remarks following it.

By changing $f$ on ${\bf{U}}\setminus \hI_K$ if necessary, one can assume that the only periodic orbits in ${\bf{U}}\setminus K$ are hyperbolic repelling. See Appendix A, Lemma~\ref{good extension}, for  details.

\end{rema}

\begin{defi}\label{basins}
For an attracting or a one-sided attracting periodic point $p$,
the \emph{immediate basin} $B_0(O(p))$ is the union of the components of $B(O(p))$ whose closures intersect $O(p)$. The unique component of $B_0(O(p))$ whose closure contains $p$ will be denoted by $B_0(p)$.
\end{defi}

Notice that if $O(p)$ is attracting then $O(p)\subset B_0(O(p))$.
If $O(p)$ is one-sided attracting then $p$ is a boundary point of $B_0(p)$ and \emph{vice versa}.
Finally notice that for each component $T$ of
$B(O(p))$ there exists $n\ge 0$ such that $f^n(T)\subset B_0(f)\cup O(p)$. (See an argument in Proof of Proposition~\ref{reduce} in Appendix A). We need to add $O(p)$ above in case $O(p)$ is indifferent and some $f^j(T)$ contains a turning critical point whose forward trajectory hits $p$.
(Compare also Example~\ref{Julia} below, where Julia set need not be backward invariant.)

\begin{exem}\label{Julia}
 If for a generalized multimodal map $f$ its domain $\hI$ is just one interval
 $I$ and $f(I)\subset I$,
then we have a classical case of an interval multimodal map. However the set of non-escaping points $K(f)$ is the whole $I$ in this case, usually too big (not satisfying topological transitivity, and not even mapped by $f$ onto itself). So
one considers smaller $f$-invariant sets, see below.

Write
$$
B(f)=\bigcup_p B(O(p)),
\; B_0(f)=\bigcup_p B_0(O(p))\ \hbox{and}\;
I^+:=\bigcap_{n=0}^\infty f^n(I).
$$
Define \emph{Julia set} by
$$
J(f):=I\setminus B(f),
$$
compare \cite[Ch.4]{dMvS}, and its {\it core Julia set} by
$$
J^+(f):=(I\setminus B(f))\cap I^+ .
$$

Notice that the sets $J(f)$ and $J^+(f)$ are compact and forward invariant. They need not be backward invariant, even $J^+(f)$
for $f|_{I^+}$, namely a critical preimage of a indifferent point $p\in\Indiff(f)$ can be in $B(O(p))$), whereas $p\notin B(O(p))$.

The definition of $J(f)$ is compatible with the definition as a complement of the domain of normality of all the forward iterates of $f$ as in the complex case.

Notice however that without assuming that $f$ is topologically transitive on $J^+(f)$ the comparison to Julia set is not justified. For example for $f$ mapping $I=[0,1]$
into itself defined by $f(x)=f_\lambda(x)=\lambda x(1-x)$ where $3<\lambda<4$ (to exclude an attracting or indifferent fixed point or an escape from $[0,1]$),
 we have $I^+=I$ where $f$ is not topologically transitive.
 However, if $f$ is not renormalizable, if we restrict $f$ to $I_\e=[\e,1]$ then $f$ is
topologically transitive on $I_\e^+$ for $I_\e^+ = [c_2,c_1]$ (independent of $\e$), where  $c_2=f^2(c), c_1=f(c)$, and $c:=1/2$ is the critical point. Notice that by $\lambda>3$ and the non-renormalizability there can be no basin of attraction. Our $\hI$ and $K$ are both equal to $[c_2,c_1]$


Notice that since $f(c_2)$ belongs to the interior of $I^+$ the set $K$ is not maximal invariant in $U$ being
a neighbourhood of $I^+$ in $\R$ whatever small $U$ we take.
\end{exem}

\begin{exem}\label{basic set}
Multimodal maps $f$  considered in the previous example, restricted to $J^+(f)$ still need not be
topologically transitive. Then, instead,
examples of generalized multimodal pairs in $\sA_+$ are provided by $(f,K)$, where $K$ is
an arbitrary maximal topologically transitive set in $\Omega_n$'s in the spectral decomposition of the set of non-wandering points for $f$, as  in \cite[Theorem III.4.2, item 4.]{dMvS}, so-called
{\it basic set}, for which $h_{\ttop}(f|_K)>0$. It is easy to verify that basic sets are weakly isolated.
\end{exem}

\begin{rema}
In regularity lower than $C^2$ there can exist a wandering interval, namely an open
interval $T$ such that
all intervals $T,f(T), f^2(T),...$ are pairwise disjoint and not in $B(f)$. In the $C^2$ multimodal case wandering intervals cannot exist, see \cite[Ch.IV, theorem A]{dMvS}. We shall use this fact many times in this paper.
\end{rema}

\begin{defi}\label{distortion}
Following \cite{dMvS}
we say that for $\varepsilon>0$ and an interval $I\subset \R$, an interval $I'\supset I$ is
an $\e$-scaled
neighbourhood of $I$ if $I'\setminus I$ has two components, call them left and right, $L$ and $R$, such that $|L|/|I|, |R|/|I| = \e$.


We say that a $C^2$ (or $C^1$) map $f:U\to \R$ for $U$ an open subset of $\R$ 
 (in particular a generalized multimodal triple $(f,K,{\bf{U}})\in \sA$ for $U={\bf{U}}$) satisfies
{\it bounded distortion}, abbr. BD, condition
if
there exists $\delta>0$ such that for every $\e>0$ there exists
$C=C(\e )>0$ such that the following holds:
For every  pair of intervals $I_1 \subset U$, $I_2\subset \R$ such that
$|I_2|\le \delta$  and for every $n>0$,
if $f^n$
maps diffeomorphically an interval $I'_1$ containing $I_1$ onto an interval $I_2'$ being an $\e $-scaled neighbourhood of $I_2$ and $f^n(I_1)=I_2$, then for every $x,y\in I_2$ we have
for $g=(f^n|_{I_1})^{-1}$
$$
|g'(x)/g'(y)|\le C(\e ).
$$
(Equivalently, for every $x,y\in I_1$, \; $|(f^n)'(x)/(f^n)'(y)|\le C(\e )$.)

 Notice that BD easily implies that for every $\e >0$ there is $\e '>0$ such that if in the above notation $I'_2$ is an $\e $-scaled neighbourhood of $I_2$ then
$I'_1$ contains an $\e '$-scaled neighbourhood of $I_1$.

For related definitions of distortion to be used in the paper and further discussion see
Section~\ref{preliminaries}:
Definition~\ref{LBD} and Remark~\ref{LBD2}.

We denote the space of $(f,K,{\bf{U}})\in\sA$ or $(f,K, {\bf{U}})\in\sA_+$ satisfying BD, by
$\sA^{{\BD}}$ or $\sA_+^{{\BD}}$, respectively. Sometimes we omit $\bf U$ and use notation $\sA^{{\BD}}$ or $\sA_+^{{\BD}}$ for generalized multimodal pairs.

\end{defi}

\

\begin{rema}\label{more periodic}
Notice that for $(f,K, {\bf{U}})\in \cA^{\BD}$ all repelling periodic orbits in ${\bf{U}}$ are hyperbolic repelling.

Furthermore, since every indifferent periodic orbit in K is
one-sided repelling by Remark~\ref{finite indifferent}, it follows from BD that every
indifferent periodic orbit in K is also one-sided attracting.

The finiteness of the set of indifferent periodic orbits in $K$ was in fact noted already in Remark~\ref{finite indifferent} without assuming BD.

 The finiteness of the set of non-repelling periodic orbits in $\bf{U}$ (where we treat an interval of periodic points as one point) follows from the standard fact that by BD for every (one-sided) attracting  $O(p)$,\, the immediate basin $B_0(O(p))$ must contain a critical point or the boundary of the basin must contain a point belonging to $\partial\bf{U}$ and that there is only a finite number of such points (in fact critical points cannot be in $B_0(O(p))$ since we have assumed $\Crit(f)\subset K$). This implies also that the only indifferent periodic points that are points of accumulation of periodic points of the same or doubled period are those belonging to intervals of periodic points.
See Proposition~\ref{finiteness} in Appendix A for more details.
In particular, by shrinking $\bf{U}$, one can assume
that every periodic orbit
in ${\bf{U}} \setminus K$ is hyperbolic repelling.
\end{rema}

\

\subsection
{Statement of Theorem A: Analytic dependence of geometric
pressure on temperature, equilibria}

\

Fix  $(f,K)\in \sA$.  Let $\sM(f,K)$ be the space of all probability measures supported on~$K$ that are invariant by~$f$.
For each $\mu \in \mathscr{M}(f,K)$, denote by~$h_\mu(f)$ the \textit{measure theoretic entropy of}~$\mu$, and by $\chi_\mu(f) \= \int \log |f'| d \mu$ the \textit{Lyapunov exponent of}~$\mu$.

If $\mu$ is supported on a periodic orbit $O(p)$ we use the notation $\chi(p)$.

Given a real number~$t$ we define the \textit{pressure of $f|_K$ for the potential $- t \log |f'|$} by,
\begin{equation}\label{e:variational pressure}
P(t) \=
\sup \left\{ h_\mu(f) - t \chi_\mu(f): \mu \in \mathscr{M}(f,K) \right\}.
\end{equation}
For each $t \in \R$ we have $P(t) < + \infty$ since $\chi_\mu(f)\ge 0$ for each $\mu \in \mathscr{M}(f,K)$, see \cite{P-Lyap} or \cite[Appendix A]{R-L} for a simpler proof.
Sometimes we call $P(t)$ \textit{variational pressure} and denote it by $P_{\var}(t)$.

A measure $\mu$ is called an \textit{equilibrium state of~$f$ for the potential $-t\log|f'|$},
if the supremum in \eqref{e:variational pressure} is attained for this measure.

As in \cite{PR-L2} define the numbers,
$$
\chiinf(f) \= \inf \left\{ \chi_\mu(f): \mu \in \mathscr{M}(f,K) \right\},
$$
$$
\chisup(f) \= \sup \left\{ \chi_\mu(f): \mu \in \mathscr{M}(f,K) \right\}.
$$
Later on we write sometimes $\chiinf(f,K)$ and $\chisup(f,K)$ or just $\chiinf, \chisup$.
By Ergodic Decomposition Theorem we can assume that all $\mu$ in the definition of $\chiinf(f)$ and
$\chisup(f)$ above are ergodic.
In Section~\ref{s:Key} we use an equivalent definition, see Proposition~\ref{asymptotes} in Section~\ref{pressure}.

\smallskip

Define
$$
\tneg \= \inf \{ t \in \R: P(t) + t \chisup(f) > 0 \},
$$
$$
\tpos \= \sup \{ t \in \R: P(t) + t \chiinf(f) > 0 \}.
$$
the \textit{condensation point} and the \textit{freezing point} of~$f$, respectively.
 As in the complex case the condensation (resp. freezing) point can take the value $-\infty$ (resp. $+ \infty$).

Similarly to \cite{PR-L2} we have the following properties:
\begin{itemize}
\item
  $\tneg < 0 < \tpos$;
\item
for all $t \in \R \setminus (\tneg, \tpos)$ we have $P(t) = \max \{ - t \chisup(f), - t \chiinf(f) \} $;
\item
for all $t \in (\tneg, \tpos)$ we have $P(t) > \max \{ - t \chiinf(f), - t \chisup(f) \}$.
\end{itemize}

\begin{defi}\label{defi:Jacobian}
Let $f : U \to \R$ be a map of class $C^2$ with only non-flat critical points, and let $K$
be a compact subset of $U$ that is forward invariant by $f$. Moreover, let
$\phi : K \to \R$ be a Borel function.
We call a finite Borel measure $\mu$ on $K$ for $(f,K)\in\sA$
a $\phi$-\emph{conformal measure} if it is forward quasi-invariant, i.e.
$\mu\circ f\prec \mu$, compare \cite[Section 5.2]{PU}, and for every Borel set $A\subset K$ on which $f$ is
injective
\begin{equation}\label{Jacobian}
\mu (f(A))=\int_A \phi\; d\mu.
\end{equation}

\end{defi}

\begin{defi}\label{K-diffeomorphism}
 Let $(f, K)$ be in $\sA$.
 We say that for intervals $W,B$ in $\R$ intersecting $K$, the map $f^n:W\to B$ is a {\it $K$-diffeomorphism} if it is a diffeomorphism, in particular it is well defined, i.e.
 for each $j=0,...,n-1$ \; $f^j(W)\subset {\bf{U}}$,
 and $f^n(W\cap K)=f^n(W)\cap K$, compare Definition ~\ref{Darboux} and Lemma~\ref{NO}.

 Since $f(K)\subset K$, \ $f^n:W\to B$ is $K$-diffeomorphism if and only if  $f:f^j(W)\to f^{j+1}(W)$ is
 $K$-diffeomorphism for all $j=0,...,n-1$.
 \end{defi}

\begin{defi}\label{conical limit} Let $(f, K)$ be in $\sA$. Denote by~$K_{\con}(f)$ the ``conical limit part of $K$'' for $f$ and $K$,  defined
 as the set of all those points~$x\in K$ for which there exists~$\rho(x) > 0$ and an arbitrarily large positive integer~$n$, such that $f^n$ on $W\ni x$, the component of the $f^{-n}$-preimage of the interval $B:=B(f^n(x), \rho(x))$ containing $x$, is a $K$-diffeomorphism onto $B$.

 Notice that $f(K_{\con}(f))\subset K_{\con}(f)$.
\end{defi}

The main theorem in this paper corresponding to \cite[Main Theorem]{PR-L2} is

\begin{generic}[Theorem A]
Let
$(f,K)\in\sA_+^3$, in particular let $f$ be topologically transitive on $K$ and have positive entropy,
and assume all $f$-periodic orbits in $K$ are hyperbolic repelling.
Then
\begin{description}
\item[1. Analyticity of the pressure function]
The pressure function $P(t)$ is real analytic on the open interval $(\tneg, \tpos)$, and linear with slope $- \chisup(f)$ (resp. $-\chiinf(f)$) on $(- \infty, \tneg]$ (resp. $[\tpos, + \infty)$).

\item[2. Conformal measure]
For each $t \in (\tneg, \tpos)$, the least value $p$
for which there exists an $(\exp p)|f'|^t$\nobreakdash-conformal probability measure $\mu_t$ on $K$ is $p=P(t)$. The measure $\mu_t$ is unique among all $(\exp p)|f'|^t$\nobreakdash-confo\-rmal probability measures.
Moreover $\mu_t$ is non\nobreakdash-atomic, positive on all open sets in $K$, ergodic, and it is supported on~$K_{\con}(f)$.

\item[3. Equilibrium states]
For each $t \in (\tneg, \tpos)$, for the potential~$\phi=-t \log|f'|$, there is a unique equilibrium measure of~$f$. It is ergodic and absolutely continuous with respect to $\mu_t$ with the density bounded from below by a positive constant almost everywhere.
If furthermore $f$ is topologically exact on $K$, then this measure is mixing, has exponential decay of correlations and it satisfies the Central Limit Theorem for Lipschitz gauge functions.

\end{description}
\end{generic}

\

It is easy to see that the assertion about the analyticity of $P(t)$ can be false without the topological transitivity assumption. See also \cite{Dobbs} for an example where analyticity fails
at an infinite set of values of $t$.

For a generalized multimodal map $(f,K)$ having only hyperbolic repelling periodic points in $K$, as in Theorem A,   the assumption $(f,K)\in \sA_+^3$ can be replaced by a bounded distortion hypothesis that is more
restrictive than BD, see Definition ~\ref{LBD}, which is in fact what we use in our proof. See
Remark~\ref{LBD2} and Lemma~\ref{good extension}.

Let us now comment on the properties of topological transitivity and positive entropy assumed in Theorem A, and also notions of exactness.

\begin{defi}\label{wexact}
Let $h:X\to X$ be a continuous mapping of a compact metric space $X$.

 We call $h$ {\it topologically exact} if  for every non-empty open $V\subset X$ there is $m>0$ such that $h^m(V)=K$ (sometimes the name \emph{locally eventually onto} is used).

We call  $h$
{\it weakly topologically exact} (or just \emph{weakly exact}) if there exists $N>0$ such that for every non-empty open $V\subset X$, \,
there exists $n(V)\ge 0$ such that
\begin{equation}\label{eq1}
\bigcup_{j=0}^N h^{n(V)+j}(V)=X.
\end{equation}

This property clearly implies that $h$ maps $X$ onto $X$, compare Lemma~\ref{infinite}.

Notice that the equality~(\ref{eq1})
implies
automatically the similar one with $n(V)$ replaced by any $n\ge n(V)$. To see this apply
$h^{n-n(V)}$ to both sides of the equation and use $h(X)=X$.

In the sequel we shall usually use these definitions for $(f,K)\in \sA$ setting $X=K$, $h=f|_K$.
For some immediate technical consequences of the property of weak exactness see  Remark~\ref{wexact2} below.

\end{defi}

\begin{rema}\label{wexact1}
Clearly topological exactness implies weak topologically exactness which in turn implies topological transitivity. In Appendix A,

\noindent Lemma~\ref{trans implies exact},  we provide a proof of the  converse fact for $f|_K$, saying that topological transitivity and positive topological entropy of $f|_K$ imply weak topological exactness.
\underbar{This allows in our theorems to assume only \emph{topological}}

\noindent\underbar{\emph{transitivity} and to use in proofs the formally stronger \emph{weak exactness}}.

This fact seems to be folklore. Most of the proof in Appendix A was told to us by Micha{\l} Misiurewicz. We are also grateful to Peter Raith for explaining us how this fact follows from \cite{Hof}. See also \cite[Appendix B]{Buzzi}.

In fact for $(f,K)\in \sA$ the properties: weak topological exactness and positive topological entropy are equivalent for $f|_K:K\to K$ , see Proposition \ref{exact implies htop}.
\end{rema}

Now we provide the notion of {\it exceptional} and related ones, substantial in the paper, though
not explicitly present in the statements of the main theorems. They are important in the study of various variants of conformal measures and pressures, see Appendix C, and in explaining the meaning if $t_-$, and substantial in the study of e.g. Lyapunov spectrum, see \cite{MS}, \cite{GPRR} and \cite{GPR2}.

\begin{defi}\label{exceptional}
1. \underline{End points.} Let $(f,K,\hI_K)\in\sA$. We say that $x\in K$ is an {\it end point} if $x\in \partial \hI_K$.
We shall use also  the notion of the
{\it singular set} of $f$ in $K$, defined by  $S(f,K):=\Crit(f)\cup \partial(\hI_K)$.

 Of more importance will be the notion of the {\it restricted singular set}
 $$
 S'(f,K):=\Crit(f) \cup {\rm{NO}}(f,K),$$
 where ${\rm{NO}}(f,K)$ is the set of points where $f|_K$ is not an open map, i.e. points $x\in K$ such that
there is an arbitrarily small neighbourhood $V$ of $x$ in $K$ whose $f$-image is not open in $K$.
Notice that  ${\rm{NO}}(f,K)\subset \Crit^T(f)\cup \partial(\hI_K)$, see Lemma~\ref{NO}.

In the proof of Theorem A we shall use inducing, that is return maps to nice domains, as commented already at the very beginning, as in \cite{PR-L2}, but we cover by the components (open intervals) of the nice sets the set $S'(f,K)$ rather than only $\Crit(f)$.

\medskip

2. \underline{Exceptional sets.} (Compare \cite{MS} and \cite{GPRR}.) We say that a nonempty  forward invariant set
$E\subset K$  is $S$-{\it exceptional} for $f$, if it is not dense in $K$ and
$$
(f|_K)^{-1}(E) \setminus S(f,K) \subset E.
$$
Analogously, replacing $S(f,K)$ by $S'(f,K)$, we define $S'$-{\it exceptional} subsets of $K$.

Another useful variant of this definition is {\it weakly-exceptional}  where we do not assume $E$ is forward invariant. For example each unimodal map $f$ of interval, i.e. with just one turning critical point $c$, has the one-point set $\{f(c)\}$ being a weakly $S$-exceptional set.

\smallskip

Notice that it easily follows from this definition for $S$ or $S'$ that if $x\in K$ belongs to a weakly exceptional set, then the set
$$
O^-_{\rm{reg}}(x):=
	\{y\in(f|_K)^{-n}(x)\colon n=0,1,..., f^k(y)\notin S', \, \hbox{for every}\, k=0,...,n-1\}
$$
 is weakly exceptional. So it makes sense to say that the point $x$ is weakly exceptional.

\

It is not hard to see that if weak topological exactness of $f$ on $K$ is assumed, then each
$S$-exceptional set is finite, moreover with number of elements bounded by a constant, see Proposition~\ref{except-finite}. Therefore the union of $S$-exceptional sets is $S$-exceptional and  there exists a maximal $S$-exceptional set $E_{\rm{max}}$ which is finite. It can be empty. If it is non-empty we say that $(f,K)$ is $S$-{\it exceptional}, or that $f$ is $S$-exceptional. Analogous terminology is used and facts hold for weakly exceptional sets and for $S'$ in place of $S$.

More generally the above facts hold for an arbitrary finite $\Sigma\subset K$ in place of $S$ or $S'$, where $E\subset K$ is called then $\Sigma$-exceptional or weakly $\Sigma$-exceptional, depending as we assume it is forward invariant or not, if $(f|_K)^{-1}(E) \setminus \Sigma \subset E$, see Proposition~\ref{except-finite}.

\

3. \underline{No singular connection condition.} To simplify notation we shall sometimes assume that no critical point is in the forward orbit of a critical point.
This is a convention similar to the complex case. Moreover we shall sometimes assume that no point belonging to $\NO (f,K)\cup \Crit(f)$ is in the forward orbit of a point in $\NO (f,K)\cup \Crit(f)$, calling it {\it no singular connection} condition.

These assumptions are justified since no critical point, neither a point belonging to $NO(f,K)$, can be periodic, see Lemma~\ref{NO}, hence each trajectory in $K$ can intersect $S'(f,K)$ in at most $\# S'(f,K)$ number of times, hence with difference of the moments between the first and last intersection bounded by a constant. In consequence several proofs hold in fact without these assumptions.

\end{defi}

\

\

\subsection{Characterizations of geometric pressure}\label{charact}

\

\

For
$(f,K)\in \sA_+^{\BD}$,
all the definitions of pressure introduced in the rational functions case, see [PR-LS2], make sense for $f|_K$.
In particular

\begin{defi}[Hyperbolic pressure]\label{hyperbolic pressure}
$$
P_{\hyp}(t):= \sup_X P(f|_X,-t\log|f'|) ,
$$
supremum taken over all compact $f$-invariant (that is $f(X)\subset X$) isolated
hyperbolic subsets of $K$.

{\it Isolated} (or \emph{forward locally maximal} or \emph{repeller}), means that there is a neighbourhood $U$ of $X$ in ${\bf{U}}$ such that
$f^n(x)\in {\bf{U}}$ for all $n\ge 0$ implies $x\in X$.
{\it Hyperbolic} or {\it expanding} means that there is a constant $\la_X>1$
such that for all $n$ large enough and all $x\in X$ we have
$|(f^n)'(x)|\ge\la_X^n$. We call such isolated expanding sets {\it expanding repellers} following
Ruelle.

We shall prove that the space of such sets $X$ is non-empty.

Notice that by our definitions $X$ is a maximal invariant set in $U$ its neighbourhood in $\R$, whereas the whole $K$ need not be maximal invariant in any of its neighbourhoods in $\R$, see Definition~\ref{extension-1}  and Example~\ref{Julia}.

$P(f|_X,-t\log|f'|)$ (we shall use also notation $P(X,t)$)
denotes the standard topological pressure for
the continuous mapping $f|_X$ and the continuous real-valued potential
function $-t\log|f'|$ on $X$ , \emph{via}, say, $(n, \e )$-separated sets, see for example \cite{Walters} or \cite{PU}.

\end{defi}

From this definition it immediately follows (compare

\noindent \cite[Corollary 12.5.12]{PU}
in the complex case) the following
\begin{prop}\label{hypdim}{\rm (Generalized Bowen's formula)}
The first zero of $t\mapsto P_{\hyp}(K,t)$ is equal
to the hyperbolic dimension $\HDhyp (K)$ of $K$, defined by
$$
\HDhyp (K):=\sup_{X\subset K} \HD(X),
$$
supremum taken over all compact forward $f$-invariant isolated
hyperbolic subsets of $K$.
\end{prop}

Sometimes we shall assume the following property
\begin{defi}\label{wisolated} We call an $f$-invariant compact set $K\subset \R$ {\it weakly isolated}
if there exists
$U$ an open neighbourhood of $K$ in the domain of $f$
such that for every  $f$-periodic orbit $O(p)\subset U$, if it is in $U$,
then it is in $K$.
We abbreviate this property by (wi).
\end{defi}

In the case of a reduced generalized multimodal quadruple $(f,K, \hI_K,{\bf{U}})$ it is sufficient to consider in this property $U={\bf{U}}$ a neighbourhood of $\hI_K$. Indeed, by the maximality property if
$O(p)$ is not contained in $K$ it is not contained in $\hI_K$. For an example of a topologically exact generalized multimodal pair which does is not satisfy (wi)  see Example \ref{notwi}.

\begin{defi}[Tree pressure]\label{treep}
For every $z\in K$ and $t\in \R$ define
$$
P_{\tree}(z,t)=\limsup_{n\to\infty}{1\over n}\log\sum_{f^n(x)=z, x\in K} |(f^n)'(x)|^{-t}.
$$
Under suitable conditions for $z$ (safe and expanding, as defined below) limsup can be replaced by liminf, i.e. limit exists in in this definition, see
Proof of Theorem B, more precisely Lemma~\ref{Phyp} and the Remark following it.

\end{defi}

To discuss the (in)dependence of tree pressure on $z$ we need the following notions.

\begin{defi}[safe]\label{safe} See \cite[Definition 12.5.7]{PU}. We call $z\in K$
\textit{safe} if $z\notin \bigcup_{j=1}^\infty(f^j(S(f,K)))$ and for every $\delta>0$ and all $n$ large enough
$B(z, \exp (-\delta n))\cap \bigcup_{j=1}^n(f^j(S(f,K)))=\emptyset$.

Notice that by this definition all points except at most a set of Hausdorff dimension 0, are safe.

Sometimes it is sufficient in applications to replace here $S(f,K)$ by $S'(f,K)$ or $\Crit(f)$, and write appropriately $S'$-safe and $\Crit(f)$-safe.
\end{defi}

\begin{defi}[expanding or hyperbolic]\label{expanding}  See \cite[Definition 12.5.9]{PU}.
 We call
$z\in K$ \textit{expanding} or \textit{hyperbolic} if there exist $\Delta>0$ and $\lambda=\lambda_z>1$ such that
for all $n$ large enough $f^n$ maps 1-to-1 the interval

\noindent $\Comp_z(f^{-n}(B(f^n(z), \Delta)))$
to $B(f^n(z), \Delta)$ and $|(f^n)'(z)|\ge \Const\lambda^n$. Here and further on $\Comp_z$ means the component containing $z$.
\end{defi}

 Sometimes we shall use also the following technical condition
\begin{defi}[safe forward]\label{safe forward}
A point $z\in K$ is called \textit{safe forward} if there exists $\Delta>0$ such that $\dist(f^j(z), \partial \hI_K)\ge \Delta$ for all $j=0,1,...$.
\end{defi}

\begin{prop}\label{tree} For every $(f,K)\in \sA_+^{\BD}$,
there exists $z\in K$ which is safe, safe forward and expanding. The pressure $P_{\tree}(z,t)$ does not depend on such $z$.

\end{prop}

For such $z$ we shall just use the notation $P_{\tree}(t)$ or $P_{\tree}(K,t)$ and use the
name \textit{tree pressure}.

For $f$ rational function it is enough to assume $z$ is safe, i.e. $P_{\tree}(z,t)$ does not depend on $z$,
except $z$ in a thin set (of Hausdorff dimension 0). In the interval case we do not know how to get rid of the assumption $z$ is expanding.

  One defines periodic orbits pressure hyperbolic and variational pressure for $(f,K)\in \sA_+^{\BD}$, analogously to [PR-LS2], by

\begin{defi}[Periodic orbits pressure]\label{periodic pressure} Let $\Per_n$ be the set
of all
$f$-periodic points in ${\bf{U}}$ of period $n$ (not necessarily minimal period). Define
$$
P_{\Per}(K,t)=\limsup_{n\to\infty}{1\over n}\log\sum_{z\in\Per_n(f)\cap K}
|(f^n)'(z)|^{-t}.
$$
\end{defi}

\begin{defi}[Hyperbolic variational pressure]\label{variational principle}
\begin{equation}
P_{\varhyp}(K,t) \=
\sup \left\{ h_\mu(f) - t \chi_\mu(f): \mu \in \mathscr{M}^+(f,K), \text{ergodic}, \right\},
\end{equation}
where $\mathscr{M}^+(f,K):=\{\mu\in \mathscr{M}(f,K): \chi_\mu(f)>0 \}$.
\end{defi}

Notice that compared to variational pressure in \eqref{e:variational pressure} we restrict here
to hyperbolic measures, i.e. measures with positive Lyapunov exponent.
The space of hyperbolic measures is non-empty, since $h_{\ttop}(f)>0$. Indeed, then there exists $\mu$,
an $f$-invariant measure on $K$ of entropy $h_\mu(f)$ arbitrarily close to $h_{\ttop}(f)$
by Variational Principle, hence positive. Hence, by Ruelle's inequality, $\chi_\mu(f)\ge h_\mu(f)>0$.

\

\begin{generic}[Theorem B]\label{Theorem B} For every $(f,K)\in \sA_+^{\BD}$
weakly isolated,
all pressures defined above coincide for all $t\in \R$.
Namely
$$
P_{\Per}(K, t) =P_{\tree}(K, t) = P_{\hyp}(K, t)
=P_{\varhyp}(K, t) = P_{\var}(K, t).
$$
For $t<t_+$ the assumption (wi) can be skipped.

\end{generic}

\

Denote any of these pressures by
$P(K,t)$ or just $P(t)$ and call {\it geometric pressure}.

\

The first equality holds for complex rational maps,
under an additional assumption H, see \cite{PR-LS2},
and we do not know whether this assumption can be omitted there.
In the interval case, we prove that this assumption (and even H*, a stronger one, see Section~\ref{pressure-per}) holds automatically.

The proof of $P_{\hyp}(K, t) = P_{\varhyp}(K, t)$ uses Katok-Pesin's theory, similarly to the complex case,
but in the $C^{1+\e}$ setting, allowing flat critical points, see Theorem~\ref{Katok}.

The proofs of the equalities
$P_{\Per}(K, t)=P_{\hyp}(K, t) = P_{\tree}(K, t)$ in the interval case must be slightly modified since we cannot use the tool of short chains of discs joining two points, omitting critical values and their images, see e.g.
\cite[Geometric Lemma]{PR-LS1}. Instead, we just use strong transitivity property, Definition~\ref{backtrans}.
This difficulty is also the reason that we assume $z$ is expanding when proving that
$P_{\tree}(z, t)$ is independent of $z$.

Finally the proof of  $P_{\varhyp}(K, t) = P_{\var}(K, t)$ here
is based on the equivalence of several notions of
non-uniform hyperbolicity, see Theorem C in Subsection~\ref{ss:TCE}
and Section~\ref{TCE}.


\

The definition of conformal pressure is also the same as in the complex case:
\begin{defi}[conformal pressure]\label{conformal pressure}
$$P_{\conf}(K,t):=\log\lambda(t),
$$ where
\begin{equation}\label{lambda}
\lambda(t)=\inf\{\lambda>0: \exists \mu \; {\rm{on}} \, K \; {\rm{which}}\; {\rm{ is}}\;    \lambda|f'|^t-{\rm{conformal}} \},
\end{equation}
see Definition~\ref{defi:Jacobian}
\end{defi}

The proof of Theorem A yields the extension of Theorem B to the conformal pressure for $t_-<t<t_+$,
for maps having only hyperbolic repelling periodic orbits.
We obtain
\begin{coro}\label{ThC-old}
For every $(f,K)\in \sA_+^3$ (or $\sA_+^{\BD}$), all whose periodic orbits in $K$ are hyperbolic repelling
for every $t_-<t<t_+$
$$
P_{\conf}(K,t)=P(K,t).
$$
\end{coro}

Our proof will use inducing and will accompany Proof of Theorem A, see
Subsection~\ref{From the induced map to the original map. Conformal measure}.

\begin{rema}\label{conf}
In Appendix C we provide Patterson-Sullivan's construction.
However it results only with {\it conformal*} measures, the property weaker than conformal, see \eqref{subconformal1} and \eqref{subconformal2}. Their use will yield a different
conformal pressure $P^*_{\conf}(t)$. In Appendix C
we shall discuss its  relation to $P(K,t)$ for all real $t$.

 \end{rema}

\

\subsection{Non-uniformly hyperbolic interval maps}\label{ss:TCE}

\

The main result of Section~\ref{TCE} is the equivalence of several notions of non-uniform hyperbolicity conditions, which are closely related to the \emph{Collet-Eckmann condition}: For every critical point c in K,
whose forward trajectory does not contain any other critical point,
$$
\liminf_{n \to \infty} \frac{1}{n} \log |(f^n)'(f(c))| > 0.
$$

Namely the following is an
extension to generalized multimodal maps of the results for multimodal
maps in \cite{NP} and \cite[Corollaries~A and~C]{R-L}, see also \cite{NS} for the case of unimodal maps and
\cite{PR-LS1} for the case of complex rational maps. See Section~\ref{TCE}  for definitions.

\begin{generic}[Theorem C]\label{Theorem C} For every $(f,K)\in \sA^3_+$ or $\sA_+^{\BD}$,
weakly isolated,
the following properties are equivalent.

$\bullet$ TCE (Topological Collet-Eckmann) and all periodic orbits in $K$ hyperbolic repelling,

$\bullet$ ExpShrink (exponential shrinking of components),

$\bullet$ CE2*$(z)$ (backward Collet-Eckmann condition at some $z\in K$ for preimages close to $K$),

$\bullet$ \ UHP. (Uniform Hyperbolicity on periodic orbits in $K$.)

$\bullet$ \ UHPR. (Uniform Hyperbolicity on repelling periodic orbits in $K$.)

$\bullet$ Lyapunov hyperbolicity (Lyapunov exponents of invariant measures are bounded away from 0),

$\bullet$ Negative Pressure ($P(t)<0$ for $t$ large enough).
\end{generic}


All the definitions will be recalled or specified in Section~\ref{TCE}. See Definition~\ref{pull-back} for the "pull-back".

The only place we substantially use (wi) is that UHP on periodic orbits in $K$ is the same as on periodic orbits in a sufficiently small neighbourhood of $K$, since both sets of periodic orbits are the same by the definition of (wi). We will use (wi) also to reduce the proof to the $(f,K)\in \sA^{\BD}_+$ case.

 A novelty compared to the complex rational case is the proof of the implication
 CE2*$(z_0)$ $\Rightarrow$ ExpShrink, done by the second author for multimodal maps in \cite{R-L}. Here we  present a slight generalization, following the same strategy, with a part of the  proof being different.

\smallskip

\begin{rema}

\

1. For complex rational maps, the TCE condition is invariant under topological conjugacy. However, since a topological conjugacy between interval maps need not preserve inflection critical points, the TCE condition is not, by itself, invariant under topological conjugacy.

2. As opposed to the case of complex rational maps, for interval maps the TCE condition
does not exclude existence of periodic orbits that are not hyperbolic repelling. So, the TCE condition is not, by itself, equivalent to the other notions of non-uniform hyperbolicity considered above.
\end{rema}

\

\subsection{Complementary Remarks}\label{complementary}

\

\

1. The following has been proven recently in \cite{GPR2} (and before in \cite{GPR} in the complex case). Consider the level sets of Lyapunov exponents of $f$
\[
    \cL(\alpha):=
    \big\{x\in K\colon \chi(x):=\lim_{n\to\infty} \frac1n \log\,\lvert (f^n)'(x)\rvert=\alpha\big\}
\]
Then the function $\alpha\mapsto \HD (\cL(\alpha))$ for $\alpha\in (\chi_{\inf},\chi_{\sup})$, called dimension spectrum for Lyapunov exponents, is equal to $F(\alpha)$, where
\begin{equation}
F(\alpha) := \frac{1}{\lvert \alpha\rvert} \inf_{t\in\R}
\left(P(t)+\alpha t \right).
\end{equation}
(a Legendre-like transform of geometric pressure $P(t)$), provided $(f,K)\in \sA_+^{\BD}$ is non-exceptional and satisfies the weak isolation condition, see Definition \ref{wisolated}.

So Theorem A item 1 yields the real analyticity of $\alpha\mapsto \HD (\cL(\alpha))$ for $\alpha\in (\chi^*_{\inf},\chi^*_{\sup})$,
where $\chi^*_{\inf}:=\inf \{\frac{dP}{dt}(t):t\in (t_-,t_+)\},\  \chi^*_{\sup} :=\sup \{-\frac{dP}{dt}(t):t\in (t_-,t_+)\}$.

The same in the complex case was noted in \cite[Appendix B]{PR-L2}.

\smallskip

2. As mentioned at the beginning, if $f$ has no preperiodic critical point in $K$, or is not exceptional, then $t_- = -\infty$ (under additional mild assumptions), see \cite{IT} or \cite[Theorem B]{Zhang}. It may happen that $t_+<\infty$ even if $(f,K)$ is Collet-Eckmann, in particular $1<t_+<\infty$ for $K=I$ for a real quadratic polynomial
with a non-recurrent critical point, see \cite{CR-L}.

The property:
TCE and all periodic orbits in $K$ hyperbolic repelling, is equivalent to $t_+>t_0:=\HD_{\hyp}(K)$ which is hyperbolic dimension of $K$ and the first zero of $P(t)$, see Subsection 1.4.

In addition, if a generalized multimodal map $(f,K)$ is TCE with all periodic orbits in $K$ hyperbolic repelling and
if $K=\hI$ (a union of closed intervals) then $t_0=1$ and Theorem A yields immediately the existence of an absolutely continuous invariant probability measure (acip); this is the equilibrium measure whose existence is asserted there. For multimodal maps
this was first shown in \cite[Corollary 2.19]{R-LS}. 

If $(f,K)$ is not TCE or if $f$ has a periodic point that is not hyperbolic repelling, then $t_0=t_+$ and the existence (or non-existence) of acip is an object of an extensive theory, see e.g. \cite{BRSS}.

\smallskip

3. There are several interesting examples of generalized multimodal maps in the literature. In \cite{R-L2} generalized  multimodal maps were considered each
being the restriction to its domain of a degree 4 polynomial, with $K$ being a Cantor set with one turning critical point in $K$. In these examples $f|_K$ is Collet-Eckmann, but the expansion rates in UHP and ExpShrink occur different ($\lambda_{\Per}>\lambda_{\Exp}$), what cannot happen in the complex case.

A bit similar example also belonging to $\sA_+^{\BD}$ was considered in \cite{P-span}, for which the periodic specification property, uniform perfectness of $K$ (in the plane) and some other standard properties, fail.

Both examples satisfy the assumptions of Theorems A,B,C.

\smallskip

4. It has been proven recently in \cite{P-span} that $P_{\tree}(z,t)$ does not depend on $z$ safe, in particular  it is constant except $z$
in a thin set, as in the rational case, thus answering the question asked after Proposition \ref{tree} (for $t>0$, assumed all periodic orbits are hyperbolic repelling and $K$ is weakly isolated).

\smallskip

5.
Given a generalized multimodal map $(f, K, \hI)\in\cA$,  one can consider its factor $g$  by contracting components
of $\R\setminus \hI$, and their iterated preimages to points, see Appendix A: Remark~\ref{factor} and Step 2 in the proof of
Lemma~\ref{trans implies exact}.
This $g$ is piecewise monotone, piecewise continuous map of an interval $I$.

Notice that it is continuous iff
for each bounded component $Q$ of $\R\setminus \hI$,
$f(\partial Q)$ is one point or
there is a component $P$ of $\R\setminus K$ such that $f(\partial Q)\subset \partial P$.


Clearly  if  $f$ is a classical multimodal map of an interval $I$ into itself and, in the notation in Example~\ref{Julia},
$K=I^+ \setminus B(f)$ is its core Julia set, if $f|_K$ is topologically transitive and $\hI:=I^+\setminus B_0(f)$, then
the factor $g$ defined above is continuous.

\smallskip

\subsection{Acknowledgements}
We thank Weixiao Shen for pointing out to us subtleties in the construction of conformal measures for interval maps. We thank also Micha{\l} Misiurewicz and Peter Raith, see Remark~\ref{wexact1}.

\bigskip


\section{Preliminaries}\label{preliminaries}

\subsection{Basic properties of generalized multimodal pairs}

\

We shall start with Lemma~\ref{infinite} (its first paragraph), which explains in particular why in Definition~\ref{multimodal} the intervals $\hI^j$ are non-degenerate.

\begin{lemm}\label{infinite} If $f:X\to X$ is a continuous map for a compact metric space $X$ and it is topologically transitive, then $f$ maps $X$ onto $X$ and if $X$ is infinite, it has no isolated points and is uncountable.

If we assume additionally that $X\subset \R$ and moreover that for $\bf{U}$ being an open neighbourhood of $X$,\, $(f,X,{\bf{U}})\in \sA$, then $X$ is either the union of a finite
 collection of compact intervals or  a Cantor set.
\end{lemm}

\begin{proof}
 Recall one of the equivalent definitions of topological transitivity which says that for every non-empty
 $U,V\subset X$ open in $X$, there exists $n\ge 1$ such that $f^n(U)\cap V\not=\emptyset$. If $x\in X$ were isolated, then
for $U=V=\{x\}$ there would exist $n(x)\ge 1$ such that $f^{n(x)}(x)=x$. Since $X$ is infinite there is $y\in X\setminus O(x)$. Let $V$ be a neighbourhood of $y$ in $X$ disjoint from $O(x)$. Then $f^n(\{x\})\cap V=\emptyset$ for all $n\ge 1$ what contradicts the topological transitivity since $\{x\}$ is open. $f$ is onto $X$ by the topological transitivity
since otherwise for the non-empty open sets $U=X$, $V=X\setminus f(X)$ for all $n\ge 1$ \; $f^n(U)\cap V=\emptyset$.
$K$ is uncountable by Baire property.

To prove the last assertion suppose that $X$ contains a non-degenerate closed interval $T$.
Extend $f|_X$ to a $C^2$ multimodal map $g:I\to I$ of a closed interval $I$ containing ${\bf{U}}$ with all critical points non-flat. Then by the absence of wandering intervals, see
\cite[Ch.IV, Theorem A]{dMvS},  there exist
$n\ge 0,m\ge 1$ such that for $L:=f^n(T)$ the interval $f^m(L)$ intersects $L$ (the other possibility, that $f^n(T)\to O(p)$ for a periodic orbit in $X$, is excluded by the topological transitivity of $f$ on $X$).  Let $M$ be the maximal closed interval in $X$ containing $L$.
 Then $f^m(M)\subset M$.
 By the topological transitivity of $f$ on $X$, for $M^\circ$ denoting the interior of $M$ in $X$,  the union
 $\bigcup_{j=0}^\infty f^j(M^\circ)$
 is a dense subset of $X$. Hence the family of intervals $M,f(M),...,f^{m-1}(M)$ covers a dense subset of $X$. But $\bigcup_{j=0}^{m-1}f^j(M)$ is closed, as the finite union of closed sets, hence equal to $X$. We conclude that $X$ is the union of a finite collection
of compact intervals.

If $X$ does not contain any closed interval, then since all its points are accumulation points,
$X$ is a (topological) Cantor set. \end{proof}

Let us provide now some explanation concerning the set $S'(f,K)$, in particular $NO(f,K)$, see Definition~\ref{exceptional}, item 1.

\begin{lemm}\label{NO} Consider $(f,K)\in \sA$.
Then $\NO(f,K)\subset \partial(\hI_K)\cup\Crit(f)$ is finite. None of points in $S'(f,K)=\Crit(f)\cup NO(f,K)$ is periodic. The set $\Crit(f)$ cuts $\hI_K=\hI^1\cup ...\cup \hI^{m(K)}$ into smaller intervals ${I'}^1,..., {I'}^{m'(K)}$
such that for each $j$ the restriction of $f$ to the interior of ${I'}^j$ is a diffeomorphism onto its  image and $f({I'}^j\cap K)=f({I'}^j)\cap K$, i.e. a $K$-diffeomorphism according to Definition~\ref{K-diffeomorphism}.

Moreover for each $T$ intersecting $K$, short enough, disjoint from $S'(f,K)$  \ $f:T\to f(T)$ is a $K$-diffeomorphism, i.e.
\begin{equation}\label{K-homeo}
f(T\cap K)=f(T)\cap K.
\end{equation}
The converse also holds for open $T$, namely $K$-diffeomorphism property implies disjointness of $T$ from $S'(f,K)$
\end{lemm}

\begin{proof}
The equalities $f({I'}^j\cap K)=f({I'}^j)\cap K$ hold by the maximality of $K$. Indeed, if the strict inclusion $\subset$ held for some $j$ we
could add to ${I'}^j\cap K$ the missing preimage of a point in $f({I'}^j)\cap K$.
(Equivalently this equality holds by Darboux property, see Proposition~\ref{reduce}).
So the only points where $f|_K$ is not open can be some end points of
${I'}^j$.
Hence $\NO(f,K)\subset \partial(\hI_K)\cup
\Crit(f)$ (in fact one can write $\Crit^T(f)$ here) and the set is finite.

No critical point is periodic. Otherwise, if $c\in\Crit(f)$ were periodic, then its periodic orbit $O(c)$ would be attracting in $K$, what would contradict the topological transitivity of $f|_K$, see Remark~\ref{finite indifferent}. (We recall an argument: a small neighbourhood od $O(c)$ could not leave it under $f^n$, therefore could not intersect nonempty open set in $K$ disjoint from $O(c)$.)

Suppose now that $x\in \NO(f,K)\setminus \Crit(f)$ is periodic. Then $x\in\partial(\hI_K)$ is a limit of $K$ only from one side; denote by $L$ a short interval adjacent to $x$ from this side. Then $f(x)$ is the limit of $K$ from both sides since otherwise $x\notin \NO(f,K)$ since $f$ is a homeomorphism from $L\cap K$ to $f(L)\cap K$. We have already proved that the periodic orbit $O(x)$ is disjoint from $\Crit(f)$. Hence $f^m(x)=x$ is a limit of $K$ from both sides, as $f^{m-1}$ image of $f(x)$. We arrive at a contradiction.

\smallskip

Let us prove the last assertion of the theorem.
Fix $\delta$ less than $g$ being the minimal length of all components of $\R\setminus K$ and all components of $\hat I_K$.
 For each point $f(x)$ for $x\in\partial \hat I_K$ denote by $g_x$ the length of the component of
 $\R\setminus K$ adjacent to $f(x)$ if $f(x)$ is not an accumulation point of $K$ from one side (there cannot be such components on both sides since $K$ has no isolated points, also we do not define $g_x$ if $f(x)$ is accumulated by $K$ from both sides). Let $U'$ be a neighbourhood of $K$ with closure in ${\bf {U}}$.
  Assume $\delta<{\rm dist} (\partial U_K, \partial U')$.
  Let $\Lip$ be Lipschitz constant for $f|_{U'}$.
  Assume finally that $\Lip\delta$ is smaller than all $g_x$.

Consider now an arbitrary, say open, interval $T$ shorter than $\delta$, intersecting $K$, disjoint from $S'$.
$K$ is forward invariant, so the only possibility \eqref{K-homeo} does not hold is
$f(T)\cap K$ strictly larger than $f(T\cap K)$. Then $T$ must contain a point $x\in\partial \hat I_K$.
Indeed, $T$ intersects $\hat I_K$ because it intersects $K$, and if $T\subset \hat I_K$ then
\eqref{K-homeo} holds but the previous part of the lemma.
Let $T'$ be the component of $T\setminus \{x\}$ disjoint from $\hat I_K$. Then if $f(x)$ is accumulated by $K$ from $f(T')$, then $x\in NO(f,K)$, so $x\in S'$, a contradiction. In the remaining case $f(T')\cap K=\emptyset$. Since $T''=T\setminus T'\subset \hat I_K$, we have again the property
\eqref{K-homeo} for $T''$ by the previous part of the lemma. Hence \eqref{K-homeo} holds for $T$.

\smallskip

The converse, i.e. that existence of $x\in S'(f,K)\cap T$ for an open $T$ implies lack of the $K$-diffeomorphism property follows immediately from definitions.
\end{proof}

\

The following property, stronger than topological transitivity and weaker than weak exactness, is useful:

\begin{defi}\label{backtrans} For $f:X\to X$ a  continuous map for a compact metric space $X$
and $x\in X$ denote
$A_\infty(x):=\bigcup_{j\ge 0} f^{-j}(x)$. Then $f$ is said to satisfy
\emph{strong transitivity} property\footnote{In previous versions of this paper we called this property: \emph{density of preimages  property}.},  abbr. (st),
if for every $x\in X$ the set $A_\infty(x)$ is dense in $X$.
\end{defi}
See  e.g. \cite{Kameyama} for this notion and references therein.
It is clear that weak exactness implies (st).
In fact for generalized multimodal pairs,  (st) corresponds to topological transitivity. Namely

\begin{prop}\label{dp}
For every $(f,K)\in\sA$ the map $f|_K$ satisfies strong transitivity property.
\end{prop}
The proof is not immediate. We provide it in Appendix A, in Proof of \ref{trans implies exact}, Step 1.
This fact is related to the fact that each piecewise continuous piecewise monotone mapping of the interval which is topologically transitive is strongly transitive, see \cite{Kameyama}.

In the general continuous mappings setting, it is easy to provide examples of topologically transitive maps which do not satisfy (st).

\

Similarly to Julia sets in the complex case, the following  holds in the interval case:

\begin{prop}\label{density periodic}
For every $(f,K)\in \sA_+$ the set of repelling periodic orbits for $f|_K$ is dense in $K$.
\end{prop}

Below are sketches of two proofs. Unfortunately the proofs are not immediate and refer to material discussed further on in this paper. On the other hand they are rather standard.

 \begin{proof}[Proof 1] Apply Theorem~\ref{Katok} or facts in Section~\ref{pressure} following it. To be concrete, consider an arbitrary safe, safe forward, expanding point $z_0$.
 In Proof of Lemma~\ref{Phyp},
 we find an invariant isolated hyperbolic set $X\subset K$ whose all trajectories (in particular a periodic one) pass arbitrarily close to $z_0$.

 Due to $h_{\ttop}(f|_K)>0$ we find infinitely many such periodic orbits and all except at most a finite set of them are hyperbolic repelling, see Remark~\ref{finite indifferent}. Finally notice that safe, safe forward, and expanding points form in $K$ a backward invariant set (i.e. its preimage is contained in it).
 Therefore having just one such point $z_0$ whose existence is asserted in Lemma~\ref{Phyp}, we have
 the set $A_\infty(z_0)$ of safe, safe forward, expanding points. Clearly this set is dense in $K$,
 see Definition~\ref{backtrans} and the discussion following it.
 \end{proof}

  \begin{proof}[Proof 2] It is easy to prove the density of periodic orbits  for piecewise continuous piecewise affine topologically transitive maps of interval with fixed slope $\beta>1$. It follows for $f|K$ due to the semiconjugacy as in Proof of Lemma~\ref{trans implies exact}, Step 2.
  \end{proof}

 \

Finally let us discuss some easy technical facts related to weak exactness property or to strong transitivity.

\begin{rema}\label{wexact2}
For
$f:X\to X$ a  continuous map for a compact metric space $X$, for $x\in X$ and for each positive integer $k$, denote similarly to $A_\infty(x)$,
$$
A_{k}(x):=\bigcup_{0\le j \le k} f^{-j}(x)
$$

Assume that $f$ is weakly exact.
Then it is easy to see (use compactness and Lebesgue number), that there exists
$N\ge 0$ depending only on $f$ such that for every $\e >0$ there exists $n(\e )\ge 0$
such that for every $n\ge n(\e )$ and every ball $B(x_0,\e )$ for $x_0\in X$, we have
$\bigcup_{j=0}^N f^{n+j}(B(x_0,\e ))= X$.

In particular $A_{n+N}(x)$
is
$\e $-dense in $X$ (that is, for each $y\in X$ there exists $y'\in A_{n+N}(x)$ such that $\dist(y,y') < \e $).
For this however it is sufficient to assume the strong transitivity property, see Definition~\ref{backtrans}.

Namely the following holds:
Assume that $f$ is strongly transitive. Then, for every $\e >0$ there exists $k(\e )$ such that for every $k\ge k(\e )$ and every $x,y\in X$ it holds $A_k(x)\cap B(y,\e )\not=\emptyset$.

\end{rema}

\

\subsection{On exceptional sets}

\

We shall  explain here why exceptional and weakly exceptional sets must be finite, see Definition~\ref{exceptional}, and provide some estimates.

\begin{prop}\label{except-finite}
For every $(f,K)\in\sA_+$ (or equivalently $(f,K)\in\sA$ such that $f|_K$ is weakly topologically exact), for every finite $\Sigma\subset K$ each weakly $\Sigma$-exceptional (in  particular weakly  $S$-exceptional) set is finite and its cardinality is bounded by a constant depending only on $\#\Sigma$ and $(f,K)$.
\end{prop}

\begin{proof}
 Part 1. Let an interval $T$ centered in $K$  be disjoint from a weakly  $\Sigma$-exceptional set $E$.
 It exists since by definition $E$ is not dense in $K$. Denote  $\e =|T|/2$.

Let $k(\e )$ be an integer found for $\e $ as in Remark~\ref{wexact2}. Then for every $z\in E$ and
the set
$A_k(z)$ intersects $T$. Suppose $E$ is infinite. Hence, as, given $k$,  all $\#A_k(z)$ are uniformly bounded,  given $k$, with respect to $z$, there is an infinite sequence of points $z_t\in E$ such that all $A_k(z_t)$ are pairwise disjoint. Indeed, it is sufficient to define inductively an infinite sequence of points $z_t\in E$ such that for every $i<j$ and every $s:0\le s \le  k$ \, $f^s(z_i)\not= z_j$ and
$z_j\notin A_k(z_i)$. If there exists $x\in A_k(z_i)\cap A_k(z_j)$ then there exist $s_1,s_2$ between 0 and $k$ such that $f^{s_1}(x)=z_i$ and $f^{s_2}(x)=z_j$. Then $f^{s_2-s_1}(z_i)=z_j$ or $f^{s_1-s_2}(z_j)=z_i$ depending as $s_2>s_1$ or $s_1>s_2$, a contradiction.

 Since $\Sigma$ is finite, we obtain at least one $A_n(z_k)$ disjoint from $\Sigma$. Then $A_n(z_k)\subset E$ by the definition of an exceptional set, hence $E$ intersects $T$, a contradiction.

\

Part 2. Now we find a common upper bound for $\#E$ for all weakly $\Sigma$-exceptional $E$, depending on $\#\Sigma$ and $(f,K)$.
 For this we repeat the consideration in Part 1. with more care.

 First, consider $O(p)$ a repelling periodic orbit in $K\cap \interior \hI_K$. It exists because there are infinitely many such orbits in $K$ by $h_{\ttop}(f|_K)>0$, see e.g. Theorem~\ref{Katok}.
 Let $T=T(p)$ be now an open interval centered at $p$ with $|T|=2\e $ small enough that the backward branches $g_n$ inverse to $f^{n}$, along $O(p)$, exist on $T$,\; $g_n(T)\subset \hI_K$ and $g_n(T)\to O(p)$ as $n\to\infty$, and else $(g_n(T)\setminus O(p))\cap \Sigma=\emptyset$.

 Define $k=k(\e )$ as in Part 1 of the Proof.
 Denote  $\sup_{z\in K} \#A_k(z)$ by $C=C(\e )$. Find consecutive points $z_t\in E$ by induction as follows.
 Take an arbitrary $z_1\in E$. Next find $z_2\in E$ so that
 $A_k(z_1)\cap A_k(z_2)=\emptyset$. This is possible if $\#E>C+k$. The summand $C$ is put to avoid $z_2\in A_k(z_1)$, the summand $k$ is put to avoid $z_2=f^s(z_1)$ for $0<s\le k$.
 To be able to repeat this until finding $z_m$ for an arbitrary integer $m>0$, we need
 $\#E>(m-1)(C + k)$.

 Assume that $E$ is disjoint from $T(p)$. If $m>\#\Sigma$, then at least one $A_k(z_t)$ for $t=1,...,m$ is disjoint from $\Sigma$ and intersects $T(p)$. Hence $E$ intersects $T(p)$, a contradiction.
 Therefore
 \begin{equation}\label{Card E}
 \#E\le C(f,K,\e ):=(\#\Sigma )(C(\e )+k).
 \end{equation}

 Consider finally $E$ being an arbitrary non-dense weakly $\Sigma$-exceptional set (previous $E$ was  assumed to be disjoint from $T(p)$). Notice that if $\#E>\#O(p)$ then $E':=E\setminus O(p)$ is also a weakly $\Sigma$-exceptional (notice that removal of a forward invariant set from a weakly $\Sigma$-exceptional set leaves the rest weakly $\Sigma$-exceptional or empty).

 We conclude with
 \begin{equation}\label{Card E'}
 \#E \le \#O(p) + (\#\Sigma )(C(\e )+k),
  \end{equation}
  see (\ref{Card E}). Indeed, if this does not hold, then $\#E'> C(f,K,\e )$, hence there exists $w\in E'\cap T(p)$. Then the trajectory of $w$ under $g_n$ omits $\Sigma$ hence it is in $E\setminus O(p)$ (notice that it is in $\hI_K$ hence in $K$ by  maximality of $K$) and therefore $E$ is infinite which contradicts finiteness proved already in Part 1. of the Proof.
\end{proof}

\

\subsection{Backward stability}

\

The following notion is useful
\begin{defi}\label{pull-back} Let $(f,K,{\bf{U}})\in \sA$. For any interval $T\subset {\bf{U}}$ intersecting $K$, or a finite union of intervals each intersecting $K$,
and
for any positive integer $n$, we call an interval $T'$ its \emph{pull-back} for $f^{n}$, of order $n$, or just a \emph{pull-back}
if it is a component of $f^{-n}(T)$ intersecting $K$.

Notice that, unlike in the complex case, $f^n$ need not map $T'$ onto $T$. This can happen either
if $T'$ contains a turning point or if an end point of $T'$ coincides with a boundary point
of ${\bf{U}}$. We shall prevent the latter possibility by considering $T$ small enough (or having all components small enough), see Lemma~\ref{shrinking}.

\end{defi}

\

The absence of wandering intervals, \cite[Ch. IV, Th. A]{dMvS}, see also comments in Example\ref{Julia} and Proof of Proposition~\ref{reduce}, and Lemma~\ref{infinite}, is reflected in the following

\begin{lemm}\label{nonwandering}
 Let $f:U\to \R$ for open set $U\subset\R$ be a $C^2$ map with non-flat critical points
 and $X\subset U$ be a compact $f$-invariant set.

Let $T\subset U$ be an open interval such that one end $x$ of $T$ belongs to $X$ and
$f^k(T)\cap X=\emptyset$ for all $k\ge 0$ (we consider $f^k$ on its domain which can be smaller than $T$).

Then either $x$ is eventually periodic,
that is there exists $n,m>0$ such that
$f^n(x)=f^{n+m}(x)$,
or attracted to a (one-sided) attracting  periodic orbit in $X$.

If $(f,X,U)\in \sA$, then in the former case the periodic orbit $O(f^n(x))$ is disjoint from $\Crit(f)$ and contains a point belonging to $\partial(\hI_X)$.
 \end{lemm}

\begin{proof}
We shall only use $f$ restricted to a small neighbourhood $U_X$ of $X$ in $U$.
We consider $U_X=\{x:\dist(x,X)<\delta\}$ for a positive $\delta$. So $U_X$ has finite number of components. We assume that $\delta$ is small enough that $U_X\setminus X$ does not contain critical points.

Let us define by induction $T_0=T$ and $T_n$ the component of $f(T_{n-1})\cap U_X$ containing $f^{n}(x)$ in its closure.
We consider two cases:

\

Case 1.

There exists $n_0$ such that for all $n\ge n_0$,\,
$f(T_{n-1})\subset U_X$, that is $T_n=f(T_{n-1})$, in particular
$f(T_{n-1})\cap U_X$ is connected so there is no need to specify a component.

Suppose that all $T_n$ are pairwise disjoint for $n\ge n_0$. As in the previous proofs
extend $f|_{U_X}$ to a $C^2$ multimodal map $g:I\to I$ of a closed interval containing $U_X$ with all critical points non-flat.
Then $T_{n_0}$ is a wandering interval for $g$ since on its forward orbit $f$ and $g$ coincide (unless it is attracted to a periodic orbit) which is
not possible by \cite[Ch. IV, Th. A]{dMvS}.

So let $n'>n\ge n_0$ be such that $T_n\cap T_{n'}\not=\emptyset$. Then the end points
$f^n(x)$ and $f^{n'}(x)$ of $T_n$ and $T_{n'}$ respectively, belonging to $X$, do not belong to the open interval $T^1:=T_n\cup T_{n'}$. If these end points coincide, then $f^n(x)=f^{n'}(x)$ i.e. $x$ is eventually periodic.  If they do not coincide
consider $T^2:=f^{n'-n}(T^1)$. The intervals $T^1$ and $T^2$ have the common end $x^1:=f^{n'}(x)$.
Since $f^{n'-n}$ changes orientation on $f^n(T)$ (i.e has negative derivative) and there are no critical points in $T^1\subset U_X\setminus X$, it changes orientation on $T^1$. Hence
$T^2=T^1$ and $x^1$ is periodic of period $2(n'-n)$.

\

Case 2. There is a sequence $n_j\to \infty$ such that $f(T_{n_j-1})$ are not contained in $U_X$.
For each such $n_j$ one end of
$T_{n_j}$ is $f^{n_j}(x)\in X$ and the other end is in $\partial U_X$. Remember that $T_{n_j}$ is disjoint from $X$. Now notice that there is only a finite number of such intervals.
So again there are two different $n_j$ and $n_{j'}$ such that $f^{n_j}(x)=f^{n_{j'}}(x)$, hence
$x$ is eventually periodic.

\

Suppose now that $(f,X,U)\in \sA$. Suppose that $x$ is preperiodic and put $p=f^n(x)$ periodic. Then
 $p\notin\Crit(f)$ since otherwise $p$ would be an attracting periodic point so it could not be in $K$.

If $O(p)\cap\partial(\hI_X)=\emptyset$ then  the forward  orbit of $T':=f^n(T)$ would stay in $\hI_X$. It cannot leave this set because it would capture a point belonging to $\partial(\hI_X)$ hence belonging to $X$.
This is however not possible since by the maximality of $X$ in $\hI_X$ the forward orbit of $T'$ would be in $X$.

\end{proof}

Now we are in the position to prove a general lemma about shrinking of
pull-backs.

\begin{lemm}\label{shrinking} For every $(f,K, {\bf{U}})\in \sA_+$
and for every $\e >0$ there exists $\delta>0$ such that  if $T$ is an arbitrary open interval in $\R$
 intersecting $K$, disjoint from
$\Indiff(f)$,
and satisfying  $|T| \le \delta$,
then
for every $n\ge 0$ and every component $T'$ of $f^{-n}(T)$ intersecting $K$ (i.e pull-back, see Definition~\ref{pull-back})
we have $|T'|\le \e $. Moreover
the lengths of all components of $f^{-n}(T)$ intersecting $K$ converge to 0 uniformly as $n\to\infty$.

In particular, for $\delta$ small enough the closures in $\R$ of all $T'$ are contained in $U$.

\end{lemm}

\begin{proof}
 Since $(f,K, {\bf{U}})\in \sA_+$ then $f|_K$ is weakly topologically exact, Lemma~\ref{trans implies exact}. (We shall not use $h_{\rm{top}}(f|_K)>0$ anymore in Proof of Lemma~\ref{shrinking}.)

Suppose there is a sequence of intervals $T_n$ intersecting $K$,  disjoint from $\Indiff(f)$, with $|T_n|\to 0$, a sequence of integers $j_n\to \infty$ and a sequence of intervals $T'_n$  being some components of
$f^{-j_n}(T_n)$ respectively, intersecting $K$, with $|T'_n|$ bounded away from 0. Then, passing to a subsequence if necessary,
we find a non-trivial open interval $L\subset I$ such that $L=\lim_{n\to\infty} T'_n$ (in the sense of convergence of the end points).

Suppose that $L$ intersects $K$. Then there exists an open interval $L'\subset L$ with
closure contained in $L$ such that $L'$ intersects $K$.
Then $T'_n\supset L'$ for $n$ large enough, hence $L_n:=f^{j_n}(L')\subset T_n$.
By weak topological exactness, see Definition~\ref{wexact},
there exists $N>0$ such that for all $n$ large enough
$\bigcup_{j=0}^N f^j(L_n)\supset K$ which is not possible, since the lengths of $L_n\subset T_n$ tend to 0 and $K$ is infinite.

Suppose now that $L$ does not intersect $K$. Since all $T'_n$ intersect $K$ we can choose
 $a_n\in T'_n\cap K$. Let $a$ be the limit of a convergent subsequence $a_{n_t}$ (to simplify notation we shall omit the subscript $t$). Then
 $a\in K$ by the compactness of $K$, hence it is one of the end points of $L$.

All iterates $f^k$ are well defined on the whole $L$ since for every $n$\; the map $f^{j_n}$ is well defined on $T'_n$ hence all $f^k, k\le j_n$ are well defined on $T'_n$, \;  $j_n\to\infty$ and $T'_n\to L$.
If we replace $L$ by $L'$ having the same end $a$ but shorter at the other end, then as in the previous case $T'_n\supset L'$ for $n$ large enough (all $T'_n$ for $n$ large enough contain $a$ since the points $a_n$ lie on the other side of $a$ than $L$; otherwise $L$ would intersect $K$). Hence $f^{j_n}(L')\subset T_n$. (Though the $f^{j_n}$ are defined on $L$ we need to shorten it to $L'$ to have the latter inclusion. So we shall use $L'$ instead of $L$ in the sequel.)

If there is $k>0$ such that
$f^k(L')$ intersects $K$ then this contradicts weak topological exactness as in the previous case.

If no $f^k(L')$ intersects $K$ we can apply Lemma \ref{nonwandering} and conclude that $a$  is eventually periodic.
Moreover by Lemma \ref{nonwandering} we can find periodic $a'=f^s(a)$  such
that $a'\in\partial(\hI_K)\setminus\Crit(f)$.

We conclude that all
$|f^k(L')|$ for $k$ large enough are bounded away from 0 by a constant $D(a')$.  This is obvious
if  $a'$ (more precisely its orbit $O(a'))$) is repelling. We just define
$D(a'):=\dist (O(a'), \partial W)$ for $W$ as in Definition~\ref{periodic}.
Finally notice that $a'$ cannot be indifferent. Indeed, if $a'$ were one-sided attracting, the attracting side would be the side on which  $f^s(L')$ lied, disjoint from $K$. Hence all $f^s(a_n)$ would be on the other side. Hence $f^s(T'_n)$ and therefore $T_n$ would contain indifferent periodic points (belonging to  $O(a')$) for $n$ large enough, what would contradict the
assumptions.

 We again arrive at contradiction with $|T_n|\to 0$.

\

The proof that if $|T|$ is small enough then for any sequence  $T'_n$ being components of $f^{-j_n}(T)$ respectively, for varying $T$, intersecting $K$ (i.e. pull-backs, see Definition~\ref{pull-back})  we have uniformly $|T'_n|\to 0$, is virtually the same.

Let us be more precise.
Since the set $\partial(\hI_K)$ is finite,  $D=\min_{a'} D(a')$ is positive, where $a'$ are
 repelling periodic points belonging to $\partial(\hI_K)$.
Let $\hD>0$ be an arbitrary number less than $D$ and such that for no $T$ with $|T|\le \hD$ the inclusion $\bigcup_{j=0}^N f^j(T)\supset K$ is possible.

Fix an arbitrary integer $k>0$. Suppose there is a sequence of intervals $T(n)$ intersecting $K$ with $|T(n)|\le \hD$, disjoint from $\Indiff(f)$ and for each $T(n)$ there exist its pull-back $T'_n$ of order $j_n$ with $j_n\to\infty$ as $n\to \infty$, such that for all $n$,\, $|T_n'|\ge 1/k$. Then we find a limit $L$ and $L'\subset L$ and arrive at contradiction as before. Therefore there exists $j(k)$ such that for all $T$ with $|T|\le \hD$ every pull-back of $T$ of order at least $j(k)$ is shorter than $1/k$. This applied to every $k$   gives for $k\to\infty$ the asserted uniform convergence of $|T'_n|$ to 0.

\end{proof}

\begin{rema}\label{back stab}

The first conclusion of
Lemma~\ref{shrinking}, that $|T'| < \varepsilon$, is called \emph{backward
Lyapunov stability}. To obtain this conclusion it is sufficient
 to assume $(f,K)\in \sA$. In particular weak exactness
assumption is not needed. The proof of backward Lyapunov stability can be then proved as follows.

In the case where no $f^k(L')$ intersects $K$,
 only Lemma \ref{nonwandering} has been used in Proof of Lemma \ref{shrinking}, where weak exactness was not assumed.

In the remaining case we can use
\cite[Contraction Principle 5.1]{dMvS}. Namely,
extend $f$ to a $C^2$ multimodal map $g:I\to I$, where $I\supset {\bf{U}}$, as in Proof of Lemma~\ref{nonwandering} (keep for $g$ the notation $f$).
Then, for $L'$ as in  Proof of Lemma \ref{shrinking}, due to $\liminf_{n\to\infty} |f^n(L')|\le \liminf_{n\to\infty} |T_n|=0$, by Contraction Principle, either $L'$ is wandering or attracted to a periodic orbit $O$. The former, wandering case, is not possible by \cite[Ch.IV, Th. A]{dMvS}. In the latter case $O\subset K$, since all $f^n(L')$
intersect $K$ for $n$ large enough by forward invariance of $K$. This however contradicts  topological transitivity of $f|_K$, see Proof of Lemma~\ref{infinite}.

\

Backward Lyapunov stability can fail in the complex case, e.g. for $f$ on its Julia set, even for $f$ being a quadratic polynomials with only hyperbolic repelling periodic orbits. See \cite[Remark 2]{Levin}.

\end{rema}

\


\subsection{Weak isolation}

\

Below is the example of a generalized multimodal pair not satisfying (wi), promised in Introduction
\begin{exem}\label{notwi}
First consider $f:[-1,1]\to [-1,1]$ defined by $f(x):=4x^3-3x$. This is degree 3 Chebyshev polynomial with fixed points $-1$ and 1 and critical points $-1/2, 1/2$ mapped to 1 and $-1$ respectively.
For $x: -1\le x \le 1, x\not=\mp 1/2$ let $i(x)=0$, 1 or 2, depending as $x<-1/2, -1/2<x<1/2$ or $x>1/2$.
Define the itinerary of each $x\in(-1,1)$  whose trajectory does not hit $\mp 1/2$ by the infinite sequence
$\underline i(x)=(i(x),i(f(x)),...,i(f^n(x)),...)$ and of each $x$ such that $f^n(x)=-1/2$ or $1/2$ by the finite sequence $(i(x),i(f(x)),...,i(f^{n-1}(x)))$, empty for $x=\mp 1/2$. Define the itinerary as empty set also for $x=\mp 1$. Later on we shall usually omit commas in the notation of these sequences. Let $T$ be the open interval of points in $[-1,1]$ for which the itinerary starts with the block $11$. Then the forward trajectory of $\partial T$ is disjoint from the closure of $ T$ (compare Definition \ref{nice couple}).
We define
$$
K:= [0,1]\setminus \bigcup_{n=0}^\infty f^{-n}(T).
$$
Notice that $K$ can be defined as the set of all points $[-1,1]$ for which the itinerary does not contain the forbidden block $11$.

Finally define $\hI_K:=[-1,1]\setminus T$, and consider $(f,K,\hI_K,{\bf{U}})\in \sA$,
where $\bf{U}$ is an arbitrary two components neighbourhood of $[-1,1]\setminus T$ and $f$ is smoothly extended to ${\bf{U}}\setminus [-1,1]$  keeping the previous $f$ on $[-1,1]$, in particular on $[-1,1]\cap{\bf{U}}$.

\

The map $f|_K$ is topologically transitive and even topologically exact.
 Here is the standard proof: $f|_K$ is the factor of the topological Markov chain $(\sigma,\Sigma_A)$, with the $3\times 3$ transition 0-1 matrix $A=(a_{ij})$ having all entries equal to 1, except $a_{11}=0$.
 $\Sigma_A$ is by definition the space of all infinite sequences $a_0a_1...$ where $a_i=0,1$ or 2 with no two consecutive 1's and $\sigma$ is the shift to the left, $\sigma((a_j))_k=a_{k+1}$.   The projection, i.e. the `coding' $\pi$, is defined in the standard\footnote{In our Chebyshev case  we can just write $\pi(a_0a_1...) = - \cos ( \pi  \sum_{n = 0}^\infty  a_n / 3^n)$,
  and use
  the identity $f(\cos(\theta)) = \cos(3\theta)$ to get $f\circ\pi = \pi\circ\sigma$.}
  way, as follows.
 Consider an arbitrary sequence $a_0a_1...$ of integers 0,1, or 2, not containing the forbidden block 11 and for each integer $n$ consider the associated cylinder $C_n(a_0...a_n)$,
  i.e. the set of all sequences in $\Sigma_A$ starting from the this block. Define $\pi'(C_n(a_0...a_n))$ as the set of all points in $[-1,1]$ whose itinerary starts from the block $a_0...a_n$. Then $\pi((a_0a_1...))$ is (uniquely) defined as $\bigcap_{n\to\infty}\overline{\pi'(C_n(a_0...a_n))}$.
 We obtain by construction $\pi\circ\sigma= f\circ\pi$. The continuity of $\pi$ also follows easily from the construction.

Consider an arbitrary $C_n=C_n(a_0...a_n)$.
Then $\sigma^{n+2}(C_n)=\Sigma_A$, i.e. the whole space. This is so because every sequence $(b_j)\in\Sigma_A$ is the $\sigma^{n+2}$-image of $(a_0...a_na_{n+1}b_0b_1...)$ belonging to $C_n$  for $a_{n+1}$ being 0 or 2, since the latter sequence does not contain forbidden 11. (Notice that in the case $a_n\not=1$ we do not need $a_{n+1}$ hence $\sigma^{n+1}$ would be sufficient.)
The topological exactness of $\sigma$ on $\Sigma_A$ immediately implies the topological exactness of the factor $f|_K$.

It is easy to calculate that $h_{\ttop}(\sigma)=\log (1+\sqrt 3)$, which is positive.
Since $\pi$ is a coding  {\it via} Markov partition,  $h_{\ttop}(\sigma)=h_{\ttop}(f|_K)$, compare e.g.
\cite[Theorem 4.5.8]{PU}; in fact this equality easily follows from the fact that our coding $\pi$ is at most 2-to-1 (use
the definition of the topological entropy {\it via} $(n,\e )$-nets, for any such net in $K$ consider its $\pi$-preimage in $\Sigma_A$).
In particular $h_{\ttop}(f|_K)$ is positive. Thus $(f,K)\in \sA_+$.
(In fact `positive entropy' follows already from the topological exactness, see Proposition~\ref{exact implies htop}.)

\

For each $n\ge 2$ there is a  periodic point in $[-1,1]\setminus K$
of period $n$ with the periodic orbit having the itinerary being the concatenation of the blocks
$1100...0$ of length $n$, hence intersecting $T$, therefore not in $\hI_K$. It is arbitrarily close to $K$ for $n$  large, since the interval encoded by the block $1100...0$ of length $n$ is adjacent to  the interval encoded by the block $1200...0$ of length $n$. Therefore the weak isolation condition (wi) is not satisfied.
\end{exem}

\

\subsection{Bounded Distortion and related notions}\label{ss-distortion}

\

\

In Section~\ref{s:nice} we shall use `bounded distortion' properties formally stronger than BD defined in Definition~\ref{distortion}.

\begin{defi}\label{LBD}
We say that $(f,K,{\bf{U}})$ satisfies \emph{H\"older bounded distortion} condition,
abbr. HBD, if there exist constants $\alpha:0<\alpha\le 1$ and  $\delta>0$ such that
for every $\tau>0$ there exists a constant $C(\tau)$ such that the following holds:

For every  pair of intervals $I_1 \subset {\bf{U}}$, $I_2\subset \R$ such that
$|I_2|\le \delta$,
if $f^n$
maps diffeomorphically  $I_1$ onto $I_2$ for a positive integer $n$, then
for every interval $T\subset I_2$ such that $I_2$ is a $\tau$-scaled
neighbourhood of $T$, for $g=(f^n|_{I_1})^{-1}$ and for all $x,y\in T$
$$
|g'(x)/g'(y)|\le C(\tau)(|x-y|/|T|)^\alpha.
$$

For $\alpha=1$ we call this condition: \emph{Lipschitz bounded distortion} condition and
abbreviate to LBD.

\end{defi}

\begin{rema}\label{LBD2}

The conditions BD and even LBD are true if $(f,K,{\bf{U}})$ is in $\sA^3$ and  Schwarzian derivative of $f$  is negative on ${\bf{U}}$.
More precisely
$$
|g'(x)/g'(y)|< ({1+2\tau\over \tau^2}+1){|g(x)-g(y)|\over |g(T)|},
$$
see e.g.
\cite[Section 2.2]{BT1} or \cite[Ch.IV, Theorem 1.2]{dMvS}.
In consequence
$$
|g'(x)/g'(y)|< C(\tau) {|x-y|\over |T|}.
$$

The conditions BD and LBD are true for every multimodal map of interval $f:I\to I$ if $f$ is $C^3$ and all periodic orbits are hyperbolic
repelling, by an argument in
\cite[Chapter 3]{BRSS} using \cite{vSV}, decomposing $f^n$ into a negative Schwarzian block till the last occurrence of $f^j(I_1)$ close to $\Crit(f)$ and an expanding one, within a distance from $\Crit(f)$.

 Hence we can assume that BD and even LBD are true for $(f,K,{\bf{U}})\in\sA^3$ if all periodic orbits in $K$ are hyperbolic repelling, by an appropriate modification (if necessary) of $f$
outside $K$ (in fact only outside $\hI_K$) so that also outside $K$ all periodic orbits are hyperbolic repelling. See Appendix A, Lemma~\ref{good extension} for details, and \cite[Chapter 3]{BRSS} cited above.

\smallskip

In this paper if we need to use BD, similarly LBD or HBD, for $(f,K,{\bf{U}})\in\sA$ but do not need smoothness higher than $C^2$, we just assume them.

\smallskip

It is conceivable, see \cite{vSV}, BD and LBD hold for all $(f,K,{\bf{U}})\in \sA$ (i.e. $C^2$)  provided the set of
all periodic orbits in $\bf{U}$
except hyperbolic repelling orbits, is in a positive distance from $K$, i.e. there are no such orbits if we shrink $\bf{U}$.

\end{rema}

\

\section{Non-uniformly hyperbolic interval maps}\label{TCE}

In this Section we shall prove Theorem C.
Let us recall some definitions, see e.g. \cite{PR-LS1}, \cite{NP} and \cite{R-L}, adapted to  $C^2$ generalized multimodal quadruples, namely for $(f,K,\hI_K,{\bf{U}})\in \sA$, see Introduction.

\

\noindent
$\bullet$ \ TCE. \  {\it Topological Collet-Eckmann condition.}
        There exist  $M \ge 0, P \ge 1$ and $r>0$ such that for every
$x\in K$ there exists a strictly increasing sequence of positive integers
$n_j$, for $j=1,2,...$ such that  $n_j \le P \cdot j$ and for each $j$
$$
        \#\{i:0 \le i < n_j, \Comp_{f^i(x)}  f^{-(n_j-i)}B(f^{n_j}(x),r)
                \cap\Crit(f) \not= \emptyset \} \le M.
$$

\noindent
$\bullet$ \ ExpShrink. \ {\it Exponential shrinking of components.}
        There exist $\lambda_{\Exp}>1$ and $r>0$  such that for every $x \in K$,
every $n > 0$ and every connected component $W$ of $f^{-n}(B(x, r))$ intersecting $K$ (pull-back, see Definition~\ref{pull-back}), we have
$$
        |W| \le \lambda_{\Exp}^{-n}.
$$

\

\noindent
$\bullet$ \ Lyapunov hyperbolicity. {\it  Lyapunov exponents of invariant measures are
bounded away from zero}.
        There is a constant $\lambda_{\Lyap} > 1$ such that the Lyapunov exponent
of any
        invariant probability measure $\mu$ supported on $K$
        satisfies $\Lambda(\mu)\ge \log \la_{\Lyap}$.

\

\noindent
$\bullet$ \ Negative Pressure. {\it Pressure for large $t$ is negative}.
        For large values of $t$ the pressure function
        $P(t)$ is negative.

\

 \noindent
$\bullet$ \ UHP. \ {\it Uniform Hyperbolicity on periodic orbits.}
        There exists $\lambda_{\Per} > 1$ such that every periodic point $p
\in K$ of period $k \ge 1$ satisfies,
$$
        |(f^k)'(p)|\ge \lambda_{\Per}^k.
$$

\noindent
$\bullet$ \ UHPR. \ {\it Uniform Hyperbolicity on hyperbolic repelling periodic orbits.}
The same as UHP but only for hyperbolic repelling periodic points in $K$.

\

\noindent
$\bullet$ \ CE2*$(z_0)$. \ {\it Backward} or {\it the Second Collet-Eckmann condition
 at $z_0 \in K$.}
        There exist $\lambda_{\CE2}= \lambda_{\CE2}(z_0) >1$ and $C>0$ such that for every $n \ge 1$ and
every $w \in f^{-n}(z_0)$ close to $K$,
$$
|(f^n)'(w)|\ge C \lambda_{CE2}^n.
$$
\emph{close to $K$} means that for  a constant $R>0$ (not depending on $n$ and $w$) the pull-back of $B(z_0,R)$ for $f^{n}$ containing  $w$
intersects $K$.



In the proof of CE2*$(z_0)\Rightarrow$ ExpShrink (see the comment in Subsection~\ref{ss:TCE}),  we shall use the following fact:

\begin{lemm}\label{bg}
There exist $L\ge 1$ and $\kappa, a>0$ such that for every interval $T$ intersecting $K$ and its
pullback $T_j$ for $f^j$ (i.e. connected component $T_j$ of $f^{-j}(T)$ intersecting $K$) if $|T|\le a$,
then
\begin{equation}\label{backward growth} |T_j|\le L^j|T|^\kappa.
\end{equation}
\end{lemm}

  The proof is as in the complex setting in \cite[Lemma 3.4]{DPU} and is based on the "Rule II" of \cite{DPU} adapted to the real case in \cite{NP}. This rule is a general estimate on the average distance (in logarithmic scale) of any finite orbit to the set of critical points. It was considered in
\cite[Appendix]{NP} for multimodal maps of interval with no attracting periodic orbits. We can extend $f$ to such a map similarly to Lemma~\ref{good extension}, as due to bounded distortion the mapping $f$ has no attracting periodic orbits in a neighbourhood of $\hI_K$.

\smallskip

We shall need also the following standard
\begin{lemm}\label{critical distortion} Let $f:U\to \R$
be a map of class $C^2$ with only non-flat
critical points.
Then there exists $C_1>0$ such that for every interval $T\subset \R$, every component
$T_1$ of its pre-image by $f^{-1}$ and for every $x\in T_1$, we have
$$
|T_1|/|T|\le C_1 |f'(x)|^{-1}.
$$
\end{lemm}

\

\begin{proof}[Proof of Theorem C]

 Notice that each of the properties listed in the statement of Theorem C implies that all periodic trajectories in $K$ are hyperbolic repelling. For the non-trivial case: UHPR, see Remark~\ref{onUHP}.

Assume that $(f,K)\in \sA^3_+$. By the hyperbolicity of periodic orbits in $K$ mentioned above (or just due to the assumption of weak isolation), the assumption 2. in
Lemma~\ref{good extension} holds. Hence by Lemma~\ref{good extension} we can change $f$ to $g$ outside a neighbourhood of $\hI_K$ so that all periodic orbits for $g$ outside $K$ are hyperbolic repelling. If we assume that also all periodic orbits in $K$ are hyperbolic repelling then $(g,K)\in\sA_+^{\BD}$, see
Subsection~\ref{ss-distortion}.

\smallskip

Since all the properties listed in Theorem C
depend only on the map on a small neighbourhood of $K$ (some only on $f$ on $K$ itself),
it is sufficient to prove Theorem C for $(f,K)\in\sA_+^{\BD}$.

\smallskip

The fact that ExpShrink implies TCE was proved in the interval case in \cite[Section 1]{NP}.
We used there the "Rule II" estimate mentioned above. Here the proof is similar. 

The opposite implication for maps having only hyperbolic repelling periodic points does not have a direct proof in \cite{NP} but one can adapt word by word
 the proof provided in \cite[Section 4]{P-Holder} for the complex case, namely the "telescope method". One applies Lemma \ref{shrinking}.

The proof of ExpShrink $\Rightarrow$ Lyapunov hyperbolicity, is the same as in the complex case, see
\cite[Proposition 4.1]{PR-LS1}. Also the proof of Lyapunov hyperbolicity $\Leftrightarrow$ Negative Pressure
is the same. Lyapunov hyperbolicity immediately implies UHP. The only difficulty is to prove that
UHP (or UHPR) implies ExpShrink. In \cite{NP} the paper \cite{NS} was referred to, but it concerned only the unimodal case. Below is a proof for the generalized multimodal case,
following the same strategy as for the
case of multimodal maps in \cite{R-L}, but with a part of the proof being
different. We use Corollary~\ref{bg}, not needed in \cite{R-L}.

\

First notice that UHPR implies EC2*$(z_0)$. For this it is sufficient to take an arbitrary $z_0\in K$ safe and expanding. Its existence follows from Lemma~\ref{Phyp}.

Fix
$\e >0$ such that $B(K,\e )\subset U$, the neighbourhood of $K$ in the definition of (wi), Definition~\ref{wisolated}. Let $R:=\delta>0$ be chosen for $\e >0$ as in Lemma~\ref{shrinking}.  Consider an arbitrary $z_n=w\in f^{-n}(z_0)$ as in the definition of CE2*$(z_0)$, i.e. such that $T_n$ being the pull-back  of $B(z_0,R)$ for $f^{n}$, containing $z_n$ intersects $K$. We find a repelling periodic orbit $O(p)$ which follows $f^j(z_n)$ for $j=0,1,...,n$, with period $n(p)$ not much bigger than $n$ and $\log |(f^n)'(z_n)|/n$ is at least (up to a factor close to 1)
$\log |(f^{n(p)})'(p)|/n(p)$ for $n$ large, hence positive bounded away from 0 by UHPR.
By (wi) $O(p)\subset K$.

We have followed here Proof of Lemma \ref{Phyp}, where more details are provided.
See also, say, \cite[Lemma 3.1]{PR-LS1} for a detailed proof.

\

Now assume that CE2*$(z_0)$ holds. We shall prove ExpShrink. Notice first that there are only hyperbolic repelling  periodic points in $K$, since otherwise the derivatives of $f^{-j}$ of backward branches at $z_0$ converging
to periodic orbits that are not hyperbolic repelling would not tend to 0 exponentially fast.

Notice that the set of $z$'s for which CE2*$(z)$ holds is dense in $K$ because it is backward invariant and $\bigcup_{n=0}^\infty f^{-n}(z_0)\cap K$ is dense in $K$, see Proposition~\ref{dp}. One should be careful however because the constant $C$ in the definition of CE2*$(z)$ can depend on $z\in \bigcup_{n=0}^\infty f^{-n}(z_0)$.

To prove ExpShrink assume to simplify notation (this does not hurt the generality) that no critical point in $K$ contains another critical point in its forward orbit, see Definition \ref{exceptional}, item 2.

\

Case 1. Consider an arbitrary interval $T$ of length not bigger than $R$, whose end points $z_0, z'_0$ satisfy CE2*
with
common constants $R$, $C$ and $\lambda$.
Consider $S\supset T$ its $1/2$-scaled neighbourhood in $\R$.
Consider consecutive pull-backs $T_j$ of $T$ intersecting $K$ and accompanying pull-backs $S_j$ of $S$ until $S_n$ captures a critical value for the first time, for some $n=n_1$.
By bounded distortion property  Definition~\ref{distortion},
writing $T_n=[z_n,z'_n]$, we get
\begin{equation}\label{ratio1}
|T|/|T_n|\ge  C(1/2)^{-1} \max \{|(f^n)'(z_n)|, |(f^n)'(z'_n)|\}.
\end{equation}
Next, for a constant $C_1$ by Lemma~\ref{critical distortion}

\begin{equation}\label{ratio2}
|T|/|T_{n+1}|\ge  C_1^{-1} C(1/2)^{-1} \max \{|(f^{n+1})'(z_{n+1})|, |(f^{n+1})'(z'_{n+1})|\}.
\end{equation}
Here we also write $T_{n+1}=[z_{n+1},z'_{n+1}]$, but unlike before both $z_{n+1}$ and $z'_{n+1}$ can be $f^{n+1}$ pre-images of the same $z$ or $z'$, i.e. both can be $f$-pre-images of the same $z_n$ or $z'_n$.

Next we consider $S^1$ the $1/2$-scaled neighbourhood of $T^1:=T_{n+1}$ and pull-back as before until for $n=n_2$ the pull-back $S^1_{n}$  of $S^1$ captures a critical point.

We continue for an arbitrarily long time $m$. Notice that each $n_s$ is larger than an arbitrary constant if $T$ is short enough. This follows from Lemma~\ref{shrinking}.

Hence we finish with
  $$
|T|/|T_{m}|\ge (C_1 C(1/2))^{-\beta m} \max \{|(f^{m})'(z_{m})|, |(f^{m})'(z'_{m})|\}.
$$
for $\beta>0$ an arbitrarily small constant.
Hence
 \begin{equation}\label{ratio3}
|T|/|T_{m}|\ge  (C_1 C(1/2))^{-\beta m}  C \lambda^m.
\end{equation}
This proves ExpShrink for $T$ with $\lambda'_{\Exp}$ arbitrarily close to $\lambda$.
Case 1 is done.

\

Let $\{Q^j\}$ be the family of all open intervals in $\R\setminus K$ with boundary points in $K$ or $\pm\infty$ of length at least $R/6$.
Let $\partial Q$ denote the set of all finite boundary points of $Q:=\bigcup Q^j$.
Denote $D_1:=f(\Crit(f))\setminus \partial Q$ and $D_2=f(\partial Q)\setminus \partial Q$.

Let $\rho=\min_{i=1,2}\dist(D_i,  \partial Q)$. By definition, $\rho>0$. Finally let $N\ge 0$ be an integer such that for each $n\ge N$, each pull-back for $f^{n}$ of every interval
of length not larger than $R$, intersecting $K$, has length less than $R':=\min\{\rho/2, R/3\}$.

Fix $r:0<r<\min \{\rho/2, R/12\}$. Fix $A\subset K$ an arbitrary finite set of points for which CE2* holds and which is $r$-dense in $K$. We consider common $C$ in the definition of CE2* for all points in $A$.

\smallskip

Now consider an arbitrary  interval $T$ intersecting $K$ of length at most $R$.
Replace $T$ by $\hat T$ being a pull-back of $T$ for $f^N$. Then $|\hat T|\le R'$ and the same estimate holds for all pull-backs of $\hat T$.

The case where $\hat T$  is contained in an interval $[z,z']$ of length not bigger than $R$,
for $z,z'\in A$,
follows from Case 1.

Below we consider the remaining case where, roughly, there is a gap $Q^j$ on at least one side of $\hat T$.

\

Case 2.  $\hat T$ intersects $[b,b']$ for $b\in\partial Q$ and $b'\in A$ the closest to $b$ point in $A$. We set $z_0=b'$ and consider a new $T$ by adding $[b,b']$ intersecting $\hat T$ to the original $\hat T$. Its length is less than
$R/12 + R/3 <R/2$.

Let us start to proceed as before by taking pull-backs of $S$ being the $1/2$-scaled neighbourhood of the new $T$. Fix an arbitrary $m>0$.
Let $k:0<k\le m$ be the first time when the pull-back $T_k$ does not intersect any $[b,b']$. Then we first estimate the ratio
$|T|/|T_k|$.
If $k$ above does not exist we estimate the final $|T|/|T_m|$.

Notice that for all $j\le k$ the corresponding preimages $z_j\in T_j$ of $z_0$ exist and belong to $\R\setminus Q$.
We apply induction. For all $0\le j < k$ we have $|T_j|<R'$ hence $T_j\subset B(\partial Q, r+R')$. Since
$r<\rho/2$ and $R'<\rho/2$,\; $T_j\subset B(\partial Q,\rho)$. Hence, for $T'_j:=T_j\setminus \overline{Q}$, by $\rho<D_2$, we obtain
$T_j'\cap f(\partial Q)=\emptyset$. Suppose that $z_j\in T'_j$. Hence $T'_{j+1}$, the pull-back of $T'_j$ for $f$ in $T_{j+1}$, intersects $K$ since
$f^{-1}(T_j\setminus T'_j)$ is disjoint from $K$, since $T_j\setminus T'_j$ is disjoint from $K$, by
$f$-invariance of $K$. Hence $T'_{j+1}$ is entirely in $\R\setminus Q$, hence $z_{j+1}\in K$ by the maximality of $K$.

This consideration finding $z_j\in T_j$ is complete as long as
none of the $T_j$ contains a turning point of $f$.
So suppose that for some $n+1\le k$ the pull-back $S_{n+1}$ captures a critical point $c$ for the first time.

By our definitions  $T_n$ intersects an interval $[b,b']$.
As before $T_j\subset B(b,\rho)$.
Hence, by $\rho< D_1$, we get $f(c)=b$.

If $T_{n+1}$ contains a turning  critical point $c$,
the $f$-image of a neighbourhood of $c$ (the "fold") cannot be on the other side of $b$ than $b'$, otherwise $c$ would be isolated in $K$.
Therefore, by $z_n\in T'_n$, the pullback
$T_{n+1}$ contains a preimage (even two preimages) of $z_n$, both in $\R\setminus Q$.

To cope with distortion we deal as in Case 1., pulling back $1/2$-scaled neighbourhoods $S^t$ of corresponding $T^{n_t}$ until consecutive capture of a critical point.

 We conclude with the estimate
(\ref{ratio3}) with $m=k$.

(Notice that in fact a capture of a critical point can happen only a finite number of times, since otherwise $\partial Q$, hence $K$, would contain an attracting periodic orbit, containing a critical point.)

\

Still an estimate of $|T_k|/|T_m|$ from below is missing.
We are not able to make an estimate depending on $|T_k|$ and will just rely on an estimate of $1/|T_m|$.

We have  $T_k\subset [z,z']$ for $z,z'\in A$ and $|z-z'|<R$.
Hence, by Case 1, for $T^*_s$ denoting the pull-back of $[z,z']$ for $f^s$ containing $T_{k+s}$ we have
$
|z-z'|/|T^*_s|\ge (C_1C(1/2))^{-\beta s}C\lambda^s$. Hence for $\xi>0$ arbitrarily close to 0 and
respective $\cC(\xi)$ we get $|T^*_s|\le \cC (\xi)\lambda^{-s(1-\xi)}$.
We have also, by Corollary~\ref{bg}, $|T_{k+s}|\le L^s|T_k|^\kappa$.
Summarizing, we have
\begin{equation}
|T_{k+s}|\le \min \{\cC (\xi)\lambda^{-s(1-\xi)}, L^s|T_k|^\kappa \}
\end{equation}

Combined with (\ref{ratio3}) in the form $|T_k|\le \cC (\xi)\lambda^{-k+\xi}$, it can be calculated that
there is $\lambda'>1$ such that $|T|/|T_m|\ge \Const {\lambda'}^m$.
One obtains $\log\lambda'$ arbitrarily close to\footnote{Note that with the proof of case 2 given here, the rate of shrinking is strictly less than $\log\lambda$. In \cite[Main Theorem']{R-L} it is shown that the rate of shrinking is at least $\log\lambda$, provided~$T$ is disjoint from the large gaps of~$\R \setminus K$.}
$$
\frac{\kappa (\log\lambda)^2}{\log L +(1+\kappa)\log \lambda}.
$$

\


\begin{figure}[h!]
\centering
\includegraphics[height=6cm]{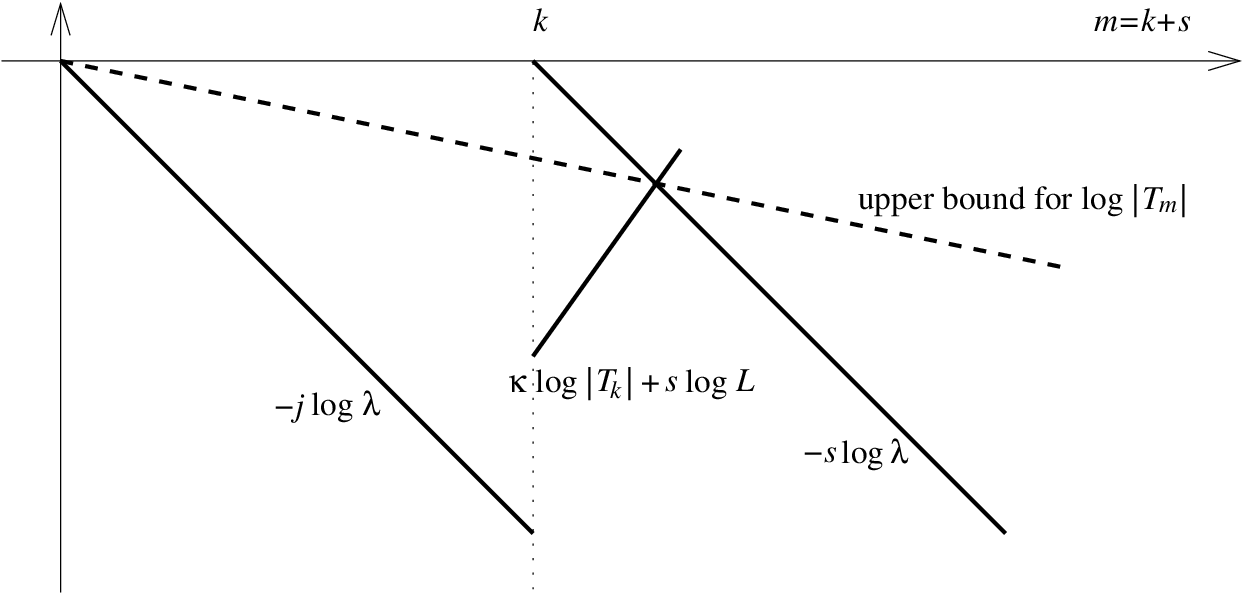}
\caption{Upper bound for $\log |T_m|$}
\end{figure}

\end{proof}

\

\begin{coro}\label{chi0}
For $(f,K)\in \sA_+^3$ or $\sA_+^{\BD}$, satisfying (wi), if $\chi(\mu)=0$ for $\mu$ an $f$ invariant measure on $K$ ($K$ as in Theorem C),
then there exists a sequence of repelling periodic points $p_n\in K$
such that $\chi(p_n)\to 0$.
\end{coro}

\begin{proof} If $p_n$ does not exist then by definition UHPR holds.
Hence by Theorem C we have $\chi(\mu)>0$, a contradiction.
\end{proof}

\begin{rema}\label{onUHP}

Notice that the equivalence of UHP and UHPR has a direct proof omitting the part of Theorem C  that EC2*($z_0)$ implies ExpShrink.
(The implication UHPR implies UHP is a special case of the above Corollary.)

Indeed. Suppose there exists a periodic point $p\in K$ that is not hyperbolic repelling.  Then it is (one-sided) repelling by Remark~\ref{finite indifferent}. Take $z_0\in K$ safe hyperbolic.  It has a backward trajectory  $x_j\in f^{-j}(x_0)$ converging to $O(p)$ (from the repelling side). Fix $n$ and find a periodic point $q\in K$ most of whose trajectory "shadows" $x_j: j=0,...,n$, as in Lemma~\ref{Phyp} (or in Proof of Theorem C, the part UHPR implies EC2$(z_0)$), with $\chi(q)>0$ arbitrarily close to 0,
that contradicts UHPR.
\end{rema}

\

\section{Equivalence of the definitions of Geometric
Pressure}\label{pressure}

Let us formulate the main theorem of Katok-Pesin's theory adapted to the multimodal case, similar to the complex case
\cite[Theorem 11.6.1]{PU}. Compare also \cite{MiSzlenk}.

\begin{theo}\label{Katok}
Consider $f:U\to \R$ for an open set $U\subset \R$, being a $C^{1+\e}$ map (i.e. $C^1$ with first derivative H\"older continuous) with at most finite number of critical points.\footnote{In the arXiv:1405.2443v1 version of this paper we assumed the critical points were non-flat.
As observed first in \cite{Dobbs3}, this is not needed for the unstable manifold theorem.
We note that \cite{Dobbs3} imposes more restrictive hypotheses to deal with backward orbits near the  end points of the domain of the map.}
Consider an arbitrary compact $f$-invariant $X\subset U$.
Let $\mu$ be an $f$-invariant ergodic  measure on $X$, with positive Lyapunov exponent.

Let
$\phi:U\to \R$ be an arbitrary continuous function.
Then  there exists a sequence $X_k$, $k=1,2, ...$, of compact
$f$-invariant subsets of $U$, (topologically) Cantor sets or individual periodic orbits\footnote{This case was overlooked in the assertion of \cite[Theorem 11.6.1]{PU} (but not in the proof).}, such that for every $k$
the restriction
$f|_{X_k}$ is an expanding repeller,
$$
\liminf_{k\to\infty} P(f|_{X_k},\phi)\ge h_{\mu}(f) + \int\phi\,d\mu ,
$$
and if $\mu_k$ is any  $f$-invariant measure on $X_k$, $k=1,2, ...$, then the
sequence $\mu_k$ converges to $\mu$ in the weak*-topology.
Moreover
\begin{equation}\label{Lyap approx}
\chi_{\mu_k}(f|_{X_k})=\int\log |f'|\,d\mu_k \to
\int\log |f'|\,d\mu=\chi_\mu(f).
\end{equation}
If $X$ is weakly isolated, see Definition~\ref{wisolated}, then one finds $X_k\subset X$.

If the weak isolation is not assumed, but $X$ is maximal, more precisely if $(f,X,\hI_X, \bf{U})\in\sA_+$ (see Notation~\ref{sets-A})
and if $h_\mu(f)>0$, then one can find $X_k\subset X$. Moreover one can find $X_k\subset \interior \hI_X$. In fact the non-flatness at critical points  need not be assumed here and $C^{1+\e}$ is enough, as in the preceding part of the Theorem.

\end{theo}

\begin{proof} It is the same as in \cite[Theorem 11.6.1]{PU} (except the last assertion where,
unlike in \cite{PU}, we do not assume $X$ is a repeller in its neighbourhood in $\R$, but merely that it is a maximal repeller in $\hI_X$). It relies on an analogon, adapted to $C^{1+\e}$ maps of interval, of \cite[Theorem 11.2.3, Corollary 11.2.4]{PU}, saying that for almost every backward trajectory $(x_n, n=0,-1,...)$ in the Rokhlin's natural extension $(\tf,\tmu)$  of $(f,\mu)$  (the topological inverse limit is denoted by $\tf$), there exists a ball (here: interval) $B(x_0,\delta)$ on which all backward branches $f^{-n}, n=0,1,...$ mapping $x_0$ to $x_{-n}$  exist and distortions are uniformly bounded (see \cite[Lemma 6.2.2]{PU}).

\smallskip

Here is a more detailed explanation: The integrability of
$$
\Psi(x):=\log \dist (x,\Crit(f))
$$
with respect to $\mu$ can be  used,
which  follows from the integrability of
$\log |f'|$ and the inequality $|f'(x)|\le  |x-c|^\e$ for every critical point $c$ of $f$.

Koebe distortion theorem referred to in \cite{PU} in
the proof of Theorem 11.2.3 is not needed. Exponential shrinking of diameters
of pull-backs is enough to guarantee bounded distortion for a map of class
$C^{1+\e}$ is.

In Pesin theory, going backward, unstable manifolds shrink exponentially,
so subexponential changes of Lyapunov coordinates do not influence them.
For the latter one could use the property that the distance of a typical trajectory
from $\Crit(f)$ shrinks at fastest subexponentially, what follows from the integrability of $\Psi(x)$.

\smallskip

Notice that (except in the last paragraph of the statement of the Theorem) we allow $\mu$ to be supported on an individual periodic orbit $O(p)$.
Then the proof of the assertion of Theorem \ref{Katok} is immediate: set $X_k:=O(p)$.
The proof in \cite{PU} covers in fact this case. There is then only one branch
of $f^{-m}$, for $m$ being a period of $p$, constituting the iterated
function system, giving a one-point limit set.

\smallskip

In view of the Ergodic Decomposition Theorem, \cite[Theorem 2.8.11]{PU}, the measures $\mu_k$ can be assumed ergodic.

\

Let us prove now the last assertion of the theorem

\smallskip

1. If we assume that $X$ is weakly isolated, then all periodic orbits in $X_k$ are contained in $X$ for $k$ large enough. Since they are dense in $X_k$ (see e.g. \cite{Bowen} or \cite[Theorem 4.3.12]{PU}), we get $X_k\subset X$.

\

2. Consider the case where we do not assume $X$ is weakly isolated.  So suppose that $(f,X,\hI_X,{\bf{U}})\in\sA_+$ and $h_\mu(f)>0$ (in particular $\mu$ is infinite). To find $X_k\subset X$ we complement an argument in the proof in \cite[Theorem 11.6.1]{PU}.
Recall that in \cite{PU} one finds a disc $B=B(x,\beta)$ and a family  of inverse branches
$f^{-m}_\nu$
on $B$ mapping it to $B(x,\frac56 \beta)$. Denoting the family of points $y_\nu =f^{-m}_\nu(x)$ by $D_m$ one has there
\begin{equation}\label{half-variational}
\sum_{y\in D_m} S_m\phi (y) \ge \exp (m(h_\mu(f) + \int \phi \, d\mu - \theta)),
\end{equation}
for an arbitrary $\theta>0$ and $m\ge m_0$ large enough, which easily yields the  rest of the proof.

Here we specify better $x$ and $\beta$ and restrict the family $f^{-m}_\nu$ to omit $\partial(\hI_X)$ by our pull-backs.
To this end we repeat briefly the construction, using a method from \cite{FLM}.

We consider a set $Y$ of $\tf$-trajectories in the inverse limit $\tX$
so that $\tmu(Y)>1/2$,
all pull-backs for $f^n$ exist along the trajectory $(x_n)\in Y$ on $B(x_0,\delta)$ for a constant $\delta>0$ and have uniformly bounded distortion, that is $|(f_\nu^{-n})'(z)/f_\nu^{-n})'(z')|\le \Const$ for all $z,z'\in B(x_0,\delta)$ for all branches $f_\nu^{-n}$ corresponding to the trajectories in $Y$. We assume also that
$$
|S_m \phi(y)-m\int \phi\, d\mu|<m \theta
$$
for $y=x_{-m}$ and every $m\ge m_0$, and moreover
\begin{equation}\label{Birkhoff}
|S_m \log \Jac_\mu (f)(y) - m h_\mu(f)|< m \theta;
\end{equation}
note that $\int \log \Jac_\mu(f)\,d\mu=h_\mu(f)$ by Rokhlin's formula, see i.e. \cite[Theorem 2.9.7]{PU}.  $\Jac$ is Jacobian in the weak sense, see \cite{PU}, compare also
Appendix \ref{s:AppendixB}.

By considering sufficiently many intervals of the form $B(x,\delta/2)$ for $x\in X$ in the support of $\mu$, covering $X$ with multiplicity at most 2 for $\delta$ appropriately small,  we find
$x$ such that
given an arbitrary integer $n_0>0$
\begin{equation}\label{disjoint}
\bigcup_{j=0}^{n_0}f^j (\partial(\hI_X))\,\, \hbox{ is disjoint from closure of}\,\, B(x,\delta/2).
\end{equation}
and $\tmu(Y_x)=C>0$  for $Y_x:=Y\cap \Pi^{-1}(B(x,\delta/2)\cap X)$ where $\Pi$ denotes the standard projection $(x_k)\mapsto x_0$ from the inverse limit to $X$.
Notice that by definition $x_0\in B(x,\delta/2)$, hence $B(x_0,\delta)$ being the domain of
$f^{-m}_{x_{-m}}$ contains $B(x,\delta/2)$, which is therefore a common domain for all
$f^{-m}_{x_{-m}}$ for all $(x_n)\in Y_x$.

By (\ref{Birkhoff}) by $h_\mu(f)>0$ for $\theta$ small enough for each
$(x_n)\in Y_x$ we have
\begin{equation}\label{Birkhoff2}
\mu(f^{-m}_{x_{-m}}(B(x,\delta/2)))\le \exp (-m (h_\mu(f)-\theta))
\end{equation}
for $m\ge m_0$, where the subscript $x_{-m}$ means we consider the branch mapping $x$ to $x_{-m}$.
When we consider pull-backs of $B(x,\delta/2)$ along trajectories belonging to $Y_x$ we remove each time the pull-backs whose closures intersect $\partial(\hI_X)$. Thus we remove each time
pull-backs for $f$ of all pull-backs already removed and additionally at most finite $\#\partial(\hI_X)$ number of pull-backs, so sets of measure $\tmu$ shrinking exponentially to 0 by (\ref{Birkhoff2}), provided $\theta<h_\mu(f)$.
If $n_0+1$, the time from which removing can start by  (\ref{disjoint}), is large enough the measure of the removed part of $Y_x$ is less than, say, $\hmu(Y_x)/2$.

Denote the non-removed part by $Y'_x$. Hence $\tmu(Y'_x)>C/2>0$. By (\ref{Birkhoff2}) the number of corresponding branches (pull-backs for $f^m$) is for each $m$ at least $\Const\exp m (h_\mu(f)-\theta)$. Each branch (except a finite number not depending on $m$) can be continued by a finite time (also not depending on $m$) to yield a pull-back of $B(x,\delta/2)\subset B(x_0,\delta)$, to be contained in $B(x, {1\over 3}\delta)$ for some $m'$ (greater from $m$ by a constant).

We conclude with (\ref{half-variational}) (with different $D_m$,
$\theta$ replaced by $2\theta$,
and $m'$ in place of $m$).
By construction closures of all pull-backs of $B(x,\delta)$ for $f^{m'}$  considered above
and their $f^j$-images for $j=0,1,...,m'$ are disjoint from $\partial(\hI_X)$, so since they intersect $X$ they are in $\hI_X$ and even in the interior of $\hI_X$.

Therefore the limit set $X_k$ of the constructed Iterated Function System (IFS) is contained in $\hI_X$. By its forward invariance and
the maximality of $X$ in $\hI_X$, it is contained in $X$. In fact by construction
$X_k\subset \interior \hI_X$.

\end{proof}

Now we shall prove Theorem B, including Proposition \ref{tree}, except the equalities for $P_{\Per}$, in a sequence of lemmas.

\begin{lemm}\label{Phypvar}
Let $(f,K)$ be a generalized multimodal map. Then,
for every  $t  \in  (- \infty, t_+)$  we
have  $P_{\hyp}(K, t) = P_{\varhyp}(K, t)$. For $t\ge t+$ the inequality $\le$ holds and if in addition (wi) is assumed, then the equality holds.

\end{lemm}

\begin{proof}
The inequality $P_{\hyp}(K,t)\le P_{\varhyp}(K,t)$ holds by the variational principle on each hyperbolic isolated subset of  $K$, provided such subsets exist.

The opposite inequality follows from Theorem \ref{Katok} applied to $\mu$ on $X=K$ with $\chi_\mu(f)>0$,  such that
$h_\mu(f|_K)-t\chi_\mu(f)$ is almost equal to $P_{\varhyp}(K,t)$, and to $\phi\equiv 0$.
Then, for all $k$, for $\mu_k$ being
the equilibrium on $X_k$ for the potential 0 (i.e. measure with maximal entropy),
we obtain
$$
\liminf_{k\to\infty} h_{\mu_k} (f|_{X_k})= \liminf_{k\to\infty} h_{\ttop} (f|_{X_k})\ge  h_{\mu} (f|_{K}).
$$

By Theorem \ref{Katok} we have also $\chi_\mu(f)=\lim_{k\to\infty}\chi_{\mu_k}(f)$.
Therefore one finds $X_k$ with $P(f|_{X_k},-t\log|f'|)$ at least $P_{\varhyp}(K,t)$ up to an arbitrarily small positive number.

Notice finally that $X_k\subset K$. This also follows from Theorem \ref{Katok} if
$h_\mu(f|_K)>0$. The only case it is not guaranteed by definition is where the function  $\tau\to P_{\varhyp}(K,\tau)$ is linear for $\tau\ge t$.

Then however $ t_+ \le t$.
Indeed, in this linear case

\noindent $P_{\varhyp}(K,t)=-t \inf_{\mu\in \sM^+(f,K)} \chi_\mu(f)$. Denote the first $\tau$ where $P_{\varhyp}(K,\tau)$ is linear on $[\tau,\infty)$, by $t'_+$.
For $\tau<t_+$ the function $P(\tau)=P_{\var}(\tau)$ is non-linear, in particular strictly decreasing.
Hence $P_{\varhyp}(K,\tau)=P(\tau)$ there, so $P_{\varhyp}(K,\tau)$ is non-linear. Hence
$t_+\le t'_+$ and in the case we consider $t'_+\le t$ has been assumed. Hence $t_+\le t$. This ends the proof.
\end{proof}

\begin{lemm}\label{hyptop}
For $(f,K)\in \sA_+$ there exists a hyperbolic isolated $f$-invariant set $X\subset K\cap \interior \hI_K$ with
$h_{\ttop}(f|_X)>0$.
\end{lemm}

\begin{proof}
By Variational Principle \cite{Walters} and $h_{\ttop}(f|_K)>0$ there exists an $f$-invariant  probability measure $\mu$ on $K$ with $h_\mu(f)>0$.  Hence, by Theorem~\ref{Katok} applied to
$\phi\equiv 0$ (as in Proof of Lemma~\ref{Phypvar}) there exists
a hyperbolic isolated $f$-invariant set $X_k\subset K\cap \interior \hI_K$ with
$h_{\ttop}(f|_{X_k})>0$.
\end{proof}

\begin{lemm}\label{Phyp} For $(f,K)\in \sA_+^{\BD}$ there exists $z_0\in K$ which is safe, safe forward  and expanding.
For each such $z_0$ we have $P_{\tree}(K,z_0,t) \le P_{\hyp}(K,t)$ for all $t\in\R$. In fact for $t\le 0$ this estimate holds for all $z_0\in K$.

On the other hand for each $(f,K)\in \sA_+$ (there is no need to assume BD),
$P_{\tree}(K,z_0,t)\ge P_{\hyp}(K,t)$ for all $z_0\in K$ for $t\ge 0$, and for all
$z_0\in K$ for which there exists a backward not periodic trajectory in $K$ omitting critical points (in particular for all $z_0$ safe)  for $t<0$.
\end{lemm}
Notice that if  every backward trajectory of $z_0$ in $K$ meets a critical point then for $t<0$, \, $P_{\tree}(K,z_0,t)=-\infty$.

\begin{proof}
Due to  Lemma \ref{hyptop} there exists an $f$-invariant hyperbolic isolated set $X\subset K \cap \interior \hI_K$ with $\mu$ supported on $X$ such that
$h_\mu(f|_X)>0$ (take just measure of maximal entropy for $f|_X$, existing due to hyperbolicity).  By hyperbolicity of $f|_X$, \, $\chi_\mu(f|_X)>0$. Hence Hausdorff dimension
$HD(X)\ge h_\mu(f|_X)/\chi_\mu(f|_X)>0$. Choosing $\delta$ arbitrarily small in Definition \ref{safe},
we see that the set of points which are not safe has Hausdorff dimension equal to 0. Hence there exists
a positive Hausdorff dimension set of safe points $z_0$, in $X\subset \interior \hI_K$ hence expanding and safe forward.

The proof that
$P_{\tree}(K,z_0, t)\le P_{\hyp}(K,t)$ for $t > 0$ is the same as in the complex case, see e.g.
\cite[Theorem 12.5.11]{PU} or \cite{PR-LS2}. Briefly: we do not capture neither critical points nor end points of $I$ in the pull-backs for $f^j$
of $B=B(z_0, \exp (-\alpha n))$ for
an arbitrarily small $\alpha>0$, for the time (order of the pull-back) $j\le 2n$.

We "close the loop"
in a finite bounded number of backward steps after backward time $n$, and the forward time $n_1$ of order not exceeding
$\alpha n \log \lambda$, where $\lambda=\lambda_z$, see Definition~\ref{expanding}.

More precisely: $f^{n_1}(B)$ is an interval of length of order $\Delta$ not depending on $n$.
By Remark~\ref{wexact2} there exists $k(\Delta)$ such that for every $z_1\in f^{-n}(z_0)\cap K$ there exists $m\le k(\Delta)$ and $z_2\in {\frac 12} f^{n_1}(B)\cap K$ such that $f^m(z_2)=z_1$. By Lemma~\ref{shrinking} the closure of the pull-back of $B$ for $f^{n+m}$ for $n$ large enough,  containing $z_2$,  is contained in $f^{n_1}(B)$. Using all $z_1$'s we obtain an Iterated Function System for backward branches of $f^{n+n_1+m}$.

We construct an isolated (Cantor) hyperbolic set $X$ as the limit set of the arising IFS,
compare Proof of Theorem~\ref{Katok}. (Notice however that here to know the limit set $X$ is hyperbolic, we need to use bounded distortion, unlike in Theorem~\ref{Katok}.)
Since all pull-backs of $B$ for $f^j, j=0,1,...,n+m$ in the "loops" intersect $K$ and are disjoint from $\partial\hI_K$ since
$z$ is "safe" for the time $2n\ge n+m$ and since $f^j(B)\subset \hI_K$ for $j=0,1,...,n_1$ (for corrected $\Delta$) by
$\dist(f^j(z),\partial\hI_K)$ bounded away from 0, see Definition~\ref{safe forward}, and by backward Lyapunov stability,
see Lemma~\ref{shrinking} and Remark~\ref{back stab}, we conclude that $X\subset\hI_K$.
Hence by its invariance and maximality, $X\subset K$. Notice that we need not assume (wi).

The set $X$ is $F:=f^{n+n_1+m}$-invariant. To get an $f$-invariant set one considers $\bigcup_{j=0}^{n+n_1+m}f^j(X)$ and yet extends it, to conclude with an $f$-invariant isolated set.

To get the asserted inequality between pressures, ones uses for each $z_1$ as above, the estimate,
see \cite[(2.1)]{PR-LS2}:
$$
|(f^n)'(z_1)|^{-t}\le C |(f^{n+n_1+m})'(x)|^{-t} L^{t(n_1+m)},
$$
where $L:=\sup |f'|$ and $x$ is an arbitrary point in the pull-back of $B$ for $f^{n+m}$, containing $z_2$.
To prove this inequality between derivatives we use bounded distortion assumption, the constant $C$ results from it.

\smallskip

Notice however that in this estimate of derivatives we use $t\ge 0$, .

\smallskip

For $t\le 0$ the inequality $P_{\tree}(K,t)\le P_{\hyp}(K,t)$ needs a separate explanation. Two proofs (in the complex case) can be found in \cite{PR-LS2}. An indirect one, in Theorem A.4, and a direct one,  in Section A3. Notice that for $t\le 0$ the function $-t\log |f'|$ is continuous, though
logarithm of it attains $-\infty$ at critical points, and the standard definition of the topological pressure makes sense. The proofs in the interval case are the same. In the indirect proof
the inequality \cite[(A.1)]{PR-LS2} $P_{\tree}(K,z_0,t)\le P(f|_K, -t\log|f'|)$ holds for all $z_0\in K$. For the existence of large subtrees having well separated branches, used in the proof, see \cite[Lemma 4]{P-Perron}. One uses the property that $f$ is $C^1$. The direct proof in \cite[Section A3]{PR-LS2} is a refinement of the proof for $t\ge 0$ (see above) and applicable only for $z_0$ safe and expanding.

\

The opposite inequality for each $z_0\in K$ (with an exception mentioned in the statement of Lemma), follows from the fact that, due to the strong transitivity property of $f$ on $K$, see Proposition~\ref{dp},
for an arbitrary isolated invariant hyperbolic set $X\subset K$,
we can
find  a backward trajectory $z_0,z_{-1},...,z_{m}$ where $m\le 0$ \, ($f(z_k)=z_{k+1}$)
such that
$z_{m}\in K$ is arbitrarily close to $X$, say close to a point $z'\in X$.

Then for each backward trajectory $z'_{n}, n=0,-1,...$ of $z'$ in $X$ there is exactly one backward trajectory $z_{m+n)}$ of $z_{m}$ so that all $|z'_n-z_{m+n}|$ are small and decrease exponentially to 0. Hence, for  $n\ge 0$
\begin{equation}\label{shadow}
\sum_{f^n(x)=z',\, x\in X} |(f^n)'(x)|^{-t}\le \Const \sum_{f^n(x)=z_m, x\in K} |(f^n)'(x)|^{-t},
\end{equation}
hence,
$P_{\tree}(K,z_m,t)\ge P_{\tree}(X,z',t)=P(X,t)$.
Finally $P_{\tree}(K,z_0,t) \ge
\lim_{n\to\infty} {-t\log |(f^{|m|})'(z_m)|
\over ||m|+n|}+P_{\tree}(K,z_m,t)\ge P_{\tree}(K,z_m,t)$
for $t\ge 0$ since then the first summand is larger or equal to 0 (0 or $\infty$) and for $t<0$ since then
the first summand is equal to 0 provided $(f^{|m|})'(z_m)\not= 0$.

\end{proof}

\begin{rema}
We have proved above that for all $z_0\in K$ safe and expanding $\limsup$ can be replaced by $\lim$
in the definition of $P_{\tree} (K,z_0,t)$, i.e. the limit exists, compare \cite[Remark 12.5.18]{PU}.
Indeed (\ref{shadow}) holds for all $n$ hence in the estimates which follow we can consider $\liminf$
in $P_{\tree}$.
\end{rema}

\

\begin{lemm}\label{Pvar} For each $(f,K)\in \sA_+$,\,
$P_{\varhyp}(K,t) = P_{\var}(K,t)$ for all $t<t_+$, and assuming (wi) for all $t\ge t_+$.
\end{lemm}
\begin{proof}
$P_{\varhyp}(K,t)\le P_{\var}(K,t)$ for all $t$ is obvious.

The opposite inequality is not trivial for $t\ge t_+$ and in the proof we shall apply Theorem C, proved in Section~\ref{TCE}.

Suppose there exist $\mu\in\sM(f,K)$  with $\chi_\mu(f)=0$ for which $h_\mu(f)-t\chi_\mu(f)$ are  arbitrarily close to $P_{\var}(K,t)$. This implies, due to
$h_\mu(f)\le 2\chi_\mu(f)=0$ (Ruelle's inequality), that $P_{\var}(K,t)=0$ and that $t\ge t_+$.  By $\chi_\mu(f)=0$,\, $(f,K)$ is not Lyapunov hyperbolic, hence not UHPR, see Theorem C, in particular Corollary~\ref{chi0}.
So there exist repelling periodic points $p\in K$ with $\chi(p)$ arbitrarily close to 0.
Thus $P_{\varhyp}(K,t)$ is arbitrarily close to 0, hence equal to 0, as well as $P_{\var}(K,t)$.

\end{proof}

\

We end this Section with commenting on various definitions of $\chiinf(f,K)$ and $\chisup(f,K)$ and therefore the {\it condensation} and {\it freezing} phase transition points $t_-$ and $t_+$, see Introduction.

\begin{prop}\label{asymptotes} (compare \cite[Proposition 2.3]{PR-L2}, and \cite[Main Theorem]{R-L}).

 For each $(f,K)\in \sA_+^{\BD}$ 
 the following holds
\begin{enumerate}
\item[1.]
Given a repelling periodic point~$p$ of~$f$, let~$m$ be its period and put $\chi(p) \= \tfrac{1}{m} \log |((f^m)'(p)|$. Assume (wi), see Definition~\ref{wisolated}.
Then we have
$$ {\chiinf}_{\Per}:=\inf \{ \chi(p): p\in K \text{ is a repelling periodic point of } f \}
=
\chiinf,$$
$$ {\chisup}_{\Per}:=\sup \{ \chi(p): p\in K \text{ is a repelling periodic point of } f \}
=
\chisup.$$
\medskip
\item[2.]
$$ \lim_{n \to + \infty} \tfrac{1}{n} \log \sup \{ |(f^n)'(x)|: x \in K \}
=
\chisup. $$
\item[3.]
For each $z \in K$ safe, safe forward and hyperbolic, we have

\begin{equation}\label{e:individual minimal pressure}
{\chiinf}_{\rm{Back}}(z):=\lim_{n \to + \infty} \tfrac{1}{n} \log
\min \{ |(f^n)'(w)|: w \in f^{-n}(z) \}
=
\chiinf,
\end{equation}
assuming (wi) for the $\le$ inequality in case $\chi_{\inf}=0$, and finally
\begin{equation}\label{e:individual maximal pressure}
{\chisup}_{\rm{Back}}(z):=\lim_{n \to + \infty} \tfrac{1}{n} \log
\max \{ |(f^n)'(w)|: w \in f^{-n}(z) \}
=
\chisup.
\end{equation}
\end{enumerate}
\end{prop}

\begin{proof}

\partn{1} The inequalities ${\chiinf}_{\Per}\ge \chiinf$ and
${\chisup}_{\Per}\le \chisup$ follows immediately   from the
use
of the measures equidistributed on periodic orbits involved in ${\chiinf}_{\Per}$ and ${\chisup}_{\Per}$.

The inequality
${\chisup}_{\Per}\ge \chisup$ follows from Katok's construction of periodic orbits.
Formally, we can refer to Theorem~\ref{Katok}. We can assume $\chisup>0$ since otherwise the inequality is obvious. Then we choose $\mu$ ergodic with $\chi_\mu>0$ arbitrarily close to $\chisup$.  For any $\phi$  we find sets $X_k$, periodic orbits $O(y_k)\subset X_k$, equidistribute measures $\mu_k$ on $O(y_k)$ and conclude with $\chi(y_k)\to \chi_\mu$ as $k\to\infty$. Finally, $y_k\in K$ by (wi). (We do not know whether we can find $\mu$ of positive entropy above, to avoid using (wi).)

It is more difficult to prove ${\chiinf}_{\Per}\le \chiinf$ in case we do not know {\it a priori}
that the latter number is the limit of a sequence $\chi_{\mu_n}$ for $\mu_n$ hyperbolic measures on $K$, i.e. if there exists a probability invariant measure on $K$ with $\chi_\mu =0$. Then the proof of the inequality immediately follows from
Theorem C, which in particular says that if such $\mu$ exists then UHPR fails, that is there exists a sequence of repelling periodic points $p_n\in K$ such that $\chi(p_n)\to 0$,
Corollary~\ref{chi0}.
Compare also Proof of Lemma~\ref{Pvar}. The assumed property (wi) is used as well.

\partn{2} The inequality $\ge$ follows from Birkhoff Ergodic Theorem. The opposite inequality
can be proved via constructing $\mu$ as a weak* limit of the measures
$ \mu_n \= \tfrac{1}{n} \sum_{j = 0}^{n - 1} \delta_{f^j(z_n)}$ for $z_n$ involved in the supremum.

\partn{3a} The inequality ${\chisup}_{\text{Back}}(z)\le \chisup$ follows immediately from the similar inequality in the previous item. It holds for every $z\in K$.

\partn{3b} The inequality ${\chisup}_{\text{Back}}(z)\ge {\chisup}_{\Per}$ (or $\chisup$)
can be proved by
choosing a backward trajectory of $z$ converging to a hyperbolic repelling periodic trajectory $O(p)$ in $K$ with $\chi(p)$ arbitrarily close to ${\chisup}_{\Per}$. Similarly we choose a backward trajectory approaching to a Birkhoff-Pesin backward trajectory
 for a measure $\mu$ with $\chi_\mu$ close to $\chisup$.

 By Birkhoff-Pesin backward trajectory we mean $x_n, n=0,-1,...; f(x_{n-1})=x_n$ in $K$, such that $\lim_{n\to\infty}\log |(f^{|n|})'(x_n)|=\chi_\mu$ and the pullbacks of an interval $B(x_0,r)$ along it, do not contain critical points, compare the beginning of Proof of Theorem~\ref{Katok}. The existence of such trajectories uses $\chi_\mu>0$.

  We do not use
neither
$z$ is safe nor hyperbolic; we only use the assumption $z$ is not in a weakly $\Crit$-exceptional set.

\partn{3c} The inequality ${\chiinf}_{\text{Back}}(z)\ge {\chiinf}_{\Per}$ (or $\chi_{\inf}$)
can be proved by finding a periodic trajectory shadowing an arbitrary piece of backward trajectory of $z$, compare Proof of Lemma~\ref{Phyp} 

 Here we do use `safe and `hyperbolic'. Notice however that if  we assume that $(f,K)$ satisfies  the property (wi), then we do not need the `safe forward' condition on $z$.

\partn{3d} Finally ${\chiinf}_{\text{Back}}(z)\le {\chiinf}_{\Per}$ can be proved similarly to 3b. It works even if a periodic trajectory $O(p)$ is not hyperbolic repelling, by Remark~\ref{finite indifferent}, since it is then (one-sided) repelling, to the side from which it is accumulated by $K$. 

 The proof of 3d. in the version  ${\chiinf}_{\text{Back}}(z)\le \chiinf$ follows from  ${\chiinf}_{\Per}\le \chiinf$ in the item 1, which however uses (wi) if $\chi_{\inf}=0$.

\end{proof}

\

\section{Pressure on periodic orbits}\label{pressure-per}

This Section complements Section 4. It is not needed for the proof of Theorem A.

We start with the interval version of a technical fact allowing to estimate the number of "bad" backward trajectories, used in the complex case (in various variants) in several papers on geometric pressure, e.g. \cite{PR-LS1}, \cite{PR-LS2}, \cite{PR-L1}, \cite{PR-L2}, \cite{GPR}, \cite{GPR2}.
(The essence of this fact is just calculating vertices of a graph.)

\begin{defi}\label{procedure}
Let $(f,K)\in \sA_+$.
Fix $n$ and arbitrary $x_0\in K$ 
and $R>0$. For every
backward trajectory of $x_0$ in $K$, namely a
sequence of points $(x_i\in K, i=0,1,...,n)$ such that $f(x_i)=x_{i-1}$ run
the following procedure. Take the smallest $k=k_1\ge 0$ such that
$\Comp_{x_{k_1}} f^{-k_1}B(x_0,R)$ contains a critical point. Next
let $k_2$ be the smallest $k>k_1$ such that
$\Comp_{x_{k_2}} f^{-(k_2-k_1)}B(x_{k_1},R)$ contains a critical point.
Etc. until $k=n$. Let the largest $k_j\le n$ for the sequence $(x_i)$
be denoted by $k((x_i))$ and let the set $\{y: y=x_{k((x_i))}$ for a backward
trajectory $(x_i)\}$ be denoted by $N(x_0)=N(x_0,n,R)$.
\end{defi}

\begin{lemm}\label{nodes} (Compare e.g. complex \cite[Lemma 3.7]{PR-LS2}, \cite[Lemma 3.6]{PR-L1}, or the interval setting: \cite[Proposition 3.19]{GPR2}).

For every $(f,K)\in \sA_+$, for every $\e>0$ and for all $R>0$ small enough and $n$ sufficiently large, for every $x\in K\setminus B(\Indiff(f), R)$ 
it holds
$\# N(x,n,R)\le \exp (\e n)$.
\end{lemm}
One can replace in this Lemma critical points (in Definition \ref{procedure}) by an arbitrary finite set not containing periodic orbits, in particular by $S'(f,K)$. A variant of this will be used in Section \ref{s:Key}.
\

\begin{proof}[Proof of Theorem B for $P_{\Per}$]
The inequality $P(K,t)\le P_{\Per}(K,t)$ follows immediately from $P(K,t)=P_{\hyp}(K,t)$.
Indeed, for every isolated hyperbolic $X\subset K$ we have Bowen's formula $P(X,t) = P_{\Per}(X,t)$.

For the opposite inequality we can adapt \cite[proof of Theorem C]{BMS} for complex polynomials with
connected Julia set, or
\cite{PR-LS2} for general rational functions. In the latter, a condition H was assumed, saying that for every $\delta>0$ and $n$ large enough, for any set $P=P_n$ of periodic points
of period $n$ such that for all $p,q\in P$ and all $i:0\le i\le n$
$\dist(f^i(p),f^i(q))<\exp(-\delta n)$, we have $\#P<\exp (\delta n)$.
In \cite[Lemma 2]{BMS} this assumption, even a stronger one, has been proved provided
all periodic points are hyperbolic.

In our  $C^2$ generalized multimodal case the \cite{BMS} condition also holds. Namely

\

(H*) For every
$(f,K)\in \sA$,
for every $\e >0$, there is $\rho>0$ such that for $n$ large enough, for any $p,q\in P\subset K$ as above, satisfying $|f^i(p)-f^i(q)|\le \rho$, we have $\#(P) < \exp (\e  n)$.

\

Indeed, having $P=P_n$ not satisfying (H*), we can shrink $\e $ and assume there is no critical point for $f^n$ in $T_n$ being the convex hull of $P_n$, and $|T_n|\to 0$ as $n\to\infty$.

Indeed, consider $f^i(T_n^0)$ where $T_n^0$ is the convex hull of $P_n$, for $i=0,...$ as long as
$f^{i_0}(T_n^0)$ contains a critical point $c_0$ for $f$ (there is only one such critical point, provided $\rho$ is small enough). $f^{i_0}$ is strictly monotone on $T_n^0$. Then the point $(f|_{T_n^0})^{-i_0}(c_0)$
divides $T_n^0$ into two subintervals and we define  $T_n^{i_0}$ as the convex hull of
the points belonging to $P$ in this one
which contains the larger
number of points belonging to $P$. Next we consider $i_1>i_0$ the first $i>i_0$ such that
$f^{i_1}(T_n^1)$ contains a critical point $c_1$. After a consecutive division we continue. We stop with
$T_n:=T_n^{i_s}$ such that $f^{n-i_s}$ does not contain critical points in $T_n^{i_s}$.
Notice that all $i_{t+1}-i_t$ are larger than an arbitrarily large constant for $\rho$ small enough.

Thus there are $n$ arbitrarily large such that
there are in $T_n$ an abundance (the number exponentially tending to $\infty$ as $n\to\infty$) of fixed points for $f^n$, in particular an abundance of attracting or indifferent periodic orbits for $f$, all in $\hI_K$ provided $\delta$ is smaller than the gaps between the components of $\hI_K$, hence in $K$. This contradicts lack of attracting periodic orbits and the finiteness of the
set of indifferent periodic orbits in $K$, see Remark~\ref{finite indifferent}.

\

As in \cite{PR-LS2} and \cite{BMS} we split $P_{\Per}$ in two parts.

\begin{defi}\label{regular}
Fix small $r>0$.
We say that a periodic orbit $O\subset K$ of  period $n \ge 1$ is
{\it regular} (more precisely: regular with respect to $r$) if
it is hyperbolic repelling and
there exists $p \in O$ such that
$f^n$ is injective on $\Comp_p f^{-n}(B(p, r))$.
If $O$ is not regular, then we say that $O$ is {\it singular}.
\end{defi}

We denote by $\Per_n$ the set of all periodic points of period $n$ in $K$ and denote
by $\Per_n^{\reg}$ the set of all points  in $\Per_n \setminus \Indiff(f)$, whose periodic orbits  are regular. Denote by $\Per_n^{\sing}$ the set of all other periodic points in $K$.

We shall first prove that
$$
P_{\Per}^{\reg}(K,t) := \limsup_{n \to \infty} {1\over n}\log
\sum_{p \in \Per_n^{\reg}} |(f^n)'(p)|^{-t} \le P_{\tree}(K,t).
$$

We choose a finite $r/12$-dense set $A\subset K$ of points, safe (for an arbitrary $\delta$) and hyperbolic  with common constants.

Let $O \subset \Per_n^{\reg}$ be a regular periodic
orbit and let $p_0 \in O$ be such that $f^n$ is injective on
$\Comp_{p_0} f^{-n}(B(p_0,r))$. Let $x\in A\cap B(p_0,r/12)$.
Consider the (unique) backward orbit $(x_j, j=0,...,n)$ such that

\begin{equation}\label{associate}
x_j \in \Comp_{f^{n- j}(p_0)} f^{-j}(B(p_0, r/12)).
\end{equation}

By BD,
if $r$ is small enough, then
$$
C^{-1}|(f^n)'(x_n)|\le |(f^n)'(p_0)|\le C|(f^n)'(x_n)|
$$
for $C:=C(11/12)$.

Notice also that
$$
\Comp_{x_n} f^{-n}B(x_0,r/2)\subset
\Comp_{p_0} f^{-n}B(p_0,2r/3)
$$
$$
\subset
B(p_0,r/4)\subset
B(x_0,r/3),
$$
the middle inclusion by Lemma~\ref{shrinking} for $n$ large enough.
To apply this lemma, if $p_0$ is too close to $\Indiff(f)$, we replace $p_0$ by its image
under an iterate of $f$ which is far from $\Indiff(f)$. This is possible since $\Indiff(f)$ is finite.
If $p_0$ is close to a indifferent periodic point $p'$, it must be on the repelling side of it, so its forward $f$-trajectory escapes from a neighbourhood of $O(p')$.

So there is exactly one $q\in \Comp_{x_n} f^{-n}B(x_0,r/2)$ such that $f^n(q)=q$.
Hence each $(x_i, i=0,...,n)$ above is associated to at most one point in
$P_{\Per}^{\reg}(K,t)$.

We conclude with
$$
\sum_{ p \in \Per_n^r }
|(f^n)'(p)|^{-t}\le
\max\{C^t,C^{-t}\} \sum_{x_n\in f^{-n}(x_0), x_0\in A} |(f^n)'(z_n)|^{-t},
$$
hence
$$
\sum_{ p \in \Per_n^{\reg} }
|(f^n)'(p)|^{-t}\le \max\{C^t,C^{-t}\} \#A \sup_{x_0\in A}
 \sum_{x_n\in f^{-n}(x_0)}  |(f^n)'(x_n)|^{-t}.
$$
hence
$$
P_{\Per}^{\reg}(K,t) \le  P_{\tree}(K,t).
$$

The next step is to find an upper bound, depending on $n$, for the number of singular periodic orbits, not indifferent. We follow \cite{PR-LS2}, but there is no need to replace $p$ by another point in $O(p)$ except to reach $p$ not too close to $\Indiff(f)$.
Given $\rho>0$, for $r$ small enough all components of $f^{-n}(B(p,r))$ have diameters smaller than
$\rho$ by Lemma~\ref{shrinking}.
For $x_j$ chosen as above and $p_j=f^{n-j}(p), j=0,...,n$ we get $|x_j-p_j|<\rho/2$, hence by (H*) the the number of singular orbits to which the same $(x_0,...,x_n)$ is assigned is bounded by $\exp (\e   n)$.

Now assume $R \gg \rho,r$ and apply Lemma~\ref{nodes} and preceding notation. We attempt to bound the number $S_n$ of backward trajectories $(x_0,...,x_n)$ associated to some points in $\Per^{\sing}_n$.

Here given singular $O(p)$ of period $n$ we associate with it a backward trajectory $(x_0,...,x_n)$ as in the regular case, by taking $x_0\in A$ and $x_j$ satisfying (\ref{associate}). Here however even chosen $x_0$ we do not have uniqueness of $x_j$.

Denote by $X_{k,x_k}$ the set of all backward trajectories $(x_i, i=0,1,...,n)$ with the same fixed $k=k((x_i)_{i=0,...,n})$ and $x_k$, associated with
singular periodic orbits i.e. belonging to $\Per_n^{\sing}$.
Repeating the proof in \cite[Step 4.1]{PR-LS2} we conclude that $\# X_{k, x_k}\le k\#\Crit(f)$.

(The proof in \cite{BMS} is more elegant at this place. The authors proceed in Definition
\ref{procedure} till $k=2n$, thus incorporating our consideration to bound $\# X_{k, x_k}$
in the bound of $N(x_0,2n,R)$.)

Thus,
$$
S_n\le \#(A) \# \{N(x_0,n,R): x_0\in A\} \sup_{k,x_k}\# X_{k, x_k} \le \#A \exp (\e  n) n\#\Crit(f).
$$
and in consequence
$$
\#\Per^{\sing}_n\le \exp (\e  n) \#A \exp (\e  n) n\#\Crit(f)  \le \exp (3\e  n)
$$
for $n$ large enough (this includes also all indifferent periodic orbits in $K$).

Since for each periodic $p\in K$ and $t\in R$ we have for the probability invariant measure $\mu_p$ on $O(p)$,
$$-t\chi(p) = h_{\mu_p}(f) -t\chi_{\mu_p} \le P_{\var}(K,t),
$$
we obtain
$$
P_{\Per}^{\sing}(K, t): = \limsup_{n \to \infty} {1\over n}\log
\sum_{p \in \Per_n^s} |(f^n)'(p)|^{-t} =
\limsup_{n \to \infty} {1\over n}\log \sum_{p \in \Per_n^s} \exp (-tn\chi(p))
$$
$$\le
\lim_{n\to\infty} {1\over n}\log (\exp (3\e  n) \exp (n P_{\var}(K,t))
\le 3\e  + P_{\var}(K,t).
$$
Finally we obtain
$$
P_{\Per}(K, t)\le \max\{P_{\Per}^{\reg}(K, t), P_{\Per}^{\sing}(K, t) \} \le P(K,t)+3\e ,
$$
hence considering $\e >0$ arbitrarily small and respective $r$ we obtain
$P_{\Per}(K, t)\le P(K,t)$.
\end{proof}

\

\

\

\section{Nice inducing schemes}\label{s:nice}

\subsection{Nice sets and couples}

\

Firstly we adapt to the generalized multimodal interval case the definitions, and review some properties, of nice sets and couples. We follow the complex setting, see e.g. \cite[Section 3]{PR-L2} and related results in \cite{CaiLi}. For the interval case
for the notion of nice sets see e.g. \cite{BRSS}; the closely related concept of nice intervals was introduced and used to consider return maps to them, much earlier (see papers by Marco Martens, e.g. \cite{Martens}).

\begin{defi}\label{nice} Let $(f,K,\hI_K, {\bf{U}})\in \sA$.
We call a neighborhood~$V$ of the  \textit{restricted singular set}
$
S'(f,K)=\Crit(f)\cup \NO(f,K)
$
(see Definition~\ref{exceptional}) in $\bf{U}$ a \textit{nice set for}~$f$, if for every $n \ge 1$ we have
\begin{equation}\label{nice1}
f^n(\partial V) \cap V = \emptyset,
\end{equation}
 and if each connected component of~$V$ is an open interval
containing precisely one  point of~$S'(f,K)$.
The component of $V$ containing $c\in S'(f,K)$ will be denoted by $V^c$.

So, let $V = \bigcup_{c \in S'(f,K)} V^c$ be a nice set for~$f$.
 According to previously used names, Definition~\ref{pull-back}, we call any component $W$ of $f^{-n}(V)$ for $n\ge 0$ intersecting $K$,  a pull-back, or pull-back for $f^{n}$, of $V$. We  denote this $n$ by $m_W$.

 All $W$ as intersecting $K$ are far from $\partial{\bf{U}}$, that is deep in ${\bf{U}}$,  if all $V^c$ are small enough, see a remark at the end of Definition~\ref{pull-back} and Lemma~\ref{shrinking}.

 We have either
$$
W \cap V = \emptyset \, \text{ or } \, W \subset V.
$$

Furthermore, if~$W$ and~$W'$ are distinct pull-backs of~$V$, then we have either,
$$
W \cap W' = \emptyset,
\, W \subset W'
\, \text{ or } \,
W' \subset W.
$$
For a pull-back $W$ of $V$ we denote by $c(W)$ the point in $S'(f,K)\cap V^c$,
where $W$ is an $f^{-m_W}$ pull-back of $V^c$.

Moreover we put,
$$
\sK(V)
=
\{ z \in K:\text{ for every $n \ge 0$ we have $f^n(z) \not \in V$} \}.
$$
\end{defi}

\begin{lemm}\label{expanding away} For all $(f,K)\in \sA$ with all periodic orbits in $K$ hyperbolic repelling, the map
$f|_K$ is expanding away from singular points, that is
expanding on $\sK(V)$ for every nice $V$ defined above.

More generally for every open (in $\R$) set $V$
containing just $\Crit(f)$ there exists $\lambda>1$ and $N=N(V)>0$ such that if $x,f(x),...,f^N(x)$ belong to
${\bf{U}}\setminus V$ then $|(f^N)'(x)|>\lambda$ (maybe for a diminished $\bf(U))$).

\end{lemm}

\begin{proof}
This follows from Ma\~n\'e's Hyperbolicity Lemma, see e.g.
\cite[Ch.III, Lemma 5.1]{dMvS}.
Indeed we can extend $f|_{\hI_K}$ to $I\supset\hI_K$,\, $C^2$, as in Lemma~\ref{good extension}, see Remark~\ref{LBD2}, thus getting    a $C^2$ map with all periodic orbits hyperbolic repelling.

Then for $x\in K$ the assertion of the lemma is precisely the assertion of Ma\~n\'e's Lemma.

For  $x\notin K$ close to $K$ we find $y\in K$ close to $x$, with $f^j(y)\in K\setminus V'$ for $V'$ a neighbourhood of $\Crit(f)$ smaller than $V$,  $j=0,1,...,N(V')$, so that $f^j(x)$ is sufficiently close to $f^j(y)$ to yield $|(f^N)'(x)|>1=\frac{\lambda+1}2>1$. This finishes the proof.

(Notice that if we assume additionally bounded distortion,  Lemma~\ref{expanding away} follows easily from Lemma~\ref{shrinking}.)
\end{proof}

\begin{lemm}[Existence of nice sets]\label{existence of nice sets} For all $(f,K,{\bf{U}})\in \sA_+$ with no indifferent periodic orbits in $K$, there exist nice sets
with all components of arbitrarily small diameters.
\end{lemm}

\begin{proof}
Let $O(p)\subset K$ be an arbitrary periodic orbit, not in the forward orbit of a  point in $S'(f,K)$.
It exists by the existence of infinitely many periodic orbits, due to $(f,K)\in\sA_+$, see Proposition~\ref{density periodic}. Then we
use the density of $Q:=\bigcup_{n=0}^\infty Q_n$ for $Q_n:=\bigcup_{j=0}^n f^{-j}(O(p))$
in $K$. 
Notice that this density holds even for $Q\cap K$, see Proposition~\ref{dp}, but in the definitions of $Q_n$ and $Q$ we do not restrict $f$ to $K$.

Fix an arbitrary  $c\in S'(f,K)$.
If $c$ is an accumulation point of $Q$ from the left hand
side, then for all $n$ large enough $Q_n$ contains points on this side of $c$. Let  $a_{c,n}\in Q_n$ be the point closest to $c$ in $Q_n$ from this side. Similarly, if $c$ is an accumulation point of $Q$ from the right hand side, we define $a'_{c,n}\in Q_n$ to be the point closest to $c$ from the right hand side.
If $c$ is not an accumulation point of $Q$ from one side (it is then an accumulation point of $Q$ from the other side since otherwise $c$ is isolated in $K$), then there is an interval $T$ adjacent to $c$ from this side with  $Q\cap T=\emptyset$. Then $f^k(T)\cap K=\emptyset$ for all positive integer $k$ since otherwise for the first $k$ such that $f^k(T)\cap K\not=\emptyset$ the set $f^k(T)$ is open (since it cannot capture earlier a critical point as $\Crit(f)\subset K$), hence it intersects $Q$, hence $T$ intersects $Q$, a contradiction.

(Notice that the case critical $c$ is not an accumulation point of $Q$ from one side is possible only if $c\in \Crit^I(f)\setminus\NO(f,K)$, in particular $c$ is a critical inflection point. Indeed if $c$ is a turning point then both $a_{c,n}$ and $a'_{c,n}$ can be defined as $f$-preimages of the same point in $Q_{n-1}$ closest to $f(c)$ (there points arbitrarily close to $f(c)$ even in $K$ since the "fold"  is on the side of $K$; otherwise $f(c)$ would be isolated in $K$. Compare Proof of Theorem C, Case 2,  in Section~\ref{TCE}).)

Thus we can apply Lemma \ref{nonwandering}, and conclude that $c$ is eventually periodic
in a repelling periodic orbit $O(z)$, i.e. there exists $k>0$ such that $f^k(c)\in O(z)$ periodic.
Then we define the missing $a_{c,n}\in T$ or $a'_{c,n}\in T$ as an $f^k$-preimage of an arbitrary point in $f^k(T)$, arbitrarily close to $c$ (for simplification we set here $f(z)=z$).

We define $V^c:=(a_c,a'_c)$ for all $c\in S'(f,K)$, arbitrarily small for $n$ appropriately large.
The property (\ref{nice1}) follows easily from the definitions.


\end{proof}

Notice that $V^c$ correspond to critical Yoccoz puzzle pieces of $n$-th generation containing critical points in the complex polynomial setting.
The above proof yields the puzzle structure in the generalized multimodal real setting:

\begin{prop}\label{partition} If $(f,K,{\bf{U}})\in \sA_+$ with no indifferent periodic orbits in $K$, then there exists a partition of a neighbourhood of $K$ in $\R$ into
intervals $W^k$ such that
$\sup_k \diam W^k$ is arbitrarily small and for all $n>0$ and $k_1,k_2$ we have $f^n(\partial W^{k_1})\cap \interior W^{k_2}=\emptyset$.
\end{prop}

\begin{proof}
Consider as above the sets $Q_n:=\bigcup_{j=0}^n f^{-j}(p)$ and additionally
$\hat Q_n:=\bigcup_{j=1}^n (f|_{{\bf{U}}})^{-j}(\hat Q)$, where $\hat Q$
consists of a finite family of  points, each one being the end point of the open interval adjacent to a respective repelling periodic orbit in $K$ on the side disjoint from $K$,
in the boundary of ${\bf{U}}$,
compare the last paragraph in Proof of
Lemma~\ref{shrinking}. Now, for each $n$ large enough, we partition a neighbourhood of $K$ into intervals with ends in $Q_n\cup \hat Q_n$. For ${\bf{U}}$ such that $f(\hat Q)\cap {\bf{U}} =\emptyset$ the property \eqref{nice1} holds.
\end{proof}

One can think about $Q_n$ as corresponding to external rays (including their end points) and $\hat Q_n$ as corresponding to equipotential lines in Yoccoz' puzzle partition of a neighbourhood of
Julia set for a complex polynomial.

\begin{defi}\label{nice couple}
A \textit{nice couple for}~$f$ is a pair $(\hV, V)$ of nice sets for~$f$ such that $\overline{V} \subset \hV$, and such that for every $n \ge 1$ we have~$f^n(\partial V) \cap \hV = \emptyset$.

If $(\hV, V)$ is a nice couple for~$f$, then for every pull-back~$\hW$ of~$\hV$ we have either
$$
\hW \cap V = \emptyset
\, \text{ or } \,
\hW \subset V.
$$
\end{defi}

\begin{lemm}[Existence of nice couples]\label{existence of nice couples} For every $(f,K, \hI_K)\in\sA_+$ with no indifferent periodic orbits in $K$,
there exist nice couples
with all components of arbitrarily small diameters.
\end{lemm}

\begin{proof}
One can repeat word by word the proof of the corresponding fact in the complex setting, the non-renormalizable case  \cite[Appendix A: Theorem C]{PR-L2}. The set to be covered here by  nice couples is $S'(f,K)$, see Definition~\ref{exceptional}, which can be larger than $\Crit(f)$. However the proof in \cite{PR-L2} works for every finite set $\Sigma\subset K$ not containing any periodic point, in particular for $S'(f,K)$, see
Lemma~\ref{NO}.


For the interval case see also \cite{CaiLi}.



\end{proof}

Finally notice that bounded distortion property, Definition~\ref{distortion}, implies the following
\begin{rema}\label{epsilon epsilon'}
Assume $(f,K)\in\sA_+^{\BD}$.
If $\e >0$ is such that for a nice couple $(\hV, V)$ for each $c\in S'(f,K)$ the interval $\hV^c$ contains an $\e $-scaled neighbourhood of $V^c$, then there is $\e '$ such that for each  pull-back $(\hW,W)$ of $(\hV,V)$ for $f^{n}$, such that $W$ intersects $K$ and $f^n$ is a diffeomorphism on $\hW$, the interval $\hW$  contains  an $\e '$-scaled neighbourhood of an interval containing $W$.
\end{rema}

\subsection{Canonical induced map}\label{ss:canonical}



The following key definition is an adaptation to the generalized multimodal case of the definition from the complex setting.

\begin{defi}\label{canonical}(Canonical induced map).

Let $(\hV, V)$ be a nice couple for~$(f,K)\in \sA$.
We say that an integer $m \ge 1$ is a \textit{good time for a point $z$} in ${\bf{U}}$, if $f^m(z) \in V$, in particular the iteration make sense,
and if $f^m$ is a $K$-diffeomorphism from the pull-back $\hW$ of $\hV$ for $f^m$, containing $z$, to $\hV$, see Definition~\ref{K-diffeomorphism}. The crucial observation following from Lemma~\ref{NO}
is that if $V$ has all components short enough, then

\smallskip

$f^m:\hW\to\hV$ is a $K$-diffeomorphism if and only if $f^j(\hW)\cap S'(f,K)=\emptyset$
for all $j=0,...,m-1$.

\smallskip


Let $D$ be the set of all those points in $V$ having a good time
and for $z \in D$ denote by $m(z) \ge 1$ the least good time of
$z$. Then the map $F: D \to V$ defined by $F(z) \= f^{m(z)}(z)$
is called {\it the canonical induced map associated to $(\hV,
V)$}. We denote by $K(F)\subset D$ the set where all iterates of $F$ are defined (i.e. the maximal $F$-invariant set) and by~$\fD$ the collection of all the connected components of $D$.
As~$V$ is a nice set, it follows that each connected component~$W$ of~$D$ is a pull-back of~$V$.
Moreover,~$f^{m_W}$ maps  $\hW$ $K$-diffeomorphically onto $\hV^{c(W)}$ and for each $z \in W$ we have $m(z) = m_W$.
\end{defi}

\begin{rema}\label{first return}
Notice that for $z\notin V$ if there exists $m\ge 1$ such that $f^m(z)\in V$ then
the number $m(z)$ is the time of the first hit of $V$ by the forward trajectory of $z$
 by the definition of the nice couple. Indeed, all consecutive pull-backs of $\hV$ until the $m(z)$-th one
 containing $z$ are outside $V$, hence omitting $S'(f,K)$ hence $f$ is a $K$-diffeomorphism on them.

For $z\in V$ the number $m(z)$ need not be the time of the first return to $V$.
\end{rema}

Later on we shall need the following (compare \cite[Section 3.2]{PR-L2}.

\begin{lemm}\label{decomposition into F}  If $m$ is a good time for $z\in V$, such that $f^m(z)\in V$, then there
exists unique $k\ge 1$ such that $f^m=F^k$ on a neighbourhood of $z$.
\end{lemm}
\begin{proof}
If $m(z)=m$ then $f^m=F$ by definition. If $m(z)<m$ then $m-m(z)$ is a good time for $z_1=f^{m_1}(z)$ for $m_1=m(z)$ and we consider $m_2=m(z_1)$. After $k$ steps we get least good times $m_1,...,m_k$
for $z_1,...,z_k$ such that $m_1+...+m_k=m$.
\end{proof}

\begin{defi}\label{bad pull-backs}(Bad pull-backs).
Let, as above, $(\hV, V)$ be a nice couple for~$(f,K)\in\sA$.
 A point~$y \in f^{-n}(V)\cap V$ is a \emph{bad iterated pre-image of order~$n$} if for every~$j \in \{ 1, \ldots, n \}$ such that~$f^j(y) \in V$ the map~$f^j$ is not a $K$-diffeomorphism on the pull-back of~$\hV$ for $f^{j}$, containing~$y$.
In this case every point~$y'$ in the pull-back~$X$ of~$V$ for~$f^n$ containing~$y$ is a bad iterated pre-image of order~$n$. So it is justified to call $X$ itself a bad iterated pre-image of $V$ of order~$n$.
We call a pull-back $Y$ of~$\hV$ for $f^{n}$, intersecting $K$,  a \emph{bad pull-back of~$\hV$ of order~$n$}, if it contains a bad iterated pre-image of order~$n$. (Notice that $Y$ can contain several $X$'s.)

Denote by $\fL_V$ the collection of all the pull-backs $W$ of $V$ for $f^{-n}$,
for all $n\ge 1$, intersecting $K$, such that $f^j(W)\cap V=\emptyset$ for all $j=0,...,n-1$.

 Denote by $\fD_{\badp}$ the sub-collection of $\fD$ of all those $W$
 that are contained in~$\badp$. We do not assume that $Y$ is bad in this notation.
 In particular we can consider $Y=\hV^c$ with $m_Y=0$.

 So $f^{m_W}$ maps $\hW$ $K$-diffeomorphically  onto~$\hV^{c(W)}$,
 $f^{m_{\badp}}(W) \subset V^{c(\badp)}$ and
 $f^{m_{\badp}}(W) \in \fD_{\hV^{c(\badp)}}$. Notice that
 $f(\fD_{\hV^{c(\badp)}})\subset \fL_V$ for $V$ small enough (such that $f(V)\cap V=\emptyset$).

\end{defi}

From these definitions it easily follows that

\begin{lemm}\label{decomposition into bad} (see \cite[Lemma 7.4]{PR-L1} and \cite[Lemma 3.5]{PR-L2})
Let~$(\hV, V)$ be a nice couple for~$(f,K)\in \sA$ and let~$\fD$ be the collection of the connected components of~$D$.
Then
$$ \fD
=
\left( \bigcup_{c \in S'(f,K)} \fD_{\hV^c} \right)
\cup
\left( \bigcup_{\badp \text{ bad pull-back of~$\hV$}} \fD_{\badp} \right). $$
\end{lemm}

\

Lemma~\ref{nodes} yields the following (see \cite[Lemma 3.6]{PR-L2} for a more precise version)

\begin{lemm}\label{nodes2}
Let~$(\hV, V)$ be a nice couple for~$(f,K)\in\sA_+$. Then for every $\e >0$, if the diameters of all $\hV^c$ are small enough,
for all $n$ large enough and for every $x\in V$ the number of bad iterated pre-images of $x$ of order~$n$ is at most $ \exp (\e n)$.
\end{lemm}

\subsection{Pressure function of the canonical induced map}\label{ss:two variable pressure}

Now we shall continue adapting the procedure from the complex case

Let as before $(\hV, V)$ be a nice couple for~$(f,K)\in\sA_+$ and let $F: D \to V$ be the canonical induced map associated to $(\hV, V)$ and
 $\fD$ the collection of its connected components of~$D$.
For each $c \in S'(f,K)$ denote by $\fD^c$ the collection of all
elements of $\fD$ contained in $V^c$, so that $\fD = \bigsqcup_{c
\in S'(f,K)} \fD^c$. A word on the alphabet $\fD$ will be called
\emph{admissible} if for every pair of consecutive letters $W',
W \in \fD$ we have
$W \in \fD^{c(W')}$
(that is $W\subset f^{m(W')}(W')=V^{c(W')}$).

We call the 0-1 infinite matrix $A=A_F$ having the entry $a_{W,W'}$ equal to 1 or 0 depending as
$W \in \fD^{c(W')}$ or not, the incidence matrix. This matrix in the terminology of \cite[page 3]{MU}
yields a {\it graph directed system}, a special case of GDMS.

 For a given integer $n
\ge 1$ we denote by~$E^n$ the collection of all admissible words
of length~$n$. Given $W \in \fD$, denote by $\phi_W$ the
the inverse of $F|_{W}$, on $V^{c(W)}$, and by $\phi_{\hW}$ its extension to $\hV^{c(W)}$.
The collection $\phi_W$ is a Graph Directed Markov System, GDMS, see
\cite{MU}.

For a finite word $\underline{W} = W_1 \ldots W_n \in E^n$ put
$c(\underline{W}) \= c(W_n)$ and $m_{\underline{W}} = m_{W_1} +
\cdots + m_{W_n}$. Note that the composition
$$
\phi_{\underline{W}} \= \phi_{W_1} \circ \cdots \circ \phi_{W_n}
$$
is well defined, extends to $\hV^{c(\uW)}$ and maps it
$K$-diffeomorphically onto $V^{c(\underline{W})}$.

For each $t,p\in\R$ define the \emph{pressure of~$F$ for the potential $\Phi_{t,p}:=- t\log |F'| - pm$}

\begin{equation}\label{e:induced pressure}
P(F, - t \log |F'| - p m) := \lim_{n \to + \infty} \tfrac{1}{n} \log
Z_n(t, p) 
\end{equation}
where, as before,  $m=m(z)$ associates to each point $z \in D$ the
least good time of~$z$ and
for each $n\ge 1$
$$
Z_n(t, p) \= \sum_{\underline{W} \in E^n} \exp(- m_{\underline{W}}
p)\left( \sup \left\{ |\phi_{\underline{W}}'(z)|: z \in
V^{c(\uW)} \right\} \right)^t.
$$

Recall, see e.g. \cite[Section 2.2]{MU}, that a function $\psi: E^\infty\to \C$ is called
H\"older continuous on the
associated symbolic space $E^\infty$ of all infinite admissible words, if the sequence
of variations

$
V_n(\psi):=\sup\{ |\psi(\omega_1)-\psi(\omega_2)|:
 $
 the first $n$ letters of $\omega_1$ and $\omega_2$ coincide$\}$

\noindent tends to 0 exponentially fast as $n\to\infty$.

\begin{prop}\label{distortion-induced}1. For all real $p,t$, if HBD is assumed, see Definition~\ref{LBD},  the potential $-t\log |F'|-pm$ is
H{\"o}lder continuous  on the associated symbolic space.

2. 
The induced map $F$ is uniformly expanding.

3. The diameters of $W=\phi_{\underline{W}}(V^{c(W)})$ shrink uniformly exponentially to 0 as $n\to\infty$ for $n$ the length of the word $\underline{W}$.
\end{prop}

\begin{proof}
Item 3. follows
from the bounded distortion condition BD, see Definition~\ref{distortion} and Remark~\ref{epsilon epsilon'}, with the exponent
of the convergence $(1+2\e ')^{-1}$,  where for each $c\in S'(f,K)$ the interval $\hV^c$ is an $\e $-scaled neighbourhood of $V^c$.
Indeed, notice that the BD implies
that for every $0<j\le n$ and
$W^j:=\phi_{W_1} \circ\cdots \circ \phi_{W_j}(V^{c(W_j)})$ the set $W^{j-1}$ is an
$\e '$-scaled neighbourhood of $W^j$.

Item 2. follows from the fact that $|W|\le (1+2\e ')^{-n}$ and
$|(\phi'_{\underline{W}}(x)/\phi'_{\underline{W}}(y)|\le C'(\e)$ for $x,y\in V^{c(W)}$.
Indeed, these estimates imply
$$
|(\phi_{\underline{W}})'(z)|\le |W|C'(\e )\le  (1+2\e ')^{-n} C'(\e ) \le 1/2 <1
$$
for
$n$ large enough.

Finally notice that for $x,y\in W$ by the assumption HBD in Definition~\ref{LBD} we have
$\log |F'(x)-\log |F'(y)|\le C(1+2\e ')^{-(n-1)\alpha}$,
since $V^c$ containing $F(W)$ is its  $(1+2\e ')^{-1}$-scaled neighbourhood, by item 3.
This expression shrinks exponentially fast to 0 as $n\to\infty$ yielding exponential decay of
variations in the symbolic space, hence H\"older continuity of $\log |F'|$, hence
H\"older continuity of $-t\log |F'| -pm$.
\end{proof}

This Proposition, Items 2 or 3, allow to define a standard limit set and its coding by the space
$E^\infty$ of all admissible infinite sequences
\begin{defi} Define $\pi:E^\infty \to K(F)$ by
$$
\pi(W_1W_2...)=\bigcap_{n=0}^\infty\phi_{W_1}\circ ...\circ \phi_{W_n} (V^{c(W_n)}).
$$
\end{defi}
Notice that due to the {\it strong separation property}, see \cite[Definition A.1.]{PR-L1}, holding by the properties of nice couples, $\pi$ is a homeomorphism.

Notice that
due to the BD assumption and due to topological transitivity of $f|_K$
we can replace in the definition of
$P(F, - t \log |F'| - p m)$ the expressions
$ \sup \left\{ |\phi_{\underline{W}}'(z)|: z \in
V^{c(\uW)} \right\}$ summed up over all $\underline{W}\in E^n$, by $|\phi_{\underline{W}}'(z_0)|$
summed up over all $F^{-n}$-pre-images of an arbitrarily chosen $z_0\in V$.
Compare this with the definition
of tree pressure Definition~\ref{treep}.

\

Denote $P(F, - t \log |F'| - p m)$ by $\pressure (t,p)$.
It can be infinite, e.g. at $t\le 0, p=0$.
The finiteness is clearly equivalent to the summability condition, see \cite[Section 2.3]{MU},
\begin{equation}\label{summability}
\sum_{W\in \fD}\sup_{z\in W} \Phi_{t,p}=\sum_{W\in \fD}\sup|\phi'_W|^{-t}\exp(-m_W p)<\infty.
\end{equation}
This finiteness clearly holds for all $t$ and $p> P_{\rm{tree}}(K,t)$,
where it is in fact negative, see \cite[Lemma 3.8]{PR-L2}.
In Key Lemma we shall prove this finiteness on a bigger domain.

We study $\pressure$ on the domain where it is finite.
$\pressure$ is there a real-analytic function of the pair of variables $(t,p)$,
see \cite[Section 2.6]{MU} for the analyticity with respect to one real variable.
The analyticity follows from
the analyticity of the respective transfer (Perron-Frobenius, Ruelle) operator, hence its
isolated eigenvalue being exponent of $\pressure$. The analyticity with respect to $(t,p)$
follows from the complex analyticity with respect to complex $t$ and $p$ and Hartogs' Theorem.

\

\section{Analytic dependence of Geometric Pressure on temperature. Equilibria.}\label{s:analytic}

To prove Theorems A and C, in particular to study $P(K,t)$, we shall use the following Key Lemma, whose proof will be postponed to the next section.

\begin{lemm}\label{key}(Key Lemma)

1. For $(f,K)\in \sA_+^{\BD}$, if moreover HBD is assumed, see Definition~\ref{LBD}, for $F$ and
$\pressure(t,p)$ defined with respect to a nice couple, the domain of finiteness of $\pressure(t,p)$ contains a neighbourhood of the set $\{(t,p)\in\R^2: t_-<t<t_+, p=P(K,t)\}$.

2. For every $t: t_-<t<t_+$,\; $\pressure(t,P(t))=0$.

\end{lemm}

 In this Section we shall prove Theorems A and C using Key Lemma.
 The proof
 follows the complex version in \cite[Theorems A and B]{PR-L2}, but it is simpler (the simplification concerns also the complex case) by omitting the considerations of subconformal measures.

\subsection{The analyticity} Notice that $\partial\pressure(t,p)/\partial p < 0$.
Indeed,  using the the bound $m\ge 1$ for the time of return function $m=m(x)$ giving $F=f^m$, considering arbitrary $p_1<p_0$ so that $(t,p_1)$ and $(t,p_1)$ belong to the domain of $\pressure$,  we get
$$
P(F, - t \log |F'| - p_1 m) - P(F, - t \log |F'| - p_2 m)
$$
$$=
\lim_{n\to\infty}{1\over n}\log Z_n(t, p_1) - \lim_{n\to\infty}{1\over n}\log Z_n(t, p_0)
 $$
 $$ = \lim_{n\to\infty}{1\over n}  \log  {
 \sum_{\underline{W} \in E^n} \exp(- m_{\underline{W}}
p_1)\left( \sup \left\{ |\phi_{\underline{W}}'(z)|: z \in
V^{c(\uW)} \right\} \right)^t \over
\sum_{\underline{W} \in E^n} \exp(- m_{\underline{W}}
p_0)\left( \sup \left\{ |\phi_{\underline{W}}'(z)|: z \in
V^{c(\uW)} \right\} \right)^t  }
$$
$$
\ge \lim_{n\to\infty}{1\over n} n|p_1-p_0| =|p_1-p_0|>0,
$$
since the ratio of each corresponding summands in the ratio of the sums above is bounded from below by
$\exp m_{\underline{W}}|p_1-p_0|\ge \exp n|p_1-p_0|$.

Thus the analyticity of $t\mapsto P(t)$ for $t_-<t<t_+$ follows from the Implicit Function Theorem and
the Key Lemma.
\begin{rema} In \cite[Subsection 4.4]{PR-L2} an indirect argument was provided, expressing the derivative $\partial{\pressure}/\partial p$ by a Lyapunov exponent for the equilibrium existing by the last item of Theorem A. The above elementary calculation is general, working also in the complex case.
\end{rema}


\subsection {From the induced map to the original map. Conformal measure. Proof of Theorem A.2.} \label{From the induced map to the original map. Conformal measure}

We shall call here any $\phi$-conformal measure  for $f$ on $K$ for $\phi:=(\exp p) |f'|^t$
a $(t,p)$-{\it conformal measure} for $f$, see Definition \ref{defi:Jacobian}. We call $\mu$ on $K(F)$ a  $(t,p)$-{\it conformal measure} for an induced map $F$, as constructed in Subsection 6.2, if it is $\phi$-conformal for $F$ and for $\phi(x):=(\exp p m(x)) |F'(x)|^t $.

A $(t,P(t))$-conformal measure $\mu_{F,t}$ for $F$, supported on $K(F)$ in the sense
$\mu_{F,t}(K(F))=1$, exists by Key Lemma and by
\cite[Theorems 3.2.3]{MU}. It is non-atomic and supported $K_{\con}(f)$ since
the latter set contains $K(F)$ by definitions.

Similarly as in \cite[Proposition B.2]{PR-L1} $\mu_{F,t}$ can be extended to the union of $\fL_V$, see Definition~\ref{bad pull-backs}, by
$$
\tilde{\mu_t}(A):=\int_{f^{m(W)}(A)}(\exp -P(t)m(W)) |(f^{-m(W)})'(x)|^{t}\;d\mu_{F,t} (x)
$$
for each $W\in \fL_V$
and every Borel set $A\subset W$ on which $f^{m(W)}$ is injective. The conformality of this measure for $f$, i.e. \ref{Jacobian}, holds for all
$A\subset W\in \fL_V$ and $A\subset K(F)$ by the same proof as in \cite{PR-L1}.
Finally normalize $\tilde{\mu_t}$ by setting $\mu_t:=\tilde{\mu_t}/\tilde{\mu_t}(K(F)$.
(Sometimes we omit tilde over not normalized $\mu$ to simplify notation.)

Notice that $\mu=\mu_t$
is finite and non-atomic, as a countable union of non-atomic pull-backs of $\mu_{F,t}$ to
$W\in \fL_V$. The summability $\sum_{W\in \fL_V} \mu(W) <\infty$ follows from BD (bounded distortion) and (\ref{e:pressure outside}).

Denote $\hat K(F):= \bigcup_{W\in \fL_V}f^{-m(W)}(K(F))$.
Let $x\in V\setminus (K(F) \cup S'(f,K))$. We prove that

\begin{equation}\label{image}
y=f(x)\notin \hat K(F).
\end{equation}

Suppose this is not the case. For $V$ small enough, $y\notin V$. Denote $z=f^{m(y)}(y)$ where
$m(y)$ is the time of the first return of $y$ to $V$, see Remark~\ref{first return}.
Then there exists an infinite word $\underline{W}=W_1W_2...\in E^\infty$ such that
$z=f^{m(y)}(y)=\pi(\underline{W})$. Let $n_0\ge 1$ be the least $n$ such that $Y_n$ being the pull-back of $\hW_n=\phi_{W_1}\circ ... \circ \phi_{W_n}(V^{c(W_n)}$ of order $m(y)+1$ containing $x$ is disjoint from $S'(f,K)$. It exists since $|\hW_n|\to 0$ as $n\to\infty$. Hence $|Y_n|\to 0$, hence $Y_n\cap S'(f,K)=\emptyset$ as $x\notin S'(f,K)$. In consequence
$$
m=m(y)+1+m_{W_1}+...+m_{W_{n_0}}
$$
is a good time for $x$. Hence, by Lemma~\ref{decomposition into F}, there exists $k\ge 1$ such that $f^m(x)=F^k(x)$.
In consequence $x\in K(F)$, a contradiction.

(More precisely $x=\pi(\underline{W}')$, where either $\underline{W}'=W_0W_1W_2...$ for $W_0$ being a pull-back of $V$ for $f^{m(y)+1}$, containing $x$ in the case $m(y)+1$ is a good time for $x$, or
 $\underline{W}'=W_0'W_{n_0+1}...$ otherwise, where $W_0'\ni x$ and $m_{W_0'}=m$. )

Suppose now that $A\subset V\setminus K(F)$.
Then by definition $\mu(A)=0$. We shall prove that also
$\mu(f(A))=0$. Indeed, we have just proved that $f(A\setminus S'(f,K))\cap \hK(F)=\emptyset$,
hence $f(A\setminus S'(f,K))$ is disjoint from the set $\hK(F)$ of full measure $\mu$.
Notice also that $\mu(f(S'(f,K)))=0$ since $\mu$ does not have atoms.

Thus the proof of existence of a $(\exp P(t))|f'|^t$-conformal measure $\mu$ for $f$ on $K$ as asserted in Theorem A, is finished. This measure is non-atomic and supported on $K_{\con}(f)$ by construction.
It is zero on (weakly) exceptional sets since such sets are finite and $\mu$ is non-atomic. It is positive on open sets in $K$ by Lemma~\ref{positive open}.

If there is another $(\exp p)|f'|^t$-conformal measure $\nu$, with $p\le P(t)$, then
it is
positive on open sets by Lemma~\ref{positive open},
$p=P(t)$ and $\mu\ll\nu$. In consequence both measures are proportional. The proof is the same as in \cite[Subsection 4.1]{PR-L2}. Also the proof of ergodicity is the same.

Caution: By the definition of conformal measures, Definition~\ref{defi:Jacobian}, in particular by the forward quasi-invariance of $\nu$, we have $\nu(f(\Crit(f))=0$ (for $t\not=0$). Hence,
if $\NO(f,K)\subset \Crit(f)$ then $K_{\con}(f)$ is forward invariant up to $\nu$-measure 0. If
$\NO(f,K)\not\subset \Crit(f)$ this a priori need not be the case.
Fortunately we only need to pull back:
$f(K_{\con}(f))\subset K_{\con}(f)$ implies  $f^{-1}(K\setminus K_{\con}(f))\subset K\setminus K_{\con}(f)$. Hence, if $\nu':=\nu|_{K\setminus K_{\con}}(f)$ is nonzero, by pulling back by iterates of $f$ we prove that $\nu'$ is positive on open sets, hence by pulling back $\nu'$ and $\mu$ simultaneously we prove that $\nu'$ is positive on $K_{\con}$,  contradiction.


\

\subsection{Equilibrium states}

The proof of this part of Theorem A is the same as in the complex case,
so we only sketch it.

\

1. {\it Existence}

As in \cite[Lemma 4.4]{PR-L2} one deduces from Key Lemma \ref{key} that the return time function $m(z)$ is $\mu_t$ integrable, where $\mu_t$ is the unique probability $(t,P(t))$-conformal measure, for every $t_-<t<t_+$.
Moreover the exponential tail inequality holds
\begin{equation}\label{tail}
\sum_{W\in \fD, m_W\ge n} \mu_t(W)\le \exp (-\e  n),
\end{equation}
for a constant $\e>0$ (depending only on $t$) and all $n$ large enough.

Due to the summability (finiteness of $\pressure(t,P(t))$ \ref{summability} and H\"older continuity of the potential function $\Phi_{t,P(t)}$ on $E^\infty$, see Proposition \ref{distortion-induced}, there exists an $F$ invariant probability measure $\rho_{F,t}$ absolutely continuous with respect to the conformal measure  $\mu_{F,t}$ on $K(F)$ being the restriction of $\mu_t$ to this set (so $\mu_{F,t}(K(F))<1$ usually).

Indeed, consider the Gibbs state on $E^\infty$ given by \cite[Theorem 2.2.4]{MU}.  Project it by $\pi_*$ to $\rho_{F,t}$ on $K(F)$. Its equivalence to
$\mu_{F,t}$ on $K(F)$ follows from Gibbs property \cite[(2.3)]{MU} holding for both measures.

 Moreover there exists $C_t >0$ such that $C_t\le d\rho_{F,t}/d\mu_{F,t} \le C_t^{-1}$ i.e the density is bounded and bounded away from 0.

\

One has used here also, while applying \cite{MU}, in particular concluding the ergodicity of $\rho_{F,t}$, the fact that the incidence matrix $A_F$ is finitely irreducible, which follows from the topological transitivity of $f$. The argument for this irreducibility is as follows, compare
\cite[Proof of Lemma 4.1]{PR-L1}.

\

1a. {\it Proof of the irreducibility}

Consider an arbitrary repelling periodic point $p\in K\setminus \bigcup_{n\ge 0}f^n(S'(f,K))$.
For every $c\in S'(f,K)$ choose a
backward trajectory $\gamma_1(c)$
converging to the periodic orbit $O(p)$, existing by Proposition~\ref{density periodic}.
It omits $S'(f,K)$ by our assumption $f^n(S'(f,k))\cap S'(f,K)=\emptyset$ for all $n>0$.
Therefore for $V$ small enough all these trajectories, hence the pull-backs of $\hV$ along them, are disjoint from $V$.

Next notice that for each $c\in S'(f,K)$ one can choose a
finite backward trajectory $\gamma_2(c)=(p=z(c)_0, z(c)_{-1},...,z(c)_{-N(c)})$ such that
$z(c)_{-N(c)}\in V^c$,  by the strong transitivity property, see Proposition~\ref{dp}.

Consider now any two $c,c'\in S'(f,K)$. Consider a pull-back of $\hV$ going along $\gamma_1(c)$
for the time $M(c)$ so long that

(a) the diameter of the associated pull-back $\hW(c)$ of $\hV^c$ is small enough that its consecutive pull-backs for $f^{j}$ along $\gamma(c')$ have diameters small enough to be disjoint from $S'(f,K)$,
for all $j=0,...,-N(c')$

(b) the distance of $\hW(c)$ from $O(p)$ is small enough that $\hW(c,c')$ being the pull-back of $\hW(c)$ is in $V^{c'}$.

Then, for $x=x(c,c')$ being the $f^{-(M(c)+N(c'))}$-pre-image of $c$ in $\hW(c,c')$, the time
$m=m(c,c'):=M(c)+N(c')$ is a good time, see Definition~\ref{canonical}. Hence by Lemma~\ref{decomposition into F}  there exists $k=k(c,c')$ such that $f^m=F^k$ on $W(c,c')$.

We conclude that for every two $W',W\in\fD$, if $W\subset V^c$ and $f^{m_{W'}}=V^{c'}$,
then there is an admissible word of length $k(c,c')+2$ on the alphabet $\fD$ starting with $W'$ and ending with $W$.

\


{\it Existence. Continuation}

Now, using the $F$-invariant measure $\rho_{F,t}$, one defines an $f$-invariant measure in the standard way:
\begin{equation}\label{rho}
\rho_t \= \sum_{W \in \fD} \sum_{n = 0}^{m_W - 1} f^n_* \rho_{F,t}|_W.
\end{equation}

Finiteness of $\rho_t$ follows from the integrability of $m(x)$ with respect to $\rho_{F,t}$
which follows from the integrability with respect to $\mu_{F,t}$ and the boundness of
${d\mu_{F,t} \over d\rho_{F,t}}$.


The absolute continuity of $\rho_t$
with respect to the conformal measure $\mu_t$ follows easily from this definition, see e.g. \cite[p.165-166]{PR-L1}. Indeed.
For each  $W\in \fD$ and $n: 0\le n< m_W$ write
$J_{W,n}:=|(f^n)'|^{-t}$ on $f^n(W)$, and~$0$ on the rest of ${\bf{U}}$.
Replace $\rho_t$ in \ref{rho} by the auxiliary $\mu'_t=\sum_{W \in \fD} \sum_{n = 0}^{m_W - 1} f^n_* \mu_{F,t}|_W$.
Then, using \ref{rho} and the boundness of $d\mu'_t/d\rho_t$, we get
$$
+\infty > \mu'_t(K)= \sum_{W,n} \int J_{W,n}\,d\mu_t=\int \sum_{W,n} J_{W,n}\,d\mu_t,
$$
by Lebesgue Monotone Convergence Theorem.
Moreover for every continuous function $u:K\to\R$
$$
\int u d\mu'_t= \sum_{W,n} \int u J_{W,n}\, d\mu_t=
\int u (\sum_{W,n}  J_{W,n})\, d\mu_t.
$$
Thus, it follows that $\mu'_t$, hence $\rho_t$ is absolutely continuous with respect to $\mu_t$.

\

1b. {\it The density function}

The density $d\rho_t/d\mu_t$ is bounded away from 0. Indeed, let $N(V)$ be such that
$\bigcup_{n=0}^N(V\cap K)=K$. Then for $\mu_t$ almost every $z\in K$ there exists $y\in K(F)$ and
$N\le N(V)$ such that $f^N(y)=z$. We get
$$
{d\rho_t \over d\mu_t}(z)\ge {d\rho_{F,t} \over d\mu_{F,t}}(y) |(f^N)'|^{-t} e^{-NP(t)}
\ge C_t
(\max_{n=0,1,...,N(V)}  \sup_K |(f^n)'|^{t} e^{nP(t)})^{-1},
$$
compare again \cite[page 166]{PR-L1}.

\

1c. {\it Equilibrium}

The proof is the same as in \cite[Lemma 4.4]{PR-L2}.
Denote $\rho:=\rho_{F,t}$. Consider its  normalized extension given by (\ref{rho}), by  $\rho':=\rho_t/\rho_t(K)$.
One uses
the fact that $\rho$ is the equilibrium state for the shift map on $E^{\infty}$ (identified with $K(F)$) and the potential $\Phi_{t,P(t)}$, i.e.
$P(F, -t \log |F'| - P(t) m)
=
h_{\rho}(F) - \int (t \log |F'| +  P(t) m) \, d\rho$,
which is equal to~$0$ by hypothesis.
Thus, using generalized Abramov's formula, see \cite[Theorem 5.1]{Zwe}, we get
\begin{equation}\label{tower-calculation}
h_{\rho'}(f)
=
(\rho_t(K))^{-1} h_{\rho}(F)
=
(\rho_t(K))^{-1} \left(\int (t \log |F'| +  P(t) m) \, d\rho\right)
\end{equation}
$$
=(\rho_t(K))^{-1} t \int \log |f'| d\rho_t + P(t)
=
t \int \log |f'|\, d\rho' + P(t).
$$
This shows that~$\rho'$ is an equilibrium state of~$f$ for the potential $-t \log |f'|$.

\

2. {\it Uniqueness}

Let $t_-<t<t_+$ and $\nu$ be an ergodic equilibrium measure for $f|_K$ and the potential $-t\log|f'|$.  Since $P(t)>-t \chi_\nu$, it follows that $h_\nu(f)>0$. Hence by Ruelle's inequality $\chi_\nu(f)>0$. Now one can refer, as in the complex case, to \cite{Dobbs2} adapted to the interval case. In our case, where nice couples exist, there is however a simpler proof using inducing, omitting Dobbs (and Ledrappier) method of using canonical systems of conditional measures on the measurable partition into local unstable manifolds in Rokhlin natural extension. See Appendix B.

\

3. {\it Mixing and statistical properties}

To prove these properties in theorem A we refer, as in \cite{PR-L1} and \cite[Subsection 8.2]{PR-L2}, to Lai-Sang Young's results, \cite{Young}.

In the case when there is only one critical point in $K$ one can apply these results directly, and in the general case one
fixes  an appropriate $c_0\in S'(f,K)$ and considers $\hF:V^{c_0}\cap K(F)\to V^{c_0}\cap K(F)$, the first return map for iteration of $F$. This makes sense since $\rho_{F,t}(V^{c_0}\cap K(F))>0$ as
$\mu_t$ is positive on open sets in $K(F)$.

$\hF$ is an infinite one-sided Bernoulli map, with $\fD$ replaced by $\hat{\fD}$ being the joining of $\fD$ and its appropriate $F^j$-preimages, see the next paragraph.
Now we refer to \cite{Young} considering Young's tower for $\hF$ and the integer-valued function
$\hat m$ on $V^{c_0}\cap K(F)$ defined $\rho_{F,t}$-a.e. by $\hF (x)=f^{\hat m}(x)$, more precisely
${\hat m}(x)=\sum_{j=0,...,m_F-1} m(F^j)(x)$, where $F^{m_F}={\hF}$.

The tower $T_{\hF, \hat m}$ is the disjoint union of pairs $(U,n)=
(\phi_{\underline{W}}(V^{c(\underline{W})}),n)$, with $W_1\subset V^{c_0}$, $c(W_{m_F}=c_0$ and $c(W_i)\not=c_0$ for $i=1,...,m_F-1$
(the length of $\underline{W}$ being  $m_F$) and $n=0,1,...,m_U -1$.

On each $(U,n)$ define the measure $\hat\rho_{F,t}$ as the image of $\rho_{F,t}$ on $U=(U,0)$  under the mapping  $(x,0)\mapsto (x,n)$. It is easy to see that $(x,n)\mapsto f^n(x)$ projects this measure to $\rho_t$ on $K$.


Then by \cite{Young}, see also \cite[Subsection 8.2]{PR-L1}, to prove mixing, exponential mixing and CLT it is sufficient to check the following.

1) The greatest common divisor of the values of $m_U$ is equal to 1.

2) The tail estimate \ref{tail} $\sum_{U, m_U\ge n}\rho_t(U)\le \exp (-\e n)$ for a constant $\e>0$,
for $U\in\hat{\fD}$.

\

To prove 1) we proceed as in 1a. (the proof of irreducibility) with some refinements, compare
\cite[Lemma 4.1]{PR-L1}. First notice that provided topological exactness of $f|_K$ there are in $K$ periodic orbits of all periods large enough. To prove this it is sufficient to
be more careful in Proof of Lemma
2.11 while "closing the loop". Given a pull-back $T$ for $f^{n}$, of

\noindent $B=B(z_0,\exp(-\alpha n))$ and large  $S=f^m(B)$ such that $f^m:B\to S$ is a diffeomorphism and $m<<n$, then we find a diffeomorphic pull-back  from $T$ into $S$ of an arbitrary length not exceeding $n$.


Now choose an arbitrary $O(p)$ as in 1a. For every $c\in S'(f,K)$ choose a backward trajectory $\gamma_1(c)$ converging to $O(p)$. Choose $\gamma_2=(p,z_{-1},z_{-2},...z_{-N})$, a backward trajectory of $p$,  such that $z_{-N}\in V$. Let $n\le N$ be the least $n\le N$ such that $z_{-n}\in V$. Let $c_0 \in S'(f,K)$ be such that $z_{-n}\in V^{c_0}$. Thus for each $c$, in particular for $c_0$, we find, as in 1a., a backward trajectory $c_0,x_{-1},...,x_{-m}$ of $c_0$ such that $x_{-m}\in V^{c_0}$ no
$f^{-j}(c_0)$ is in $V$ for $j=1,2,...,m-1$ (on its way going several times along $O(p)$).  Next find a point $y\in V^{c_0}$ shadowing this trajectory, hence $F(y)=f^m(y)\in V^{c_0}$. Hence $\hF =F=f^m$ at $y$, with the return time $m$ being an arbitrary integer of the form $a+k m_p$, where $m_p$ is a period of $p$.

Considering now $\#S'(K,f)+1$ number of periodic points in

\noindent $K\setminus\bigcup_{n\ge 0} f^n(S'(f,K))$ with pairwise mutually prime periods, we find two,  having the same $c_0$. Then the condition on the greatest common divisor of return times being 1 is satisfied for appropriate choices of $k$. For details see \cite[Lemma 4.1]{PR-L1}.

\

To prove 2) one uses \ref{tail} for $F$, i.e. for $W\in \fD$.
One repeats roughly the estimates in \cite[the top of p. 167]{PR-L1}.
The key point is the finiteness of the sum of the middle factors in the decomposition
${\hF}'(x)=F'(x)\cdot (F^{m-2})'(F(x))\cdot F'(F^{m-1}(x))$, for $F^m=\hF$, namely,
writing $m_j(y):=m(y)+m(F(y))+...+m(F^{j-1}(y))$,

\begin{equation}\label{sumeps}
\sum_{j \ge 1} \; {\sum_y}^*
|(F^j)'(y)|^{-t} \exp{-(m_j(y)) (P(t)-\e)} < \infty
\end{equation}
for a positive $\e$. Here the star * means that we consider only $y$ one for each word $\underline{W}$ (of length $j$), such that no $F^i(y)$ belongs to
$V^{c_0}$ for $i=0,1,...,j$. This summability holds because the pressure of the subsystem of the system in  \ref{e:induced pressure} where we omit symbols in $V^{c_0}$ is strictly less than the full
$P(F, -t\log |F'|-P(t)m =0$ by Key Lemma~\ref{key}, hence negative. So we have a room to subtract $\e$ in \ref{sumeps} and preserve convergence.


\

\section{Proof of Key Lemma. Induced pressure}\label{s:Key}

Again the proof repeats the proof in \cite{PR-L2}, so we only sketch it (making some simplification
due to real dimension 1).

\


\

{\bf Part 1.}

\partn{1.1} First notice that for every $t\in\R$ and every $p$ close enough to $P(t)$
\begin{equation}\label{e:pressure outside}
\sum_{W \in \fL_V} \exp( - p m_W) \diam(W)^{t} < + \infty.
\end{equation}
The proof is the same as in \cite[Lemma 6.2]{PR-L2} and uses the fact that $f$ is expanding on
${\sK}(V)$, i.e. outside $V$, see notation in Definition \ref{nice}. Briefly: ${\sK}(V)$
can be extended to an isolated hyperbolic subset ${\sK}'(V)$ of $K$ (even better than in \cite{PR-L2} where we can guarantee only the existence of Markov partition, compare   \cite[Remark 4.5.3]{PU}). This set can be extended to an even larger isolated hyperbolic set ${\sK}''(V)\subset K\setminus S'(f,K)$. Then
$P(f|_{{\sK}'(V)}, -t\log |f'|)    < P(f|_{{\sK}''(V)}, -t\log |f'|) \le P(K,t)$.

\

\partn{1.2} Let ${\sW}_{n,k}$ be the family of all components of the set
$$
A_{n,k}:=\{x\in {\bf{U}}: 2^{-(k+1)}\le \dist(x, S'_n)\le 2^{-k}\}
$$
for
$S'_n:= \bigcup_{j=1,...,n+1}f^j(S'(f,K)))$,
with the exception that if such a component is shorter than $2^{-(k+1)}$ we add it to an adjacent
component of $A_{n,k+1}$.
Clearly for each integer $k$
\begin{equation}\label{Whitney}
\# \sW_{n,k} \le 2n\#S'(f,K),
\end{equation}
in particular the bound is independent of $k$. The coefficient 2 appears because given $k$ the intervals of $\sW_{n,k}$ can lie on both sides of each point in $S'_n$.

Denote $\sW_n:=\bigcup_{k=0}^\infty \sW_{n,k}$.

(For each $n$ the family $\sW_n$ is a Whitney decomposition of the complement of $S_n'$, similarly to  the Riemann sphere case in \cite{PR-L2}.)

\

\partn{1.3} 
   Clearly for every $a>1$ there exists $b>1$ such that for every $L\in \sW_n$
  and an interval $T$ intersecting $L$, if $aT$ is disjoint from $S'_n$, then $T\subset bL$. We denote by $aT$ the interval $T$ rescaled by $a$ with the same center, similarly $bL$.

We apply this for $T=W\in\fL_V$ such that $W$ is disjoint from $S'_n$. Hence $a$ as above exists by
Remark~\ref{epsilon epsilon'}.

  For each $L\in \sW_n$ let $m(L)$ be the least integer $m$ such that $f^j(b L)$ is disjoint from $S'(f,K)$ for all $j=0,1,...,m$. Hence $f^{m(L)+1}$ is a $k$-diffeomorphism on $bL$.
  Notice also that $f^{m(L)+1}(bL)$ is large since it joins a point $c\in S'(f,K)$ to $\partial V^c$,
  since $L$ intersects a set $W\in {\fL}$ (the family $\fL_V$ is dense in $K$ by the topological transitivity)  hence $W\subset bL$. Therefore $f^{m(L)+1}(L)$ is also large, since otherwise $L$ would be small in
  $bL$ see Definition~\ref{distortion}.

For every $Y$ being a bad pull-back of~${\hV}^c$ or $Y={\hV}^c$ in the decomposition in Lemma~\ref{decomposition into bad}, denote by $\sW (\badp)$ the set of all components $L_Y$ of
$f^{-(m_Y+1)}(L)$ intersecting $Y$ for all $L\in {\sW}_{m_Y}$, such that there exists $W\in {\fD}$ intersecting $L_Y$.
Notice that for each $L$ as above and its pull-back $L_Y\in \sW (\badp)$ the mapping $f^{m_Y+1}:L_Y\to L$ is a $K$-diffeomorphism.


For each $L_Y\in \sW (\badp)$, using distortion bounds, we get 
\begin{multline}\label{one branch}
\sum_{W\in {\fD}_Y, \, W\cap L_Y\not=\emptyset }  \exp (-pm_W) |W|^t \\
\le \Const
\exp ( -p(m_Y+1+m(L))\Big({|L_Y| \over |f^{m(L)}(L)|}\Big)^{t}
\sum_{W\in\fL_V} \exp(-pm_W) |W|^t.
\end{multline}


The latter sum is finite by \ref{e:pressure outside}.
Since $f^{m(L)}$ is large it contains a safe hyperbolic point $z$ from a finite {\it a priori} chosen set $Z\subset K$. Let $z_{L_Y}$ be its $f^{m_Y+1+m(L)}$-preimage in $L_Y$.
Hence,  in \ref{one branch}, the term
$\Big({|L_Y| \over |f^{m(L)}(L_Y)|}\Big)^{t}$ can be replaced up to a constant related to distortion, by   $ |(f^{m_Y+1+m(L)})'(z_{L_Y})|^{-t}$ which can be upper bounded by $\Const\gamma^{m_Y+1+m(L)}$ for
any $\gamma$ satisfying
$$
\exp \left( - \tfrac{1}{2} (P(t) - \max \{ -t \chiinf, - t \chisup \}) \right) <\gamma<1.
$$
For $t\le 0$ this holds due to
$ \lim_{n \to + \infty} \tfrac{1}{n} \log \sup \{ |(f^n)'(z)|: z \in K \}
=
\chisup $, compare \cite[Proposition 2.3, item 2]{PR-L2}.
For $t>0$ we use the same with $\chisup$ replaced by $\chiinf$, but we use the assumption that $z\in Z$ is safe and expanding, compare \cite[Proposition 2.3, item 2]{PR-L2}.

In conclusion
\begin{equation}\label{one branch2}
\sum_{W\in {\fD}_Y, \, W\cap L_Y\not=\emptyset }  \exp (-pm_W) |W|^t \le \Const
\gamma^{m_Y+1+m(L)}.
\end{equation}


\partn{1.4}
Now we collect above estimates and use Lemma~\ref{nodes2} to prove the summability \ref{summability}.
which yields the finiteness of $\pressure (t,p)$.
\begin{multline}
\sum_{W\in \fD}    \exp (-pm_W) |W|^t \\
\le \Const \sum_Y\sum_{L_Y}\sum_{W\in {\fD}_Y, \, W\cap L_Y\not=\emptyset}
  \exp (-pm_W) |W|^t  \\
 \le \Const\sum_{m,k} \sum_{Y: m_Y=m} \sum_{L\in \sW_{m,k}}
\sum_{L_Y} \gamma^{m+1+m(L)} \\
 \le \Const \sum_{m,k} (\exp m\e ) 2m\#S'(f,K)
\gamma^m\gamma^{\alpha k} < \infty,
 \end{multline}
 provided $\gamma\exp\e  <1$, where $\alpha:={\inf_{n,k} \inf_{L\in{\sW}_{n,k}} m(L) \over k}$
is positive since $f$ is Lipschitz continuous. The factor $\exp m\e $ bounds $\#\{Y\}$, compare Lemma~\ref{nodes}.

\

{\bf Part 2.}

The main point is to prove that $p(t)$, the zero of $\pressure(t,p)$, is not less than $P(t)$.
For this end we prove that for all $p>p(t)$ we have $P(f|_K, -t\log|f'|) \le p$ which is equivalent to
$P(f|_K, -t\log|f'|-p) \le 0$.
For this it is sufficient to prove that the summability of the "truncated series"

\begin{equation*}
T_F(w) \= \sum_{k = 1}^{+ \infty} \sum_{y \in F^{-k}(w)} \exp( - p m_k(y) ) |(F^k)'(y)|^{-t},
\end{equation*}
where $m_k(y):=\sum_{j=0,...,k-1} m(F^j(y))$, for an arbitrary $w\in K(F)$
implies the summability of the "full series"
$$
\sum_{n = 1}^{+ \infty} \exp( - p n) \sum_{y \in f^{-n}(z_0)} |(f^n)'(y)|^{-t},
$$
for an arbitrary safe expanding $w\in K$. The proof repeats word by word the proof in \cite[subsection 7.2, page 694]{PR-L2} and will be omitted.

The idea is that fixed a safe expanding point $w\in V$, for every $y\in f^{-n}(w)\cap K$ we consider
$s:0\le s\le n$ the time of the first hit of $V$ by the forward trajectory of $y$. Next let $m:s\le m\le n$ be the largest positive time so that $y':=f^s(y)$ is a good pre-image of $f^m(y)$. If such $m$ does not exist (i.e. $y'$ is a bad pre-image of $w$) then we set $m=s$. By Lemma~\ref{decomposition into F}
$f^{m-s}(y')=F^k(y')$. The sum over each of three kind of blocks, to the first hit of $V$, of $F^k$-type, and over bad blocks, is finite.

\

\

\

\appendix
\section{More on generalized multimodal maps}\label{s:AppendixA}

\smallskip

\subsection{Darboux property approach and other observations}

\

\

Instead of starting from a multimodal triple (or quadruple) defined in Introduction with the use of the notion of the maximality one can act
from another end. Let $K\subset\R$ be a compact subset of the real line, $U$ be an open neighbourhood of $K$
being the union of a finite number of open pairwise disjoint intervals $U^j, j=1,...,m(U)$,
and let $f:U\to\R$ be a $C^2$ mapping having a finite number of critical points, all non-flat and all in $K$, and such that $f(K)\subset K$.

\begin{defi}\label{Darboux}
We say that the triple $(f,K,U)$ satisfies \emph{Darboux property}
if
for every interval $T\subset U$  there exists an interval $T'\subset \R$ (open, with one end,  or with both ends) such that
$f(T\cap K)=T'\cap K$, compare \cite[page 49]{MiSzlenk}.

Call each set $L\subset K$ equal to $T\cap K$ for an interval $T$ a $K$-{\it interval}.
Then Darboux property means that the $f$-image of each $K$-interval in $U$ is a $K$-interval.

Finally notice that Darboux property for $f$ implies Darboux property for all its forward iterates on their domains.
\end{defi}

\begin{prop}\label{reduce}
 1. Let $(f,K,U)$ be a triple as above. Assume that $f|_K$ is topologically transitive.
 Assume also that the triple satisfies Darboux property.
 Then we can replace $U$ by a smaller open neighbourhood ${\bf{U}}$ of $K$ such that $f$ restricted to the union $\hI_K$ of convex hulls $\hI_K^j$ of
${\bf{U}}^j\cap K$,
where as in the notation above ${\bf{U}}$ is the union of a finite number of open pairwise disjoint intervals ${\bf{U}}^j, j=1,...,m({\bf{U}})$,
gives a reduced generalized multimodal quadruple $(f,K, \hI_K,{\bf{U}})$ in the sense of Definition~\ref{extension-1}, in particular $K$ is the maximal repeller in $\hI_K$.

2. The converse also holds. The maximal repeller $K=K(f)$ for
$(f,K,\hI)\in \sA$
satisfies Darboux property on each $\hI^j\cap K$.
\end{prop}


\begin{proof}
For each $j:1\le j \le m(U)$ denote by $\hK^j$ the convex hull of $U^j\cap K$ and
$\hK=\hK_{U}:=\bigcup_{j=1,...,m(U)} \hK^j$.

Let $V$ be a bounded connected component of $\hK^{j_0}\setminus K$ for an integer $j_0$.
Since all the
critical points of  $f$  are contained in  $K$,   the map  $f$  maps  $V$  diffeomorphically to  $f(V)$.  Then Darboux property with  $T$  equal to the closure of  $V$  implies that
either  $f(V)$  is disjoint from  $\hK$  or  $f(V)$  is a bounded
connected component of  $\hK^{j_1} \setminus K$ for an integer $j_1$.
We say a bounded connected component of  $\hK^j \setminus K$  is \emph{small} if for every integer  $n \ge 1$  the map
$f^n$  is defined on  $V$  and if  $f^n(V)$  is contained in  $\hK$. A small component is \emph{periodic} if
there is an integer  $k \ge 1$  such that  $f^k(V) = V$.

Extend $f|_{\hK}$ to a $C^2$ multimodal map $g:I\to I$ of a closed interval containing $\hK$ in its interior, with
all critical points non-flat. Then by \cite[Ch. IV, Theorem A]{dMvS} there is no wandering interval for $g$, in particular an interval $V$ as above that is small is either  eventually periodic or is uniformly attracted to a periodic orbit $O(p)$.
Notice that in the latter case all $f^n(V), n=0,1,...$ are pairwise disjoint since otherwise a boundary point of one of them, hence a point in $K$ would belong to another $f^n(V)$ that contradicts the disjointness of all $f^n(V)$ from $K$.

The latter case however cannot happen. Indeed, $O(p)$ is in $K$ by compactness of $K$. However attracting periodic orbits cannot be in $K$, as $f|_K$ is topologically transitive and $K$ has no isolated points
compare arguments after Corollary~\ref{no periodic}
If $O(p)$ is indifferent and $V$ is not eventually periodic then for $x\in K$ being a boundary point of $V$ we have $f^n(x)\to O(p)$ and
moreover $f^n(W)\to O(p)$ for $W$ a neighbourhood of $x$.
Indeed denoting by $m$ a period of $p$ we can assume that $f^{m}$ is strictly monotone in a neighbourhood $D$ of $p$ since
$\Crit(f)$ is finite, and all the points $f^n(x)$ belong to $D$ for $n$ large enough. Then by the monotonicity the intervals between them are also attracted to $O(p)$. Hence $f^n(W)$ converge to $O(p)$. (In particular $f^n(W)$ is contained in  $B_0(O(p))$ and $O(p)$ is attracting or attracting from one side.)
This again contradicts topological transitivity, as there is no return of $W$ to a neighbourhood of $x$.


 Suppose now that $V$ is a periodic small component.  By \cite[Ch. IV,Theorem B]{dMvS} all such components have minimal periods bounded from above by a constant. Notice that in each period $m$ there can be only a finite number of them.  Otherwise
in a limit of such components we have a periodic point
$p$ of minimal period $m$ in a limit of such components $V_n$, so for $p_n\in\partial V_n$ we have $p=\lim p_n$. All $p_n$ for $n$ large enough belong to an interval where $f^m$ is defined and has no critical points.
 Consider all $V_n$ on the same side of $p$. Then the orbits of the interiors $H_n$ of the convex hulls of $p_{2n}, p_{2(n+1)}$ are pairwise disjoint and each $H_n$ intersects $K$, at $p_{2n+1}$.
 This contradicts topological transitivity of $f|_K$.

Now we end the proof
by removing from $\hK$ the family of all small periodic components. Due to the just proved finiteness of this family we obtain the decomposition of this new set, denoted by $\hI_K$, into a finite family of closed
intervals $\hI^1,...,\hI^m$ (where $m$ can be larger than the original $m(U)$). Finally we take ${\bf{U}}$, a small neighbourhood of $\hI_K$ as in Definition~\ref{extension-1}, contained in $U$ and
consider the original $f$ on it.


This completes the proof of the part 1 of the Proposition.

The part 2 says that the maximal repeller $K(f)$ in $\hI$ satisfies Darboux property. To see this notice that
if $x\in K(f)$ i.e. its forward trajectory stays in $\hI$, then obviously the forward trajectory of its every preimage $y$ in $\hI$ stays in $\hI$, hence $y\in K(f)$, which yields Darboux property for $K(f)$.

\end{proof}

\begin{coro}\label{no periodic}
For each
triple  $(f,K,U)$ as above, $f$ being $C^2$, Darboux, topologically transitive, there exists a reduced quadruple $(f,K, \hI_K, {\bf{U}})\in \sA$ with no attracting periodic orbits in $\hI_K$ and with at most finite number
of periodic orbits that are not hyperbolic repelling,  in $K$, hence in $\hI_K$.
\end{coro}

Notice that for this finiteness in $\hI_K$ (or $K$) we do not need to assume BD (bounded distortion).

\begin{proof} The existence of $(f,K, \hI_K, {\bf{U}})\in \sA$ with no attracting periodic orbits in $\hI_K$ was proved in Proposition~\ref{reduce}. To prove the finiteness of the set of periodic orbits in $K$  that are not hyperbolic repelling suppose that there exists a strictly monotone sequence of points $p_n\in K$
 such that $ p_n\to p$ and all $p_n$, all have the same minimal period $m$ and  all lie
in an open interval $T$ with one end at $p$ and $f^m$ is strictly monotone on $T$. Then for
$n$ large
enough
the set $[p_{n},p_{n+2})\cap K$ is non-empty as containing $p_{n+1}$
and  its forward trajectory by $f$  does not contain $p$. This contradicts topological transitivity of $f|_K$.

Finally use the fact  that the set of possible periods $m$ of non-hyperbolic  periodic orbits is finite, see
\cite[Ch. IV, Theorem B]{dMvS}.
Compare Proof of Proposition A2.

\end{proof}

This Proof explains a part of Remark~\ref{finite indifferent} stated in Introduction. Let us explain
the remaining part of Remark~\ref{finite indifferent} concerning one-sided repelling or attracting points (it partially repeats considerations in the proof od Proposition~\ref{reduce} and uses only the topological transitivity):

\smallskip

If $p$ is a one-sided attracting periodic point of period $m$ or a limit of periodic points $p_n$ of period $m$ 
lying on one side of $p$, say in $T=(p,p+\e)$,  then this side is disjoint from $K$. Indeed. In the one-sided attracting case, attracting from $T$, if $x\in T$ consider
$r>0$ small enough that for $B=B(x,r)$ all $f^n(B)$ are pairwise disjoint (attracted to $O(p)$). So $x\notin K$ by the topological transitivity of $f|_K$. In the case where $p=\lim p_n$ for $p_n\in T$, if there exist $x_k\in T\cap K$ such that
$x_k\to p$, then there exist $n(k_j)$ for a sequence $k_j\to\infty$ such that
$x_{k_j}\in S_j:=(p_{n(k_j)}, p_{n(k_j)+2})$ and all $S_j$ are pairwise disjoint. This contradicts the topological transitivity of $f|_K$ by arguments similar to ones used in the proof of Corollary~\ref{no periodic} .

The point $p$ cannot be isolated  in $K$, see Lemma~\ref{infinite}. Therefore $p$ is repelling on one side and simultaneously  accumulated by $K$ from there.





\

Finally we prove the following important

\begin{lemm}\label{good extension}
1. Given $(f,K,\hI_K)\in\sA^r$ for $r\ge 2$ there exists an extension of $f|_{\hI_K}$ to a $C^r$-mapping $g:I\to I$ for a closed interval $I$ containing
a neighbourhood of $\hI_K$,
such that all periodic orbits for $g$ in $I\setminus K$ are hyperbolic repelling.

2. If no periodic point in $\partial \hI_K$  is  accumulated (from outside $\hI_K$) by a sequence of  periodic  points of the same period that are not hyperbolic repelling, then there exists $g$ as above such that $g=f$ on a neighbourhood of $\hI_K$.
\end{lemm}

\begin{proof}
We already know that there are no periodic orbits in $\hI_K\setminus K$.
Extend $f$ from a neighbourhood of $\hI_K$ to $g:I\to I$ of class  $C^r$, so that the extension has only non-flat critical points and $g(\partial I)\subset \partial I$.
Using \cite[Ch. IV, Theorem B]{dMvS} we can correct $g$ outside of $\hI_K$, so that there is only finite number of  periodic orbits that are not hyperbolic repelling, compare Proof of
Proposition~\ref{reduce} above. If the condition 2. is satisfied then it is sufficient to correct $g$ only outside a neighbourhood of $\hI_K$.

By Kupka-Smale Theorem we can now smoothly perturb $g$ outside a  neighbourhood of $\hI_K$, while keeping
$g(\partial I)\subset \partial I$ and making $\partial I$ repelling, so that outside $K$ all periodic orbits are hyperbolic. 

Finally, let $p$ be a hyperbolic attracting periodic point of $g$
that is not in $K$, and denote by
$m$ its period. Its orbit
$O(p)$ is not contained in $\hat{I}_K$, so we can assume
$p$ is not in $\hat{I}_K$.


For each $j=0,1,...,m-1$ denote
by $B(j)$ the component of $B_0(O(p))$ the immediate basin of attraction of $O(p)$ for $g$,
containing $g^j(p)$.
 Notice that no $B(j)$ intersects $K$. Indeed, if $x$ is in the intersection then
 $f^n(x)\in K$ for $n=0,1,...$ .
 Hence $f^n(x)=g^n(x)$ accumulates at $p\notin \hI_K$ which contradicts $f^n(x)\in K\subset \hI_K$.
 We conclude that
  for $j$ such that $g^j(p)\in \hI_K$, the set $B(j)$  cannot contain a critical point, as we assumed that all $f$-critical points in $\hI_K$ are in $K$. Hence the map $g=f$ on $B(j)\subset \hI_K$, maps $B(j)$ diffeomorphically onto $B(j+1)$ (where we identify $m$ with 0).

Notice that for each $j=0,1....m-1$ we have $g(B(j))\subset B(j+1)$  and  $g(\partial B(j))\subset \partial B(j+1)$.  Hence
$g^m(\partial B(j))\subset \partial B(j)$. The boundary $\partial B(j)$ is repelling for $g^m|_{B(j)}$ because iterates of this map pointwise converge to $g^j(p)$.
So for all $j=0,1,...,m-1$ we can choose  closed intervals $V_j\subset B(j)$, such that
\begin{equation}\label{telescope-j}
g(V_j)\subset \interior V_{j+1}.
\end{equation}

\smallskip

Now we can smoothly modify $g$ on a neighbourhood of each $V_i$ in $B(i)$, for all $i$ such that
$B(i)\subset I\setminus \hI_K$, so that the new $g$ is a diffeomorphism (even affine non-constant) on each $V_i$ onto its image,
\begin{equation}\label{telescope}
g(V_i)\subset \interior V_{i+1},
\end{equation}  and
still for each $x\in B(0)$ there is $n$ such that $g^{nm}(x)\in V_0$.


Finally we can smoothly modify $g|_{V_0}:V_0\to V_1$, using the fact $g^{m-1}$ is already a diffeomorphism on
$V_1$ to its image, so that the new $g^m$ is a smooth bimodal map from $V_1$ into $V_1$, so that
all $g^m$-periodic points in $V_1$ are hyperbolic repelling. Hence all $g$-periodic points in $B_0(O(p))$) are hyperbolic repelling.

\smallskip


For example, let $\widehat g:=h \circ (g^{m-1}|_{V_1})^{-1}$, where
$h:V_1\to V_1$ is, after an adequate rescaling (and translation) of coordinates,
mapping $[-1,1]$ onto itself defined by $x\mapsto ax(x^2-1)$ where $a=\frac1{3^{-1/2}-3^{-3/2}}$ is such that the critical values are $-1$ and 1, next mapped to the repelling fixed point 0.
We get then $\widehat g \circ g^{m-1}= h$ on $V_1$.

We can easily extend $\widehat g$ from $V'_0:=g^{m-1}(V_1)$ to $B(0)$ first by putting the old $g$ outside $V_0$ and next smoothly extending it to $V_0\setminus V'_0$
so that for the perturbed map
$\widehat{g}$ we have $\widehat{g}(V_0)
\subset V_1$.
We can maintain the range of this extension, considered as a new $g$, to be in $V_1$ due to
$h(\partial V_1)\in\interior V_1$. (Notice that in this construction we do not care about the strict inclusion in \eqref{telescope} for $i=0$, because we do not need it, due to our choice of the bimodal map $h$.)  


\smallskip

 We do a similar modification in the immediate basin $B_0(O(p))$ for every attracting periodic orbit $O(p)$ and
 conclude with $g$ with all periodic orbits outside $K$ hyperbolic repelling.

 \smallskip




\end{proof}

Here is the proof of another fact (a standard one, put here for completeness) mentioned in Introduction, Remark~\ref{more periodic}:
\begin{prop}\label{finiteness}
For every $(f,K,\bf{U})\in \sA^{\BD}$, the set of periodic points in $\bf{U}$ that are not hyperbolic repelling
is finite,
 if we count an interval of periodic points as one point. Such a non-degenerate  interval cannot intersect $K$.
\end{prop}

\begin{proof}
By BD, for each periodic point $p\in \bf{U}$ attracting or one-sided attracting,  the immediate basin of attraction $B_0(O(p))$
contains a critical point or its boundary contains a point belonging to $\partial \bf{U}$. Otherwise for $g$ being the branch of $f^{-m}$ such that $g(p)=p$ and $m$   a period of $p$, the ratio $|(g^n)'(x)/(g^n)'(p)|$ would be arbitrarily large for $n$ large and appropriate $x$ such that the branch $g^n$ exists on $B(p,2|x-p|)$. So, by the finiteness of the sets $\Crit(f)$ and $\partial{\bf{U}}$, there is only finite number of (one-sided) attracting periodic points $p\in{\bf {U}}$.

If $O(p)$ is repelling with $|(f^{m})'(p)|=1$, then BD fails for backward branches of iterates of $g$ by
$|(g^n)'(x)/(g^n)'(p)|\to 0$. So repelling but not hyperbolic repelling cannot happen at all.

The case $O(p)$ is one-sided repelling can happen. Then $O(p)$ is attracting from the other side. Otherwise
$p$  would be accumulated by $f^m$-fixed points $p_j$ on that side hence $g^n$ would be well-defined on $B(p,|p_j-p|)$ and again $|(g^n)'(x)/(g^n)'(p)|\to 0$ for $x=2p-p_j$  symmetric to $p_j$ with respect to $p$.
So, as we proved already,  there is at most finite number of such orbits.

\smallskip

Suppose now that $p$ is an indifferent periodic point accumulated by periodic points of the same or doubled period.
If there existed attracting or one-sided
attracting periodic points, all on the same side of $p$ (with the same or doubled period) arbitrarily close to $p$, this would
contradict the fact already proven that there is only a finite number of them. So this "accumulation" case cannot happen.

\smallskip

The only remaining case is a periodic point $p$ which belongs to a non-degenerate closed interval $T$ of periodic points. 

 Suppose the interior of $T$  contains $x\in K$.
Then by topological transitivity $x$ is isolated in $K$, which is not possible by Lemma~\ref{infinite}. Compare Corollary~\ref{no periodic}.

 So suppose $x\in \partial T\cap K$. The previous argument shows that $(\interior T)\cap K =\emptyset$. Then   $x$ must be a limit of 
 periodic points of period $m$ from the side $T'$ different from $T$ since otherwise it either attracts from the side of $T'$ so it is isolated in $K$ which contradicts topological transitivity, or it repels on that side but then it attracts from the side $T$ (by bounded distortion), a contradiction. However the  accumulation by $f^m$-fixed points from $T'$ is not  possible by arguments as in Proof of Corollary~\ref{no periodic}.

We conclude that $T\cap K=\emptyset$. The number of non-degenerate intervals $T$ of periodic points must be finite, since their ends must be (one-sided) attracting or in $\partial {\bf{U}}$ and we have already proven that their number is finite.


\end{proof}

Due to this Proposition  we do not need to care about intervals of periodic points since we can just get rid of them by shrinking $\bf{U}$, compare Remark~\ref{more periodic}.

\

\begin{rema}\label{factor}
Notice that Darboux property for any triple $(f,K,U)$ as in Definition~\ref{Darboux} allows to consider a canonical factor of $f|_{\hK}$, where $\hK=\bigcup_{j=1,...,m} \hK^j$ is the union of the convex hulls $\hK^j$ of $U^j\cap K$ in the notation from the beginning of Proof of Proposition~\ref{reduce}.

  Namely, one contracts to points all the closures $Q^j, j=1,...,m-1$ of the bounded connected components of $\R\setminus \hK$
 and components of their $(f|_{\hK})^n$-preimages. The resulting map $g$ is a piecewise strictly monotone, piecewise continuous  map of interval $g:I\to I$. More precisely, for the points that arise from contracted intervals $Q^j$, the turning critical points, and the most left and most right end points of $\hI_K$ denoted according the order in $\R$ by $a_0<a_1<...<a_{m'}$, where $m'\ge m$, the mapping $g$ is strictly monotone and continuous on each
$(a_{i-1},a_i)$ for all $i=1,...,m'$. It can have discontinuities at the points that arise from the contracted intervals $Q^j$.

There is no interval whose each $g^k$-image belongs entirely to some $(a_{i_k-1},a_{i_k })$;  
such interval maps are called regular, see e.g. \cite{HU}. The lack of wandering intervals is inherited from $f$.

Of course we loose the smoothness of the original $f$ so
this factor is useful only to study topological dynamics of $f$.

All this applies in particular to $(f,K)\in\sA$, where $f|_K$ is topologically transitive.
\end{rema}






\smallskip

\subsection{On topological transitivity and related notions}

\

\

Now we prove the fact promised in Introduction, Remark~\ref{wexact1}, that allows in our theorems to assume only \emph{topological transitivity} and to use in proofs formally stronger \emph{weak exactness}.

\begin{lemm}\label{trans implies exact}

For every $(f,K)\in \sA_+$ the mapping $f|_K$ is weakly topologically exact.

\end{lemm}

\begin{proof} (mostly due to M. Misiurewicz)

\

\underline{Step 1. Density of preimages} (strong transitivity).

For $(f,K)\in \sA$ (i.e. not assuming positive entropy) we prove that for all $x\in K$
the set $A_\infty(x):=\bigcup_{n\ge 0} (f|_K)^{-n}(x)$ is dense in $K$.
This proves Proposition~\ref{dp}

Indeed, let $T$ be an open interval intersecting $K$. Denote $\hf:=f|_{\hI_K}$. Let
$$
W:=\bigcup_{n\ge 0} \hf^n(T),
$$
while iterating we each time act with $\hf$ on the previous image intersected with $\hI_K$, the domain of $\hf$. The set $W\cap K$ is dense in $K$.
Indeed, this set coincides with $\bigcup_{n\ge 0} (f|_K)^n(T)$ since if a trajectory $x_0,...,x_n=x$ is in $\hI_K$, then by the maximality of $K$ this trajectory is in $K$,
so $\hf$ is equal to $f|_K$ on it. Hence the density of $W\cap K$ follows from the topological transitivity of $f|_K$.

The components of
$f(W)\cap  \hI_K$ are contained in the components of $W \cap\hI_K$.
There is a finite number (at most the number of the boundary points of $\hI_K$) of components of $W$ touching $\partial\hI_K$. We call them "boundary components". This number is positive and the $\hf$-forward orbit of each of them touches
or intersects  $\partial\hI_K$ (i.e. a "cut" happens) in finite time, i.e. a "boundary component" returns to a "boundary component". Otherwise there would be a wandering component, or a periodic component $W'$ with $\bigcup_{n} \hf^n(W')$ not having points of  $\partial\hI$ in its closure, which is not possible by the
 topological transitivity of $f|_K$. So there is only a finite number of components of $W$ (bounded by the sum of the times of the returns). This together with the density of $W$ proves that $W$ covers $K$ except at most a finite set, hence the density of each $A(x)$ in $K$ for $x$ not belonging to a finite set $\sE:=\hI_K\setminus W$.

 In fact $A(x)$ is dense in $K$ for all $x\in K$. 
Indeed. If $\hf(y)=x\in \sE$ then $y\in \sE$, since $W$ is forward invariant. Hence by the finiteness of $\sE$ the point $x$ is periodic. The orbit $O(x)$ is repelling or one-sided repelling, since on the side it is not repelling it cannot be an accumulation point of $K$, and it cannot be isolated in $K$. If $x$ has an $f$-preimage $z\in\hI\setminus O(x)$
then
$z$, as non-periodic, belongs to $W$, hence $x\in W$, a contradiction. So there exits an open
neighbourhood $U$ of $O(x)$ such that the compact set $\hf(\hI_K\setminus U)$ is disjoint from $O(x)$. Denote $U':=\hI_K\setminus \hf(\hI_K\setminus U)$.

Thus, there exists $S\subset U'$, an open interval intersecting $K$, its forward orbit leaves  $U$ after some time $m$, with the sets $f^j(S)$ disjoint from $S$ for $j<m$, and never again returns to $U'$, hence never intersect $S$,  contrary to the topological transitivity.

\

Now we prove weak exactness, provided $(f,K)\in \sA_+$, that is $h_{\ttop}(f|_K)>0$.
The proof above does not yield this, since we do not have $N$ (for interval exchange maps such $N$ does not exist).  We did not even proved that
in $K=\bigcup_{n=0}^\infty f^n(W\cap K)$ one can replace $\infty$ by a finite number.

\

\underline{Step 2. Semi-conjugacy to maps of slope $\beta$.} First replace $f|_{\hI_K}$ by  its factor $g$ as in Remark~\ref{factor}. The map $g$ is piecewise strictly monotone, piecewise continuous, with $m'$ pieces $I_j=(a_j,a_{j+1}), j=1,...,m'$.
Denote $\bigcup_j \{a_j\}$ by $a^h$.

  If $g$ is continuous, then by Milnor-Thurston Theorem, see \cite{MT} or
  \cite[Theorem 4.6.8]{ALM}, $g$ is semi-conjugate to some $h$, piecewise continuous, with constant slope, $\log |h'|=\beta=h_{\ttop}(f|_K)>0$, via a non-decreasing continuous map. In fact this semi-conjugacy is a conjugacy via a homeomorphism since $g$ is topologically transitive, since $f|_K$ is. Milnor-Thurston Theorem holds also  for $g$ piecewise continuous, piecewise monotone, which is our case, with the proof as in \cite{ALM}, \cite{Misiurewicz}.

\

\underline{Step 3. Growth to a large size.}  As before (for $f$) we consider an open interval $T$ and act by $h$. More precisely for $T=T_0$ in one of the intervals $(a_j,a_{j+1})$ define inductively $T_{n+1}$ as a longest component of
$h(T_{n})\cap I_j$,  $j=1,...,m'$. As long as $T_n$ are short, $h(T_n)$ capture $a_j$ (i.e the "cuts" happen)
rarely. So, due to expansion by the factor $\beta>1$ the interval $T_n$ is  larger than a constant $d$ not depending on $T$, for  $n$ large enough.
Fix $n$ such that $T_n$ contains a point in $a^h$ in its boundary. It exists since otherwise $|T_n|$ would grow to $\infty$. Denote this $n$ by $n(T)$. 

\

\underline{Step 4. Cover except a finite set.} We go back to $f$ and denote the semiconjugacy from $f|_K$ to $h$ (via $g$) by $\pi$. We consider now $f$ piecewise strictly monotone, that is $\hI_K$ is cut additionally at turning critical points.
We keep the same notation $f$ and $\hI_K=\bigcup \hI^j$ except that now $j=1,...,m'$ rather than $m$, compare Remark~\ref{factor}. We write $\hI^j=[a^f_{j,0},a^f_{j,1}]$. It is useful now to
consider
$\breve{f}=f|_{\bigcup_j (a^f_{j,0},a^f_{j,1})} $,
i.e. restriction to the union of the open monotonicity intervals. Denote $\bigcup_{j,\nu}\{a^f_{j,\nu}\}$ by $a^f$. Finally let
$d'>0$ be such a constant that if $|x-y|<d'$ then $|\pi(x)-\pi(y)|<d$.

Consider  open intervals $T^{j,\nu}\subset\hI^j$ of length $d'$ adjacent to $a_{j,\nu}$.
Set, compare Step 1,
\begin{equation}\label{union}
 \breve{W}^{j,\nu}:=\bigcup_{n\ge 0} \breve{f}^n(T^{j,\nu})\setminus {a}^f,
\end{equation}
subtracting each time ${a}^f$ while iterating $\breve{f}$.

By Step 1, $\breve{W}^{j,\nu}$ covers $K$ except at most a finite set $\breve{\sE}\subset K$.
(It does not matter in Step 1 whether we deal with $\hf$ or $\breve{f}$, but now we prefer $\breve{f}$ to deal with open sets.)

\

\underline{Step 5. Doubling.} Now we shall prove that in the union in (\ref{union}) one can consider  in fact a finite set of $n$'s and if we do not subtract
 ${a}^f$ and act with $f$ it covers $K$.

Notice that $\breve{\sE}\supset a^f$, by the definitions. Double each point of $\breve{\sE}$ which is the both sides limit (accumulation) of $K$.
More precisely, consider the disjoint union $\hK$ of the compact sets of the form $S\cap K$,
covering $K$, where $S$'s are closed intervals with ends in $\breve{\sE}$  and pairwise disjoint interiors. Denote the projection from $\hK$ to $K$ gluing together the doubled points by $\hat{\pi}$.
Denote $\hat{\pi}^{-1}(\breve{\sE})$ by $\hat{\sE}$ and more generally $\hat{\pi}^{-1}(X)$ by $\hat{X}$ for any $X\subset K$.
Denote the lift of $f|_K$ to $\hK$ by $F$. Notice that maybe it is not uniquely defined at
points not doubled, whose $f$-images are doubled. We treat $F$ as 2-valued there.

$F$ maps $\hK$ onto $\hK$ by topological transitivity of $f|_K$, similarly to $f|_K$ mapping $K$  onto $K$ see Lemma~\ref{infinite}.
Suppose $x\in \hat{\sE}$ and $F(y)=x$. If $y\notin  \hat{\sE}$, then $y\in\widehat{W\cap K}$, where
$W:=\breve{W}^{j,\nu}$.
Then $x\in F(\widehat{W\cap K})$, moreover $x$ is in the interior of this set in $\hK$ since $f|_K$ is open in a neighbourhood of $\hat{\pi}(y)$. 

If $y\in \hat{\sE}$, then consider $z\in \hK$ such that $F(z)=y$. If $z\notin  \hat{\sE}$ then as above $y$ is in the interior of $F(\widehat{W\cap K})$. Then $x$ is in the interior of $F^2(\widehat{W\cap K})$, but for this we use the fact that $F$ is an open map in a neighbourhood of $y$ due to doubling at $\hat{\pi}(x)$. We continue and if we have an $F$-backward trajectory of $x$ in $\hat{\sE}$, then it is periodic and we arrive at a contradiction with topological transitivity of $f|_K$ as in Step 1.



\

Finally by the compactness of $\hK$ we can choose finite subcovers from open covers,
hence, after projecting $\hK$ back to $K$ we can write
$$
 K=\bigcup_{n=0}^{N(j,\nu)} (f|_K)^n(T^{j,\nu}\cap K).
 $$
We end the proof by setting $N=\max_{j,\nu} N(j,\nu)$.

Since $f^{n(T)}(T)$ contains an interval of the form $T^{j,\nu}$ by the definitions of $d'$ and $d$, then
$$
K=\bigcup_{i=0}^{N} (f|_K)^{n(T)+i}(T\cap K).
$$

\end{proof}

\

Therefore the following holds

\begin{prop}\label{exact implies htop}
If $(f,K)\in \sA$ then the properties: $h_{\ttop}(f|_K)>0$ (i.e. $(f,K)\in \sA_+$) is equivalent to
 weak topological transitivity of $f|_K$.
\end{prop}

This follows immediately from Lemma \ref{trans implies exact} and the following general fact easily following from the definition of topological entropy.

\begin{lemm}
Let $f:K\to K$ be a continuous map of a compact metric space $K$ that is not reduced to a single point. Then weak topological exactness implies positive topological entropy.
\end{lemm}

\

\section{Uniqueness of equilibrium via inducing}
\label{s:AppendixB}

We shall prove here uniqueness of equilibrium state, asserted in Theorem A, using our inducing construction.\footnote{A similar method was used by Pesin and Senti in \cite{PS}.}

Given a nice couple $(\hV,V)$ and the induced map $F$ as before, consider the following subsets of $K$:
$
\sK(V), \;
K(F),\;
$ defined before, and
$$
Q(F):=\{z\in V\cap K: \exists\;\hbox{infinitely many returns to}\; K,\; \hbox{but no good returns}\}.
$$
We have $\nu(\sK(V))=0$ since otherwise, by forward invariance of $\sK(V)$ and ergodicity, $\nu$ is supported on $\sK(V)$ and $h_\nu(f)-t\chi_\nu(f)=P(f,t)> P(f|_{\sK(V)}, -\log |f'|)$, the latter inequality by Part 1.1 of Proof of Lemma~\ref{key}, compare \cite[Lemma 6.2]{PR-L2}. This contradicts the opposite Variational Principle inequality $h_\nu(f)-t\chi_\nu(f)\le P(f|_{\sK(V)}, -\log |f'|)$.

We conclude for the "basin" $\sB \sK(V):=\bigcup_{j=0}^\infty f^{-j}(\sK(V))$, using the $f$-invariance of $\nu$,
that $\nu(\sB \sK(V))=0$.

We prove now that also $\nu(Q(F))=0$. Denote $\psi:=\Jac_\nu(f)$,  Jacobian in the weak sense, that is
a $\nu$ integrable function such that (\ref{Jacobian}) holds after removal from $K$ a set $Y$ of zero $\nu$ measure (i.e. \ref{Jacobian} holds for $A$ satisfying additionally $A\cap Y=\emptyset$). Such $\psi$ exists under our assumptions and moreover Rokhlin entropy formula $h_\nu(f)=\int \log \psi d\nu$ holds, see e.g. \cite[Proposition 2.9.5 and 2.9.7]{PU}.
(One sided, countable, even finite, generator exists by weak exactness and the finiteness of the number of maximal monotonicity intervals of $f$ in the definition of generalized multimodal maps.)

By Birkhoff Ergodic Theorem for every $a,b>0$ there exists a set $K_{\nu,b}\subset K$ such that $\nu(K_{\nu,b})>1-b$ and for all  $n$ large enough and all
$y\in K_{\nu,b}$
$$
\big| {1\over n} \sum_{j=0}^{n-1} \log \psi (f^j(y) -  \int\log\psi\;d\nu \big|< a.
$$

For each integer $n\ge 0$ denote $Q_n:=\{y\in K\cap V: x=f^n(y)\in V, \; \hbox{and} \; y \; \hbox{is a bad iterated preimage of}\; x \; \hbox{of order}\; n\}$. By Lemma~\ref{nodes2} the set $Q_n$ is covered by at most $\exp (\e  n)$ pull-backs of $V$ of order $n$.
Hence
$$
\nu(Q_n\cap K_{\nu,b})\le \exp (\e  n) \exp (-(h_\nu(f)- a)n).
$$
We have $Q(F)=\bigcap_{m=0}^\infty \bigcup_{n\ge m} Q_n$ hence
$\nu(Q(F)\cap K_{\nu,b})=0$ if $\e  + a <(h_\nu(f)$. Since $b$ can be taken arbitrarily close to 0 we obtain $\nu(Q(F))=0$.

 Let now $z\in V\cap K$ has infinitely many returns to $V$ but only a finite number of good return times. Let $n_0$ be the largest one. Assume it is positive. Denote $y:=f^{n_0}(z)$. If $n_1<n_2<...$ are consecutive times of return of $z$ to $V$ bigger than $n_0$, then either for their subsequence $n_{j_k}$, for each $k$ the point $y$ is a bad iterated preimage of order  $n_{j_k}-n_0$, hence it belongs to $Q(F)$,  or there is an arbitrarily large $n_j$ such that $n_j-n_0$ is a good time for $y$. Since $n_j$ is not good for $z$ we conclude that for $\hW$ being the pull-back of $\hV$ for $f^{n_j}$ one of the sets $f^i(\hW), i=0,1,...,n_0-1$ intersects $S'(f,K)$. Since diameters of $\hW$ tend to 0 as $j\to\infty$ we conclude that $z\in \bigcup_{i=0}^{n_0-1} f^{-i}(S'(f,K))$. Compare the respective
 reasoning in Subsection~\ref{From the induced map to the original map. Conformal measure}.

 Thus, for the "basins" $\sB Q(F):=\bigcup_{n=0}^\infty f^{-n} Q(F)$ and
 $\sB S'(f,K):=\bigcup_{n=0}^\infty f^{-n} S'(f,K)$
 $$
 \nu (\sB Q(F)\setminus \sB S'(f,K))=0.
 $$

 Notice finally that $\nu$ has no atoms in $S'(f,K)$. Indeed, by the invariance of $\nu$ for every $z\in K$ it holds $\nu(\{f(z)\})\ge \nu(\{z\})$. Hence if $\nu$ had an atom in $S'(f,K)$ then $\bigcup_{n=0}^\infty f^n(S'(f,K)$ would be finite and in consequence due to ergodicity it is supported on a periodic orbit. Then $h_\nu(f)=0$ what contradicts the assumption it is an equilibrium for
 $t_-<t<t_+$ (see the beginning of the proof of Uniqueness).

 We conclude that the "basin" $\sB K(F):= \bigcup_{n=0}^\infty f^{-n} K(F)$  has full measure $\nu$, hence $\nu(K(F))>0$. 





Consider the inverse limit (natural extension) $(\tK,\tf,\tnu)$.
Consider also the inverse limit  $(\tK(F),\tF)$ (just topological).
Define $\Pi$ the projection of  $\tK$ to $K$ defined for every $f$-trajectory $y=(y_j)_{j\in Z}$,  by $\Pi(y)=y_0$. Denote by $\iota$ the embedding of $\tK(F)$ in $\tK$ defined by the completing of each $F$-trajectory to $f$-trajectory.
Notice that $m=m(\Pi(y))$ for $y\in \tK(F)$ is time of the first return to $\iota\tK(F)$ for the mapping $\tf$ .
Indeed, $m$ is by definition the least good return time of $\Pi(y)$ to $V$. So suppose there is $j:0<j<m$ a time of return of $y$ to $\tK(F)$. Consider $k<j$ such that $j-k$ is a good time for $y_k$.
We have $k>0$ since otherwise $j$ is a good time for $\Pi(y)=y_0$, due to $f^j\circ f^{-k}=f^{j-k}$. But by $\tf^j(y)\in\tK(F)$ there are infinitely many $k<j$ such that $j-k$ is a good time for $y_k$, a contradiction.

Denote the normalized restriction of $\tnu$ to $\tK(F)$ by $\tnu*$
Hence $m\circ \Pi$ is $\tnu*$-integrable on $\iota\tK(F)$, more precisely $\int_{\iota\tK(F)} m\, d\tnu*=1$ by ergodicity.

Define $\nu=\iota^{-1}_*\Pi_*(\tnu*)$. By above,  the function $m$ is $\nu*$-integrable.
Notice that $\nu*$ is an equilibrium measure for $F$ and the potential $-t\log |F'|-mP(t)$.
This follows from the fact that $\tnu$ is an equilibrium for $f$ and $-t\log |f'|$ and the
calculation
(\ref{tower-calculation}) done in a different order, for $\nu,\nu*$ playing the role of $\rho',\rho$.

Now we refer to \cite[Theorem 2.2.9]{MU} saying in particular that the equilibrium for $F$ and our $\Phi$ is unique.
Thus $\nu*=\rho$, hence $\nu=\rho'$, due to the formula (\ref{rho}) for both measures.
Notice that (\ref{rho})
 gives $\nu$ because it makes $\tnu$ out of $\tnu*$ because $m$ is the time of the first return map for $\tF$.
The proof of  Uniqueness is finished.

\

\

\section{Conformal pressures}\label{section conformal}

Here  we shall define and discuss  various versions of conformal pressures announced in Remark~\ref{conf} and compare them to $P(K,t)$, thus complementing Theorem B.


\begin{defi}\label{Pconf}
Recall after Introduction  that similarly to the complex case \cite{PR-L2} and \cite{P-conical} define
 {\it conformal pressure} for $t\in \R$ by
$$
P_{\conf}(K,t):=\log\lambda(t),
$$
where, as in \eqref{lambda},
\begin{equation}\label{lambda1}
\lambda(t)=\inf\{\lambda>0: \exists \mu \; {\rm{on}} \; K \; {\hbox{which is}} \;   \lambda|f'|^t-{\rm{conformal}} \}.
\end{equation}
 For $\phi=\lambda|f'|^t$ as above, we also call $\mu$ a $(\lambda,t)$-conformal measure for $f$ on $K$.
 \end{defi}

It is immediate, see the end of Proof of Lemma~\ref{conformal-exceptional1}, that
 \begin{lemm}\label{positive open}
 For all $t$ real and $\lambda$ positive, if $\mu$ is a $(\lambda,t)$-conformal measure on an
 $f$-invariant set $K\subset U$ for $f:U\to\R$ a generalized multimodal map, with $f|_K$ being
 strongly transitive, then $\mu$ is positive on all open subsets of $K$.
 \end{lemm}

 If the equation in \ref{Jacobian} holds only up to a constant $\e $, namely
 $$|\mu(f(A))- \int_A\phi\;d\mu|\le\e
 $$
 for every Borel $A$ where $f$ is injective, we say
the measure $\mu$ is  $\e$-$(\lambda,t)$-conformal.


\medskip

In the interval case,
it is useful to weaken the definition of conformal measure and
to assume \eqref{Jacobian} only on $A\subset K$ disjoint from $\NO(f,K)$ (the set of points in $K$ where $f|_K$ is not open, see
Definition~\ref{exceptional}), and to assume only the inequalities

\begin{equation}\label{subconformal1}
\mu(f(x))\ge \lambda|f'(x)|^t \mu(x)
\end{equation}
and
\begin{equation}\label{subconformal2}
\mu(f(x))\le \sum_{z\in f^{-1}(f(x))} \lambda|f'(z)|^t \mu(z),
 \end{equation}
for all $x\in {\rm{NO}}(f,K)$, except where the expressions $|f'(y)|^t \mu(y)= \infty\cdot 0$ (for $y=x$ or $y=z$) appear for  $t<0$.
We call such $\mu$ a $(\lambda,t)$-{\it conformal* measure}.

\smallskip

We do not assume here that $\mu$ is forward quasi-invariant, unlike in Definition~\ref{defi:Jacobian}.

Notice that if $x\in\Crit(f)$ then \eqref{subconformal1} implies $\mu(x)=0$ if $t<0$; since otherwise $\mu(f(x))=\infty$.

\smallskip

(For a finite continuous potential $\phi$ one considers the summands in the latter expression with weights 2 at turning critical points $z$, as 2 is the local degree at them. Here however for $\phi=-t\log |f'|$ we have
$\phi (z)= \pm \infty$ depending whether  $t>0$ or $t<0$, so the factor 2 does not make a difference.
For $t=0$ the are no atoms, otherwise the measure would be infinite.)

\medskip

Assume $(f,K)\in\sA$. Then the set $\NO(f,K)$, hence the set where only the inequalities \eqref{subconformal1} \eqref{subconformal2} hold instead for the equalities,
is finite, see Lemma~\ref{NO}.

In particular
a $(\lambda,t)$-{\it conformal* measure} is allowed to have atoms at $f$-images of turning critical points, but not at
inflection critical points that are not end points, for $t>0$.
(Compare the notion of almost conformal measures in \cite{HU} or compare \cite{DU}).

Unfortunately we do not know whether even for $t>0$ there always exist at least one conformal measure with Jacobian  $\lambda|f'|^t$, to define $\lambda(t)$ in (\ref{lambda1}). So we do not know whether always $P_{\conf}(K,t)$ makes sense.
We can prove this existence  for $t_-<t<t_+$, see Corollary~\ref{ThC-old} and Section~\ref{s:analytic}. For $t<t_-$ the existence can be false, e.g. for $f(z)=z^2-2$, even in the complex case, so for $t<0$ it is better to work with backward conformal measures, see Definition~\ref{bpressure})

\medskip

We define the {\it conformal star-pressure} $P^*_{\conf}(K,t)$ as in (\ref{lambda1}) but allowing the  measures $\mu$ to be  $(\lambda,t)$-conformal*.
This always makes sense. Indeed, as we shall prove in Proposition~\ref{Pconf2} below,
the set of measures in (\ref{lambda1}) defining $P^*_{\conf}(K,t)$ is non-empty.

\begin{lemm}\label{conformal-exceptional1} Let $(f,K)\in\sA$. 
Assume $t\in\R $. Suppose that $\mu$ is a
$(\lambda,t)$-conformal* measure
on $K$. Then either $\mu$ is positive on all open (in $K$) subsets of $K$ or $\mu$ is supported
in a finite  $S'$-exceptional
set $E\subset K$.

If $\mu$ is $(\lambda,t)$-conformal on $K$ then it is positive on all open sets in $K$.
\end{lemm}

\begin{proof} (compare \cite[Lemma 3.5]{MS})
Suppose that there exists open $U\subset K$ with $\mu(U)=0$.

Then, by the strong transitivity of $f$ on $K$, see Lemma~\ref{trans implies exact} Step 1, implying
that there is $n>0$ such that $\bigcup_{j=0}^n f^j(U)=K$, we conclude that
$\mu$ is supported by a finite set of atoms
$\Xi\subset \bigcup_{j=1}^n \bigcup_{t=0}^{j-1} f^{j-t}(f^t(U)\cap\NO(f,K))$.


The set $\Xi$ is weakly $S'$-exceptional. Otherwise there  exists $x\in \Xi$
and its infinite backward trajectory $G\subset K$ omitting $S'(f,K)$. Hence this trajectory consists of atoms by \ref{Jacobian}, whose set therefore is infinite, a contradiction. (Another argument: $O^-_{\rm{reg}}(x)$, see Definition~\ref{exceptional}, would be dense in $K$, since otherwise it would be weakly $S'$-exceptional. Hence infinite, a contradiction.)

\smallskip

Consider first $t<0$. Then if $x$ is an atom all its forward trajectory consists of atoms by \eqref{subconformal1}. Hence it is finite. Hence $\bigcup_{j=0}^\infty O^-_{\rm{reg}}(f^j(x))$ is finite, hence as forward invariant by definition, it is $S'$-exceptional.

\smallskip

Consider now $t\ge 0$. Then, for each atom $y\in K$, there is
an atom $z\in f^{-1}(y)$, by \eqref{subconformal2}. Therefore if $x$ is an atom, there is its backward trajectory of atoms. Hence this trajectory must be finite, periodic. Again $\bigcup_{j=0}^\infty O^-_{\rm{reg}}(f^j(x))$ occurs $S'$-exceptional.






\smallskip

If $\mu$ is conformal and $\mu(U)=0$, then $\mu(K)=0$ since $\mu(f^j(U))=0$ for all $j\ge 0$.
Indeed, by the conformality of $\mu$ 
there is no way to get positive measure (in particular atoms) in $f^j(U)$.
\end{proof}

\begin{prop}\label{Pconf2}  For every $(f,K)\in \sA_+^{\BD}$, for all $t\in\R$ and
$\lambda=\exp P(K,t)$ there exists
a $(\lambda,t)$-conformal* measure
and the inequalities
$P^*_{\conf}(K, t) \le P(K,t) \le P_{\conf}(K, t)$ hold, provided the latter pressure makes sense.
If $(f,K)$ is not
$S'$-exceptional
then
the equality
$P(K,t)=P^*_{\conf}(K,t)$ holds.
\end{prop}

\begin{proof}
The strategy of the proof is the same as, say, in \cite[Proof of Theorem 12.5.11, Parts 2 and 3]{PU} in the complex case.

 First we prove the existence of a $(\lambda,t)$-conformal* measure for

 \noindent $\lambda=\exp P_{\tree}(K,z_0,t)$ for an arbitrary safe and expanding point $z_0\in K$ hence  $P^*_{\conf}(K,t)\le P(K,t)$. Consider a sequence of
measures
$$
\mu_n=\sum_{k=0}^\infty \sum_{x\in (f|_K)^{-k}(z_0)} D_x\cdot \phi_k\cdot \lambda_n^{-k}|(f^k)'(x)|^{-t}/ \Sigma_n,
$$
where $\lambda_n\searrow\lambda$ for $\lambda=\exp P_{\tree}(K,z_0,t)$, where each $D_x$ denotes Dirac measure at $x$.
Each $\Sigma_n$ is the normalizing denominator such  that $\mu_n$ is a  probability measures. The numbers $\phi_k$ are chosen so that $\phi_{k+1}/\phi_k\to 1$ as $k\to\infty$ and $\Sigma_n\to\infty$ as $n\to\infty$.
Compare \cite[Lemma 12.5.5]{PU}

All $\mu_n$ are $\e $-$(\lambda_n,t)$-conformal for an arbitrary $\e >0$ and all $n$ respectively large.
Define $\mu$ as a weak$^*$ limit. This procedure to get $\mu$ is called Patterson-Sullivan method.

Then $\mu$ has Jacobian
$\lambda|f'|^t$ for $\lambda=\exp P_{\tree}(K, z_0, t)$ in the above sense for all $A\subset K$ not containing points belonging to $\NO (f,K)$, in particular turning critical points. For each critical point $c\in K$ we obtain
for all $n$ and $r$ small enough
$\mu_n(f(B(c,r))\le 2\sup\lambda_k (2r)^t +\e $,
where $\epsilon>0$ is arbitrarily close to 0, compare \cite[Remark 12.5.6]{PU}.
Therefore for $t> 0$ we have a uniform bound for $\mu_n(f(B(c,r))$, arbitrarily small if
$r$ and $\e$ are small. If $c$ is an inflection point for $f$ then $f$ is open at $c$, i.e.
$f(B(c,r))$ is a neighbourhood of $f(c)$, therefore $\mu(f(c))=0$ yielding the conformality of $\mu$ also at $c$. If $c$ is a turning point then $f(B(c,r))$ need not be a neighbourhood of $f(c)$ in $K$ and the mass
$\mu(f(c))$ can be positive, coming from the other side of $f(c)$ than the `fold' $f(B(c,r))$.
Therefore $\mu$ satisfies merely \ref{subconformal1} and \ref{subconformal2} at $c$. The same concerns end points belonging to $\NO(f,K)$. Hence $\mu$ is an $\exp P(K, t)|f'|^t$-conformal* measure.
Taking infimum over such measures we end with $P^*_{\conf}(K,t)\le P(K,t)$.

\

To prove $P(K,t)\le P^*_{\conf}(K,t)$ in non-exceptional case, or merely $P(K,t)\le P_{\conf}(K,t)$ in the general case,  we consider the version
$P_{\hyp}(K,t)$ of $P(K,t)$. For every
$(\lambda,t)$-conformal* measure $\mu$ on $K$
we consider a small $U\subset K$ open in $K$ intersecting a hyperbolic isolated $X\subset K$, where
$P_{\hyp}(K,t)$ is almost attained by $P(f|_X,-t\log |f'|)$,
see Definition~\ref{hyperbolic pressure}.
Then, for $\mu(U)>0$, see Lemma~\ref{conformal-exceptional1} the rest of the proof is identical as in \cite[Proof of Theorem 12.5.11, Parts 2]{PU}, giving $P(f|_X,-t\log |f'|)\le  \log\lambda$.
So if $(f,K)$ is not $S'$-exceptional the proof of $P(K,t)\le P^*_{\conf}(K,t)$ is finished.

In the exceptional case the proof of $P(K,t)\le P_{\conf}(K,t)$ is the same, using $\mu(U)>0$.

\end{proof}

For $ t<0$ and conformal pressures we follow \cite[Appendix A.2]{PR-LS2}.
Notice again (compare Proof of Lemma \ref{Phyp} the case $t\le 0$) that the pressure
$P_{\tree}(K,t)$ is the classical pressure since $|f'|^{-t}$ is continuous, i.e.
$$
P(K,t)=P(f_|K,-t\log |f'|).
$$
To see this
consider $P_{\var}(K,t)$, \cite[Theorem 4.4.1]{Keller}
for the variational principle for potential functions with range in $\R\cup -\infty$.
See also \cite[Theorem A.7]{PR-LS2}.

\smallskip

For $t<0$ it is natural to  consider a different definition of conformal pressure, called in \cite{PR-LS2} backward conformal pressure.
In the complex setting each conformal measure is an eigenmeasure for the operator dual to transfer operator (Ruelle operator), with weight function $|(f')|^{-t}$. In the complex setting the transfer operator is a bounded operator on the space of continuous functions since $f$ is open.
Unfortunately for the interval multimodal maps this is not so; the  transfer operator on continuous functions can have the range not in continuous functions, since $f$ need not be an open map.
The discontinuity of the image cannot be caused by critical values, since at critical points the weight function is 0 (remember that $t<0$). It can be caused by noncritical end points not mapped to end points.


\begin{defi}\label{bpressure}
For $t<0$ define the $(\lambda,t)$-backward conformal pressure by
$$
P_{\Bconf}(K,t):=\log\lambda(t),
$$
where
\begin{equation}\label{blambda}
\lambda(t):=\inf\{\lambda>0: \exists \mu \ \hbox{backward conformal on}\ K\
\hbox{with Jacobian}\ \la^{-1}|f'|^{-t}\}
\end{equation}
and $\mu$ backward conformal means here that
for every Borel set $A\subset K$ on which
$f$ is injective
\begin{equation}\label{Jacobian2}
\mu(A)=\int_{f(A)}\la^{-1}|f'(f|_A)^{-1}(x))|^{-t}d\mu(x).
\end{equation}
\end{defi}

In this definition compared to $(\lambda,t)$-conformal measure, we just resign from assuming the
$\mu$ is forward quasi-invariant, see Definition~\ref{defi:Jacobian}, i.e. we allow the image of a set of measure 0, to have positive measure. But we still cannot prove the existence and repeating Patterson-Sullivan construction we land with conformal*-measures, which is  a weaker property than backward conformal.


For $t<0$ Proposition~\ref{Pconf2} can be completed as follows

\begin{prop}\label{BPconf}  For every $(f,K)\in \sA_+^{\BD}$,  for all $t < 0$ and
$(\lambda,t)$-conformal* measure, $\lambda \le \exp P(K,t)$. If $f$ is not $S'$-exceptional then
$\lambda \ge \exp P(K,t)$.


\end{prop}

\begin{proof}



Let us repeats the calculation in \cite[A.3]{PR-LS2}.
$$
\int \sum_{x\in (f^n|_K)^{-1}(z)}\lambda^{-n}|(f^n)'(x)|^{-t}\,d\mu(z)
\ge \mu(K)=1.
$$
This inequality follows from \eqref{Jacobian} outside $S'(f,K)$, and from \eqref{subconformal1} at $x\in \NO(f,K)$, which for $x\in\Crit(f)$ implies $\mu\{(x)\}=0$ due to $t<0$, as mentioned already right after   formulation of \eqref{subconformal1}

For $\delta>0$ arbitrarily small, for every $z$ under the integral and $n$ large enough,
$$
\exp n(P(f|_K,-t\log|f'|)+\delta)\ge \sup_{z\in K}(\sum_{x\in (f^n|_K)^{-1}(z)}|(f^n)'(x)|^{-t},
$$
see \cite[Lemma 4]{P-Perron},
compare Lemma~\ref{Phyp}, hence we get

\noindent
 $\log\lambda \le P(f|_K,-t\log|f'|)$.
 The opposite inequality has been already proved in Proposition~\ref{Pconf2}.
\end{proof}

The next proposition says that for  $t_-<t<t_+$ there are conformal* measures positive on open sets, even for $S'$-exceptional $f$, constructed by Patterson-Sullivan method as above.
 In this paper we already proved for $t_-<t<t_+$
 the existence of a true conformal measure with nice properties, see Theorem A,
 but in the proof we used more complicated `inducing'  techniques.

\begin{prop}\label{open}
For $(f,K)\in \sA^{\BD}_+$ not $S'$-exceptional for all real $t$, or,
allowing the $S'$-exceptional case, for $t: t_-<t<t_+$,
there exists on $K$ a $(\lambda,t)$-conformal* measure, positive on open sets, zero on all $S'$-exceptional sets, with $\log\lambda=P(K,t)$.
\end{prop}

\begin{proof}
The non-exceptional case has been already dealt with.

\smallskip

In the exceptional case consider $\mu$ constructed by Patterson-Sullivan method as in Proof of Proposition~\ref{Pconf2}. Hence, the resulting measure $\mu$ is
 $(\lambda,t)$-conformal*  for $\lambda$ satisfying $\log\lambda=P(K,t)$. If $\mu(U)>0$ fails, then $\mu$ is supported in an $S'$-exceptional set $E$ by Lemma~\ref{conformal-exceptional1}.
It follows from the proof of Lemma~\ref{conformal-exceptional1}, that there is a periodic orbit
$O(p)\subset E$ with $\mu(O(p))>0$.
Hence
$\lambda^m|(f^m)'(p)|^t\ge 1$.  Therefore $\log\lambda \ge -t\chi(p)$, where
$\chi(p):={1\over m} \log|(f^m)'(p)|$.

In fact the equality $\log\lambda = -t\chi(p)$ holds, since by Lemma~\ref{NO} $O(p)$ is disjoint from $S'(f,K)$. So $\mu$ is conformal.


Thus
$$
P(K,t)=-t\chi(p)
$$
hence $t\le t_-$ or $t\ge t_+$.
(Compare the proof of \cite[Lemma 3.5]{MS}.)

\end{proof}

\begin{rema}\label{conformal-varia}
 For $t<0$, for exceptional maps, to have the equality of pressures it is sometimes
appropriate to consider supremum instead of infimum in the definition
\ref{blambda}. See \cite[Remark A.8]{PR-LS2}.

\end{rema}



\end{document}